\documentclass[leqno,a4paper]{amsart}
\usepackage[utf8]{inputenc}
\usepackage{amssymb,latexsym}
\usepackage{mathrsfs}
\usepackage{euscript,latexsym}
\usepackage[margin=1in]{geometry}
\usepackage{amsmath,amsthm,amssymb,amsrefs,bbm,color,esvect,float,graphicx,mathrsfs}
\usepackage{euscript,latexsym}
\usepackage{accents}
\usepackage{epstopdf}
\usepackage{hyperref}
\usepackage[fontsize=13pt]{scrextend}

 \newtheorem{thm}{Theorem}[section]
 \newtheorem{lemma}[thm]{Lemma}  
 
 \newtheorem{definition}[thm]{Definition}
 \numberwithin{equation}{section}
 \newtheorem{proposition}[thm]{Proposition}
 \newtheorem{remark}[thm]{Remark}

\newcommand\smland{\raise.2ex\hbox{$\scriptstyle\land$}\hspace{.01cm}}
\newcommand\smlor{\raise.1ex\hbox{$\scriptstyle\lor$}}

\def\diam{{\rm diam}}
\def\dist{{\rm dist}}
\def\supp{{\rm supp \,}}
\def \rdist {{\rm \, rdist}}
\def \inrdist {{\rm \, inrdist}}

\def \ec {{\rm ec}}
\def \child{{\rm ch}}

\def \ch{{\rm ch}}

\def\Im{{\rm Im}\,}
\def\Re{{\rm Re}\,}
\def\card{{\rm card}\,}
\def \sign{{\rm sign}}

\makeatletter
\@namedef{subjclassname@2020}{%
  \textup{2020} Mathematics Subject Classification}
\makeatother

\begin{document}
\title[Sparse domination on non-homogeneous spaces]{Sparse domination  
\\ of singular bilinear forms \\ on non-homogeneous spaces 
}
%Compact sparse local T1}
%\author{Cody B. Stockdale}
%\address{Cody B. Stockdale, Department of Mathematics and Statistics, Washington University in St. Louis, One Brookings Drive, St. Louis, MO, 63130, USA}
%\email{codystockdale@wustl.edu}
\author{Paco Villarroya}
\address{Department of Applied Mathematics, Santa Clara University, Santa Clara, CA 95053}
\email{paco.villarroya@scu.edu}
%\author{Brett D. Wick}
%\address{Brett D. Wick, Department of Mathematics and Statistics, Washington University in St. Louis, One Brookings Drive, St. Louis, MO, 63130, USA}
%\email{bwick@wustl.edu}


\begin{abstract}
We introduce a new sparse $T1$ theorem that estimates the dual pair associated with a Calder\'on-Zygmund operator by a sub-bilinear form supported on a sparse family of cubes.  

The main result in the paper improves previous sparse $T1$ theorems in several ways:
it applies to non-homogeneous measures of power growth, it only requires a numerable family of testing conditions, 
and it can be used to prove boundedness of Calder\'on-Zygmund operators on weighted spaces for a class of weights larger than the Muckenhoupt $A_p$ weights. 
%In a  variation of the main theorem, we prove an alternative type of sparse $T1$ theorem that uses testing conditions on cubes of different dimensions. 
\end{abstract}

\maketitle

\section{Introduction}

The term sparse domination refers to a collection of methods to estimate a singular, signed, non-local operator by  another operator that is non-singular, positive and localized on a sparse family of cubes. 
% typically of the form
% $$
%     Sf:=\sum_{Q\in\mathcal{S}}\langle f\rangle_Q\mathbbm{1}_{Q}
% $$ 
% where $f$ is locally integrable $f$, 
% and $\mathcal{S}$ is a sparse collection of cubes.

In 2013 Lerner published the first work on sparse domination (\cite{L2013}, and  \cite{L2013-2}) as a way to obtain a simplified proof of the $A_2$ Theorem \cite{H2012}, which is a result about sharp estimates for Calder\'on-Zygmund operators on weighted spaces. 
In particular, he proved that the $L^p$-norm of a Calder\'on-Zygmund operator can be dominated by the $L^p$ norm of a sparse operator of the form 
\begin{equation}\label{normpoint}
Sf=\sum_{Q\in \mathcal S}\langle |f|\rangle_Q \mathbbm{1}_Q,  
\end{equation}
where $\langle f\rangle_Q=|Q|^{-1}\int_Q f(x) dx$, $|Q|$ is the volume of a cube $Q$, and $\mathcal S$ is a sparse collection of cubes. 
In the next few years, the initial sparse estimates in norm were replaced with pointwise estimates by sparse operators as in \eqref{normpoint} (see \cite{CAR2016},  \cite{L2016}, \cite{L2017}, \cite{LN2019}). 
By 2017 Lacey and Mena introduced in \cite{LM2017} the notion of 
sparse domination estimates in terms of bilinear forms. This consists on estimating the modulus of the dual pair $\langle Tf,g\rangle$ associated with a singular integral operator $T$ by a sparse bi-sublinear form such as 
\begin{equation*}%\label{bilinearform}
\sum_{Q\in \mathcal S}
\langle |f|\rangle_Q
\langle |g|\rangle_Q |Q|. 
\end{equation*}
%a concept that was shortly afterwards extended in \cite{CPO2018}. 
This new approach allowed the possibility to prove sparse estimates for a long list of operators including: Bochner–Riesz operators (16, \cite{LMR2019}), spherical maximal functions \cite{L2019},
modulation invariant singular integrals such as the bilinear Hilbert transform \cite{CPO2018} and Carleson-type operators \cite{B2018}, singular Radon transforms (\cite{CO2018}, \cite{H2020}), and pseudo-differential operators \cite{BC2020}.

The goal of the current paper is to further develop the sparse $T1$ theory in three directions. 
First, 
%that theorem beyond the constraint of doubling measures to the setting of non-
the sparse $T1$ Theorem in Lacey and Mena's seminal paper \cite{LM2017}, and in much of the work that followed, 
%the seminal work on sparse domination via bilinear forms, 
applies to operators and spaces defined via the Lebesgue measure.
Our aim is to go beyond the constraint of doubling measures and obtain a result that can be used in the setting of non-homogeneous spaces. 
%This way much of the work that followed could also be extended to  non-doubling measures.
%Furthermore, we improve the result in two other directions. 

Second, most papers in the literature on singular operators with non-doubling measures make use of random grids. This is the case even for results prior to the development of sparse estimates (see \cite{NTV1997}). Randomization is a very useful technique that largely simplifies the analysis of such singular operators. However, the method requires a-priori knowledge on how the operator under study 
%has to be bounded when 
behaves over a non-numerable class of testing functions, typically the class of characteristic functions of all tensor-product cubes of $\mathbb R^n$.  
Since our proof approach does not rely on random grids, 
%a technique frequently used to study singular operator with non-doubling measures (see \cite{NTV1997}). The new approach allows 
we are able to strengthen the final result by providing weaker testing hypotheses, namely, the use of a numerable (sparse) family of testing functions. This is of importance when dealing for example with operators associated with fractal measures. 

Finally, thanks to a careful control of the constants appearing in our estimates, we provide a new sparse domination result that can be used to prove boundedness on weighted spaces for weights more general than the classical Muckenhoupt $A_p$ weights. 
This allows, for example, to recover known results on boundedness of the double and single layer potential operators associated with boundary value problems for degenerate elliptic equations in divergence form, ${\rm div} (A\nabla u)-V\cdot u = 0$, with appropriate non-negative potentials $V$ (see \cite{BEO} for the case of Schr\"odinger operators). 
%This direction of research
%, which is merely highlighted in an appendix, is based on ideas presented in \cite{V2022} and \cite{SVW2021}. 
%will not be explored in the current paper. 

%the sparse domination for boundedness can be applied to prove boundedness of the double and single layer potential operators associated with boundary value problems for degenerate elliptic equations in divergence form, ${\rm div} (A\nabla u)-V\cdot u = 0$, with appropriate non-negative potentials. In \cite{BEO} it is shown that the Riesz potentials associated with these equations have kernels that decay for large cubes or cubes that are away from the origin, but not for small cubes. This would correspond to the case in which the functions $L$ and $D$ tend to zero, but not the function $S$.
%This implies that the operator cannot be compact. But it is bounded on weighted spaces for weights that do not need to satisfy Muckenhoupt $A_p$ condition for large and distant cubes, only for small cubes. 

The paper is structured as follows. 
In section 2, we introduce some notation and the statement of the main result Theorem \ref{mainresult}.
The next four sections contain some preliminary work for its proof: properties of measures of power growth (section 3); 
two Haar wavelet systems (section 4); a stopping time argument to select the appropriate sparse family of cubes (section 5); 
estimates of a Calder\'on-Zygmund operator acting on Haar wavelets, and appropriate square functions and paraproducts (section 6). 

%In section 3 we prove technical lemmas about measures of power growth. In section 4 we introduce two Haar wavelet systems.
%In section 5 we describe a stopping time argument to select the sparse family of cubes that defines the required sparse operator. In section 6 we state known estimates of a Calder\'on-Zygmund operator when acting on Haar wavelets, and define appropriate square functions and paraproducts. 

Sections 7 to 11 are devoted to the proof of the main result. In section 7, we estimate the dual pair $\langle Tf, g\rangle$ by paraproducts and bilinear forms acting on square functions. In section 8,  we prove a restricted estimate of the paraproducts by sparse forms, and also estimate the contribution of close cubes with large eccentricity by a trivial sparse form. This is the part of the operator that, in the classical literature, is handled with random grids. 
Meanwhile, in section 9 we use an iterative process to extend the estimates of section 8 to its full generality. 
In section 10, we prove that the square functions define operators bounded on $L^2(\mu)$ and from $L^1(\mu)$ into weak $L^1(\mu)$. 
We end the paper with section 11, where we use previous boundedness results to define a recursive process that estimates the bilinear forms on square functions by a sparse sub-bilinear form. 
%We end the paper with an appendix, section 11, where we show that the new type of sparse operators can be used to prove some weighted estimates. 

\section{Notation, definitions, and statement of main result}

Let ${\mathcal C}$ be the family of cubes in $\mathbb R^n$ that are tensor products of intervals of the same length, that is,  cubes with sides parallel to the coordinate axes. We note that the intervals can be open, closed, or half-open, defining in each case different cubes.  
% namely, 
%$I=\prod_{i=1}^{n}\{a_{i},a_{i}+l\}$ with $a_i,l\in \mathbb R$, and where $\{$ denotes indistinctively either $[$ or $($ and similar for $\}$. 
For each cube $I\in \mathcal C$, we denote its center by $c(I)$, its side length by $\ell(I)$, and its boundary 
%in the euclidean topology of $\mathbb R^n$ 
by
$\partial I$.

Let $\mathcal D_0$ be the family of dyadic cubes  $I=\prod_{i=1}^{n}2^{-k}[j_{i},j_{i}+1)$, 
%with $j_i, k\in \mathbb Z$. 
and $\tilde{\mathcal D}_0$ be the family of open dyadic cubes $I=\prod_{i=1}^{n}2^{-k}(j_{i},j_{i}+1)$ in both cases with $j_i, k\in \mathbb Z$. 

Let $\lambda_1,\lambda_2,\ldots, \lambda_n \in \mathbb R$ such that $\lambda_i\in \mathbb R\setminus (\cup_{j=1}^{i-1}(\lambda_j+\mathbb Q))$, and 
let $a_i=\lambda_i (1,\ldots, 1)\in \mathbb R^n$. 
For $i\in\{ 1,\ldots, n\}$, 
we define the grids of cubes
\begin{equation}\label{grids}
%\mathcal{T}_i
\mathcal{D}_i=a_i+\mathcal{D}_1=\{ a_i+I :  
I\in \mathcal{D}_0\},
\end{equation}
where $a_i+I\in \mathcal{C}$ is uniquely defined by translation. 
%such that 
%$c(a_i+I)=a_i+c(I)$ and $\ell(a_i+I)=\ell(I)$. 
We define $%\mathcal{T}_i
\tilde{\mathcal{D}}_i$ in a similar way but using $\tilde{\mathcal{D}}_0$ instead.

In large parts of the paper, we write an arbitrary choice 
of grid $\mathcal{D}_i$ simply as $\mathcal D$, and a 
choice  
of grid $\tilde{\mathcal{D}}_i$ simply by $\tilde{\mathcal D}$. We hope this license will not cause any trouble. This is the case, for example, in the following definitions. 
For $I\in \mathcal D$ (resp. $I\in \tilde{\mathcal D}$), we define the following collections of cubes: 
\begin{itemize}
\item The children of $I$, denoted by $\child(I)$, are  
the cubes $I'\in \mathcal D$ such that $I'\subset I$ and $\ell(I')=\ell(I)/2$. If $I',I''\in \child(I)$, we say that $I'$ and $I''$ are siblings. 
\item The parent of $I$, denoted 
by $\widehat{I}$, is the 
smallest cube in $\mathcal D$ such that $I\subsetneq \widehat{I}$. Note that 
$I\in \child(\hat{I})$.
\item The $k$-th generation descendants of $I$, denoted by $\child_k(I)$, are  
the cubes $I'\in \mathcal D$ such that $I'\subset I$ and $\ell(I')=\ell(I)/2^k$. We note that $\child(I)=\child_1(I)$.  
\item The friends of $I\in \mathcal D$, denoted by
$I^{\rm fr}$, are the $3^n$ cubes $J\in {\mathcal D}$ such that $\ell(J)=\ell(I)$ and $\dist(I,J)=0$, 
where $\dist(I,J)$ denotes the set distance between $I$ and $J$.
%\item The splits of $I\in \mathcal D$, denoted by
%$I^{\rm sp}$, is the collection of all $3^n$ cubes in $J\in \mathcal C$ defined by translation 
%$J=(v_i-c(I))+I$
%such that %$\ell(J)=\ell(I)$ and 
%$c(J)=v_i$ is either a vertex of $I$ or the midpoint of a $k$-dimensional face of $I$ for all $k\in \{1, \ldots, n\}$. 
\item For each $I\in \mathcal D$ and $r>0$, the cube $rI\in \mathcal C$ is defined by dilation so that $c(rI)=c(I)$, $\ell(rI)=r\ell(I)$. 
%This implies that  the intervals in the tensor product defining $rI$ are half-open with open lower endpoint. 
\end{itemize}

Let $\mu $ be a positive Radon measure on $\mathbb R^n$.
Given a grid $\mathcal D$, %a family of cubes $\mathcal S\subset \mathcal C$, 
and  a $\mu$-measurable set $\Omega \subset \mathbb R^{n}$, 
%and a collection of cubes $\mathcal S$, we denote by 
%$\mathcal S(\Omega )$ the family of cubes $S\in \mathcal S$ such that $S\subsetneq \Omega$.
%In particular, 
we denote by 
$\mathcal D(\Omega )$ the collection of cubes 
$I\in \mathcal D$ such that $\widehat{I}\subsetneq \Omega$. 
%We note that if $\Omega \in \mathcal D$ and 
%$I\in \mathcal D(\Omega)$ then $\widehat{I}\subseteq \Omega$.
Given $Q\in \mathcal C$, 
we also write $\mathcal{D}(Q)_{\geq N}=\{ I\in \mathcal{D}(Q)/ \ell(I)\geq 2^{-N}\ell(Q)\}$.

Given $I,J\in \mathcal{C}$,
if $\ell(J)\leq \ell(I)$ we write $I\smland J=J$, $I\smlor J=I$,
%$\ma=1$ and $\mi=2$.
%$r_{I\smland J}=r_{2}$, $r_{I\smlor J}=r_{1}$
while if $\ell(I)<\ell(J)$, we write
$I\smland J=I$, $I\smlor J=J$. 
We then define the eccentricity and the 
relative distance of $I$ and $J$ as
$$
\ec(I,J)=\frac{\ell(I\smland J)}{\ell(I\smlor J)},
\hskip20pt 
\rdist(I,J)=
\frac{\dist(I,J)}{\ell(I\smlor J)}.
$$
%We also define a vector-valued version of the last definition: $\vrdist{(I,J)}=\frac{1}{\ell(I\smlor J)}(c(I)-c(J))$.

We define
the inner boundary of $I$ as 
${\displaystyle \mathfrak{B}_{I}=
\cup_{I'\in \child(I)}\partial I'}$, and 
the inner relative distance 
of $I$ and $J$ by
$$
\inrdist(I,J)=\frac{\dist(I\smland J,{\mathfrak B}_{I\smlor J})}{\ell(I\smland J)}.
$$
With this we also define the following quantity: for $0<\theta <1$, let  
$$
\lambda(I,J)=\ec(I,J)^\theta(1+\inrdist(I,J)),
$$
 which is used to control nearby close cubes with very different side lengths.

Let $\langle I,J\rangle$ be the only closed cube in $\mathcal C$ containing $I\cup J$ with the smallest possible side length and 
such that  $\sum_{i=1}^{n}c(I)_{i}$ is minimum. Finally, we denote the cardinality of a set $A$ by $\card(A)$.

% For any positive integer $N$, let $\mathcal{D}_N$ denote the \emph{lagom cubes} 
% $$
%     \mathcal{D}_N:= \{Q\in\mathcal{D}: 2^{-N}\leq l(Q)\leq 2^N \,\,\text{and} \,\, \text{rdist}(Q,\mathbb{B}_{2^N})\leq N\},
% $$ 
% where 
% %$\text{rdist(P,Q)}:=\frac{\text{diam}(P,Q)}{\max\{|P|,|Q|\}}$ and 
% $\mathbb{B}_{2^N}$ is the ball centered at the origin with radius $2^N$. We write $\mathcal D_{N}^{c}:=\mathcal D\backslash \mathcal D_{N}$. 

\begin{definition}
Let $\mu $ be a positive Radon measure on $\mathbb R^n$.
For any locally integrable function $f$ and $I\in \mathcal C$, we define
$\langle f\rangle_I=\mu(I)^{-1}\int_{I}f(x)d\mu(x)$ if 
$\mu(I)\neq 0$, while $\langle f\rangle_I=0$ otherwise.

%We also define the two following average operators: for $I, J, Q\in \mathcal C$,
%$$
%E_{I}f=\langle f\rangle _{I}\mathbbm{1}_{I}, 
%\hskip100pt 
%E_{I}^{J,Q}f=\langle f\rangle _{I}\mathbbm{1}(c(J))\mathbbm{1}_{Q}. 
%$$
\end{definition}

\begin{definition}
Given a positive Radon measure $\mu$, 
a collection of cubes $\mathcal{S}\subset \mathcal C$ is sparse if there is a constant $0<C<1$ such that for every $S \in \mathcal{S}$ 
$$
    \sum_{\substack{S'\in \mathcal{S}\\ S'\subset S}} \mu(S')\leq C\mu(S).
$$ 
%where $\mathcal{S}(S)$ was defined before as the collection of cubes $S'\in \mathcal S$ such that $\widehat{S'}\subsetneq S$.

%strictly contained in $S$.
%denotes set of elements of $\mathcal{S}$ that are strictly contained in $S$.
% if there is a constant $0<c\leq 1$ such that
% for every $S \in \mathcal{S}$ 
% there exists a subset $E_S\subset I$ such that 
% $\mu(E_S)\geq C\mu(I)$ and 
% $\|\sum_{S\in \mathcal{S}}\mu(S)\|_{L^{\infty }(\mu)}\lesssim c^{-1}$. 

%Let the sparse projection $P_{S}f=\sum_{Q\in \mathcal S(S)}\langle |f|\rangle_Q \mathbbm{1}_Q $. We define
%$\| P_Sf\|_{l^2(\mu)}=$
 
Let $\mathcal S$ be a sparse collection of cubes $\mathcal S$, and $a= (a_{S})_{S\in \mathcal S}$ be 
a non-negative bounded sequence. 
%a correspondence $S\hookrightarrow S'$ where 
%We assume that for each  $S\in \mathcal S$ there are 
%$S_1, S_2\in \mathcal S$ uniquely determined by $S$ with 
%$S\subseteq S_1\cap S_2$. Then
We define the following sparse sub-bilinear form:
%\begin{equation}\label{largerform}
%\Lambda_{\mathcal S ,\varepsilon }(f,g)=\sum_{S\in \mathcal S}\varepsilon_S
%%\langle f\rangle_S\mu(S)^{\frac{1}{2}}
%\Big(\sum_{S'\in \mathcal S(S)}
%\langle f\rangle_S^2\mu(S')\Big)^{\frac{1}{2}}
%\Big(\sum_{S'\in \mathcal S(S)}
%\langle g\rangle_S^2\mu(S')\Big)^{\frac{1}{2}}, 
%\end{equation}
\begin{equation}\label{largerform}
\Lambda_{\mathcal S}(f,g)=\sum_{S\in \mathcal S}a_S
\langle |f|\rangle_{S}\langle |g|\rangle_{S}\mu(S). 
\end{equation}
%Although the sparse form depends on the sequence $a$ and the choice of cubes $S_1, S_2$, we do not reflect this dependence on the notation.
\end{definition}

%We note that this is not the most common sparse operator in the literature. It is a weaker variation which, as we see in the  appendix, is strong enough to imply some boundedness and compactness estimates on weighted spaces. 

%\begin{remark}
%The bilinear for sparse domination that generalizes \eqref{bilinearform} should be 
%$$
%\Lambda_{\mathcal S ,\varepsilon }(f,g)=\sum_{S\in \mathcal S}\varepsilon_S
%\langle f\rangle_S
%\langle g\rangle_S\mu(S).
%$$
%The form in \eqref{largerform} is larger than the previous one, but as we show in the Appendix \ref{}, it is sufficient to obtain weighted estimates. 
%\end{remark}

\begin{definition}\label{powergrowth}
Let $\mu $ be a positive Radon measure on $\mathbb R^n$, and $0<\alpha \leq n$. 
We say that $\mu$ has power growth $\alpha $ if 
$\mu(I)\lesssim \ell(I)^\alpha$ for 
all $I\in \mathcal{C}$, or at least for an appropriate sub-family of cubes in $\mathcal C$. 
\end{definition}

\begin{definition}\label{suitable}
%We define three densities of a measure with power growth: 
%for $I\in \mathcal{C}$, let 
%$$
%\rho(I)=\frac{\mu(I)}{\ell(I)^\alpha},
%$$
%$$
%\displaystyle{\rho_{\rm in}(I)
%=\sup_{\substack{t\in I\\ 0<r<\ell(I)}}
% \frac{\mu(I\cap B(t,r))}{r^{\alpha}}
% =\sup_{\substack{t\in I\\ r>0}}
% \frac{\mu(I\cap B(t,r))}{r^{\alpha}}}
% $$
%where $B(t,r)=\{ x\in \mathbb R^n : |t-x|<r \} $, 
%and given $0<\delta \leq 1$,  
%$$
%\rho_{\rm out}(I)=%\sup_{J\subset I}
%\sum_{m\geq 1}\frac{\mu(mI)}{\ell(mI)^{\alpha }}
%\frac{1}{m^{\frac{\delta}{2}+1}}.
%$$
%We denote 
%\begin{equation}\label{suitable2}
%\rho_{\mu}(I)=\rho_{\rm in}(I)+\rho_{\rm out}(I). 
%\end{equation}
%\end{definition}
We define the density of a measure with power growth as
\begin{equation}\label{suitable2}
\rho(I)=\sup_{\substack{t\in I\\ 0<r<\ell(I)}}\hspace{-.3cm}
 \frac{\mu(I\cap B(t,r))}{r^{\alpha}}\, +\, \sum_{m\geq 1}\frac{\mu(mI)}{\ell(mI)^{\alpha }}
\frac{1}{m^{\frac{\delta}{2}+1}}. 
\end{equation}
\end{definition}

\begin{definition}\label{prodCZoriginal}
Let $\mu$ be a positive Radon measure on $\mathbb R^n$ with power growth $0<\alpha \leq n$.
% , namely,
% for all $I\in \mathcal C$ we have  
% $\mu(I)\lesssim \ell(I)^{\alpha}$.

A function $K:(\mathbb R^{n}\times \mathbb R^{n}) \setminus 
\{ (t, x)\in \mathbb R^{n}\times \mathbb R^{n} : t=x\} \to \mathbb C$ is a Calder\'on-Zgymund kernel if 
%be a function such that 
it is bounded on compact sets of its domain
and
there exist $0<\delta \leq 1$ 
%and $0<\alpha \leq n$ 
%and a function $F$
and bounded functions $L,S,D:[0,\infty)\rightarrow [0,\infty )$ 
%satisfying Definition \ref{LSDF}
%like in Definition \ref{LSDF}
satisfying 
%$$
%\begin{array}{l}
%|K(t,x)|\le C {\displaystyle \frac{1}{|t-x|}} \\
\begin{equation}\label{smoothcompactCZ}
|K(t,x)-K(t',x')|
\lesssim
%\frac{1}{(|x|+|t|)^{\theta}} 
\Big(\frac{(|t-t'|+|x-x'|)}{|t-x|}\Big)^{\delta}\frac{F(t,x)}{|t-x|^\alpha },
\end{equation}
whenever $2(|t-t'|+|x-x'|)<|t-x|$ 
with $F(t,x)=L(|t-x|)S(|t-x|)D(|t+x|)$. 
%$\leq |t'-x'|$
%(careful here). 

% We  
% say that the measure $Kd\mu\times d\mu$ is a Calder\'on-Zygmund kernel
% if the smoothness condition \eqref{smoothcompactCZ} holds and
% \begin{equation}\label{Fbound}
% \sup_{I,J\in \mathcal{D}}F_{K}(I,J)
% %(\rho_{\mu}(I\smland J)+
% \rho_{\mu}(I\smlor J)\lesssim 1. 
% \end{equation}

%We 
%say that $Kd\mu\times d\mu$ is a compact Calder\'on-Zygmund kernel
%if \eqref{smoothcompactCZ} holds 
%and
%% \begin{align}\label{limits}
%% \lim_{\ell(I)\rightarrow \infty }L(\ell(I))\rho_{\mu}(I)
%% &=\lim_{\ell(I)\rightarrow 0}S(\ell(I))\rho_{\mu}(I)
%% \\
%% \nonumber
%% &=\lim_{\rdist(I,\mathbb B)\rightarrow \infty }D(\rdist(I,\mathbb B))\rho_{\mu}(I)=0.
%% \end{align}
%\begin{align}\label{limits}
%\lim_{M\rightarrow \infty }
%\sup_{I\in \mathcal D_M^{c}}
%F(I)\rho_{\mu}(I)=0,
%\end{align}
%where 

%We define
%$
%F(I_{1}, I_{2}, I_{3})=L(\ell(I_{1}))S(\ell(I_{2}))D(1+\rdist(I_{3},\mathbb B))
%$, 
%and $F(I)=F(I,I,I)$.

\end{definition}

%Given $0<\delta \leq 1$, we define
%\begin{equation}\label{Dtilde}
%\tilde{D}%\rho}_{\mu,d}
%(I)=
%%\Big(
%\sum_{k\geq 0}2^{-k\frac{\delta}{2}}
%%\frac{\mu(2^{k}I)}{(2^{k}\ell(I) )^\alpha}
%%\rho(2^kI)
%D(1+\rdist(2^{k}I,\mathbb B)).
%%^2\Big)^{\frac{1}{2}}
%\end{equation}

%If $D$ is such that the corresponding $F$ satisfies \eqref{limits}, then by Lebesgue's Domination Theorem so does $\tilde{D}$.
%%also satisfies \eqref{limits}.

\begin{definition}\label{intrep}
A linear operator $T$ is associated with 
a Calder\'on-Zygmund kernel $K$ if 
%the following representation holds
\begin{equation}\label{kernelrep}
Tf(x) =\int_{\mathbb R^{n}} f(t) K(t,x)\, d\mu(t) 
\end{equation}
for all bounded, compactly supported functions $f$, 
and $x\notin \sup f$.
\end{definition}

%\section{Main result}
We denote $\delta(x)=1$ if $x=0$ and 
$\delta(x)=0$ otherwise. Let $\lfloor \cdot \rfloor$ be the greatest integer function. 
With all these definitions, we can now state the main result in the paper.
\begin{thm}\label{mainresult}
Let $\mu$ be a positive Radon measure on $\mathbb R^n$ with power growth $0<\alpha \leq n$. 
Let $k=n-\lfloor\alpha \rfloor+\delta(\alpha -\lfloor\alpha\rfloor)$.
Let
$T$ be a linear operator with a Calder\'on-Zygmund kernel as in \eqref{smoothcompactCZ} and \eqref{kernelrep} with parameter $0<\delta \leq 1$, and 
bounded on $L^2(\mu)$.

We assume that 
%$\alpha+\delta >\frac{n+1}{2}$, 
%Let $1<p<\infty$.
%and that 
there exist a bounded function $F_T$ and 
$k$ grids of cubes $%\mathcal T_i
\mathcal D_i$  
as in \eqref{grids} 
with $i\in \{1,2, \ldots, k\}$ for which 
the testing condition 
\begin{equation}\label{restrictcompact1-2bddness}
\| \mathbbm{1}_IT\mathbbm{1}_I\|_{L^{2}(\mu)}
+\| \mathbbm{1}_IT^{*}\mathbbm{1}_I\|_{L^{2}(\mu)}
\lesssim 
\mu(I)^{\frac{1}{2}}F_T(I)
\end{equation}
% or
% \begin{equation}\label{restrictcompact1-2bddness2}
% \langle |T\mathbbm{1}_I|, \mathbbm{1}_{J}\rangle 
% +\langle |T^*\mathbbm{1}_I|, \mathbbm{1}_{J}\rangle
% \lesssim 
% \mu(I)^{\frac{1}{2}}F_T(I)
% \end{equation}
%holds for all %($J\subset I$), 
%$I\in \mathcal D$
holds for all $I\in \mathcal D_i$ (only on a sparse family of cubes.) %\cup 2\mathcal D_i$. 
% and satisfying 
% %\begin{align}\label{FT}
% ${\displaystyle \lim_{M\rightarrow \infty }\sup_{I
% \in (\mathcal T_i\mathcal D)_{M}^{c}}\hspace{-.2cm}
% F_T(I)=0}$.
% %\end{align}

%We further assume that either $\alpha +\delta>\frac{n+1}{2}$, and $\delta >\frac{n-1}{2}$, or 
%$\alpha +\delta>1$, and $\alpha>n-2\delta $. 

Then, for $f, g$ bounded with compact support,  
%on $Q\in \mathcal C$,  
there exists a sparse collection of cubes $\mathcal{S}\subset \bigcup_{i=1}^k\mathcal D_i$ %(Q)$ 
%operators $S_{\varepsilon, i}$ 
satisfying 
$$
    |\langle Tf,g\rangle|\lesssim 
    \Lambda_{\mathcal S}(f, g) .
    %:=
    %\sum_{S\in \mathcal S_i} \langle |f|\rangle_{S}
    %\langle |g|\rangle_{S}\mu(S).
$$
%where $\Lambda_{\varepsilon, i}$ are sparse bilinear forms.
%where $\mathcal{S}_i\subset \mathcal T_i\mathcal D(Q)$ is a sparse collection of cubes. 
%${\displaystyle \tilde \varepsilon_{Q}:=\sup_{Q'\in {\mathcal D}(Q)}|\varepsilon_{Q'}|}$, 
%and $\mathcal{D}(Q)$ is the set of dyadic cubes properly contained in $Q$.
The implicit constant in last inequality depends on $n$, $\delta$, $\alpha $, and the implicit constants in the kernel \eqref{smoothcompactCZ} and testing condition  \eqref{restrictcompact1-2bddness}, but it is independent of the norm of $T$. 
\end{thm}

%Since the conditions assumed on $\alpha $ and $\delta$ imply that 
%$\alpha >\frac{n-1}{2}$, this result does not cover all measures of power growth for which Calder\'on-Zygmund operators are known to be bounded \cite{}. This limitation may be a consequence the proof methods rather than an indication of the impossibility of sparse domination estimates for operators with measures of power growth such that $0<\alpha \leq\frac{n-1}{2}$.

%\begin{remark} 
%%The conditions $\alpha +\delta>\frac{n+1}{2}$, and $\delta >\frac{n-1}{2}$ are very restrictive. In fact, the latter implies that $n=1$ or $n=2$, which is a serious constraint. But at least, 
%Either set of conditions on $\alpha $ and $\delta$ allow to use the result for 
%%the result applies to 
%one of the most relevant examples of a Calder\'on-Zygmund operator defined with a non-doubling measure: the Cauchy integral associated with the Hausdorff measure $d\mathcal H^1$ restricted to a general set of the complex plane, for which $n=2$, and $\delta =\alpha =1$. 
%\end{remark}

To prove Theorem \ref{restrictcompact1-2bddness} we can make, without loss of generality, some particular assumptions about the functions $f, g$. We will make the assumptions explicit right now so that the choice of wavelets and the selection process in the stopping time argument make uses of such assumptions. 

First, we assume $f, g$ to be positive. Otherwise we can express each function as the complex linear combination of four positive functions. 

Second, the parameter $\lambda_i$ in \eqref{grids} determines the translation of the grid $\mathcal D_i$. We choose 
all $\lambda_i$ with sufficiently large absolute value so that the supports of $f$ and $g$ are both included in the same quadrant of the $k$ grids we plan to use. In other words, we assume that  the supports of $f$ and $g$  are contained in the first quadrant of all the grids we use. This way we assume that there exists $Q'\in \mathcal D$ such that 
$\sup f \cup \sup g\subset Q'$. 
% (or just  $Q\in \mathcal C$  works better). 

\section{Measures of power growth}

In this section, we prove three technical lemmas about measures of power growth. 
%Then we define and describe two Haar wavelet systems. 
The first result was proved in \cite{V2022}. But, due to its relevance to the current work, we include here its short proof. 

\begin{lemma}\label{zeroborder}
Let 
$\mu$ be a positive Radon measure in $\mathbb R^n$ with power growth $0<\alpha \leq n$. 
Let $Q\in \mathcal C$,  
%$\theta\in (0,1)$, 
$N_0\in \mathbb N$, and $\theta \in (0,1)$ be fixed. 

Let $I\in \mathcal{D}(Q)_{\geq N_0}$ with 
$\ell(I)=2^{-k_I}\ell(Q)$, $0\leq k_I\leq N_0$. For 
$k\geq N_0$,   
%$k>N_0$, 
let $C_k(I)$ be the union of the open cubes 
$J\in \mathcal{\tilde D}(3I)$ such that $\ell(J)=2^{-k}\ell(Q)\leq \ell(I)$ 
and $\lambda_{\theta}(I,J)=\ec(I,J)^{\theta}(1+\inrdist(I,J)) \leq 1$. 
% and such that there exists $I\in \mathcal{D}(Q)_{\geq N_0}$ with $\ell(R)\leq \ell(I)$ and $\inrdist(R,I)-1\leq \lambda $. 
Let finally $C_k=\bigcup_{I\in \mathcal{D}(Q)_{\geq N_0}}  C_k(I)$.
   
Then 
for each $\epsilon>0$ there exist $k_0\in \mathbb N$ such that $\mu(C_k)<\epsilon$ for all $k>k_0$. 

\end{lemma}
\begin{proof}
%For every fixed $k$, 
%we define the following sequence of sets closely related with $C_k$: for $k>M_0$ 

The family of cubes $\mathcal{D}(Q)_{\geq N_0}$ has cardinality 
up to $2^{(N_0+1)n}$. 
For  $I\in \mathcal{D}(Q)_{\geq N_0}$ fixed, we remind that $\mathfrak{B}_{I}=
\cup_{I'\in \child(I)}\partial I'$. Then, 
for each cube $J\in C_k(I)$, 
%the union of the interior of all cubes 
%$R\in \mathcal{D}(3I)$ such that $\ell(R)=2^{-k}\ell(Q)\leq \ell(I)$ and $\inrdist(R,I)-1\leq \lambda $. 
the condition $\lambda(I,J)\leq 1$ implies 
$$
\frac{\dist(J,\mathfrak{B}_{I})}{\ell(J)}
\leq 1+\inrdist(I,J)
%\lesssim 1+\inrdist(\widehat{I},\widehat{J}\,)\leq 
%\ec(I,J)^{-\theta}=
\leq \Big(\frac{\ell(J)}{\ell(I)}\Big)^{-\theta},
$$
and so
$$
\dist(J,\mathfrak{B}_{I})
\leq \Big(\frac{\ell(J)}{\ell(I)}\Big)^{1-\theta}\ell(I).
$$
Since $\ell(I)=2^{-k_I}\ell(Q)$ and $\ell(J)=2^{-k}\ell(Q)$, we have
$
\dist(J,\mathfrak{B}_{I})
\leq 2^{-(k-k_I)(1-\theta)}\ell(I) 
$.

% Let $D_{j}$ be the union of the interior of all cubes 
% $R\in \mathcal{D}(3I)$ such that $\ell(R)=2^{-(j+1)}\ell(Q)$ and 
% $$
% 2^{-(j-k_0+1)(1-\theta)}\ell(I)<\dist(R,\partial I)
% \leq 2^{-(j-k_0)(1-\theta)}\ell(I)
% $$
Now, for $j\geq k$, we define the set 
$$
D_{j}(I)=C_{j}(I)\setminus C_{j+1}(I)
\subset \{x\in 3I/ 
2^{-(j-k_I+1)(1-\theta)}\ell(I)\leq \dist(x,\mathfrak{B}_{I})
<2^{-(j-k_I)(1-\theta)}\ell(I)\}. 
$$

% We note that $D_j$ also depends on $I$ (as well as $Q$ and $N_0$) even though this dependence is not reflected in the notation. 

The sets $(D_j(I))_{j\geq k_I}$ are pairwise disjoint and, for each $k>k_I$, the following inclusions hold: 
$\bigcup_{j\geq k}D_j(I)\subseteq C_k(I)\subseteq \bigcup_{j\geq k-1}D_j(I)$. 
% $\lambda<\inrdist(R,I)\leq 2\lambda $. 
% %for 
% %some $I\in \mathcal{D}(Q)_{\geq N_0}$. 
% For each fixed $k\in \mathbb N$, the sequence of sets $(D_j)_{j\geq k}$ form a disjoint partition of $C_k$. Then
%$C_k=\cup_{j\geq k}D_k$. 
Then
$$
\sum_{j\geq k}\mu(D_j(I))\leq \mu(C_k(I))\leq \mu(C_k)\leq \mu(Q)<\infty .
$$
Therefore, for any $\epsilon>0$ there 
exists  
$k_{0,I}\geq k_I$ such that for all $k\geq k_{0,I}$, 
$$
\sum_{j\geq k}\mu(D_j(I))<2^{-(N_0+1)n}\epsilon.
$$ 

Let $k_0=\max\{ k_{0,I}: I\in \mathcal{D}(Q)_{\geq N_0}\}$. %Since $C_{k+1}\subset C_k$ for all $k\in \mathbb N$, we have 
Then for $k>k_0\geq k_{0,I}$, since $k\geq k_{0,I}+1$, 
\begin{align*}
\mu(C_k)&\leq 
%\mu(C_{k_{0}})\leq 
\sum_{I\in \mathcal{D}(Q)_{\geq N_0}}\mu(C_{k}(I))
\leq \sum_{I\in \mathcal{D}(Q)_{\geq N_0}}\mu(\bigcup_{j\geq k-1}D_j(I))
\\&
\leq \sum_{I\in \mathcal{D}(Q)_{\geq N_0}}\sum_{j\geq k_{0,I}}\mu(D_j(I))
<\sum_{I\in \mathcal{D}(Q)_{\geq N_0}}2^{-(N_0+1)n}\epsilon \leq \epsilon.
\end{align*}

\end{proof}

The next result is a direct application of Abel's Theorem, whose elementary proof can also be found in \cite{V2022}. 
\begin{lemma}\label{Abel} Let %$\beta>\alpha $, 
$I\in \mathcal D$, and $I_m=mI\setminus (m-1)I$  for $m\in \mathbb Z^{+}$. Then
\begin{align*}
\frac{1}
{\ell(I)^{\alpha}}\sum_{m\geq 1}\frac{1}{m^{\alpha +\delta }}
\mu(I_m)
\lesssim \sum_{m}\frac{\mu(mI)}{\ell(mI)^{\alpha} }\frac{1}{m^{\delta +1}}
\lesssim \rho(I). 
\end{align*}

\end{lemma}

The last lemma is a new result that describes how to find, inside a given cube with non-zero measure, a family of cubes with collective arbitrarily small measure. It will be later used to relate the measures of two disjoint cubes (see in Proposition \ref{dualbysquare} the work done with distant cubes and nested cubes). 

\begin{lemma}\label{subsets} Let $\mu$ be a positive Radon measure on $\mathbb R^n$ with power growth $0<\alpha $. 
Let $I\in \mathcal D$ with $\mu(I)>0$, and $a\in \mathbb R$ such that $0<a<\mu(I)$. Then 
there exists $A\subset I$ consisting of a finite union of cubes in $\mathcal D$ of the same side-length such that 
$\frac{a}{2}<\mu(A)\leq a$.  
\end{lemma}
\begin{proof}
Let $s\in \mathbb Z^{+}$ be the smallest positive integer such that $\epsilon =2^{-s}\leq \frac{(C^{-1}\frac{a}{2})^{\frac{1}{\alpha}}}{\ell(I)}$, where $C$ is the implicit constant in  Definition \ref{powergrowth} of the power growth of $\mu$ (that is, $\mu(I)\leq C\ell(I)^{\alpha}$).

Let $\mathcal I_0$ the collection of all cubes $I'\in \mathcal D$ such that $I'\subset I$ and $\ell(I')=\epsilon \ell(I)$. 
Let $A_0=\cup_{I'\in \mathcal{I}_0}I'=I$. 
Then the cubes $I'\in \mathcal I_0$ are pairwise disjoint with $\mu(I')\leq C\epsilon^{\alpha}\ell(I)^{\alpha}\leq \frac{a}{2}$.  
 Moreover, $\card(\mathcal I_0)=\epsilon^{-n}$, and $\mu(A_0)=\sum_{I'\in \mathcal I_0}\mu(I')=\mu(I)>a$. 

We choose any cube $I'_0\in \mathcal{I}_0$.  Since $\mu(I_0')\leq \frac{a}{2}$, we have $\mu(A_0\setminus I'_0)=\mu(A_0)-\mu(I'_0)>\frac{a}{2}$. We define 
$\mathcal I_1=\mathcal{I}_0\setminus \{ I_0'\}$, and $A_1=\cup_{I'\in \mathcal I_1}I'$.
Then  $\mathcal{I}_1$ is a non-empty finite family of pairwise disjoint cubes, 
$\card( \mathcal I_1)=\card( \mathcal I_0)-1<\card( \mathcal I_0)$, and 
$\mu(A_1)=\sum_{I'\in \mathcal I_1}\mu(I')=\mu(A_0\setminus I'_0)>\frac{a}{2}$.
%$\sum_{I'\in \mathcal I_0\setminus \mathcal I_1}\mu(I')\geq \mu(I_0')>0$. Then the claim is proved. 
%We note that, since there is at least one cube $I_0'$ which is not in $\mathcal I_1$, the cardinality of $\mathcal I_1$ is strictly smaller than the cardinality of $\mathcal I_0$. 

If $\mu(A_1)\leq a$, we set $A=A_1$ and the result is proved. If $\mu(A_1)>a$, we repeat previous argument. We choose any cube $I'_1\in \mathcal{I}_1$. Since $\mu(I_1')\leq \frac{a}{2}$, we have $\mu(A_1 \setminus I'_1)>\frac{a}{2}$. Then we define $\mathcal I_2=\mathcal{I}_1\setminus \{ I_1'\}$ and 
$A_2=\cup_{I'\in \mathcal I_2}I'$, which satisfy: $\mathcal{I}_2$ is a non-empty finite family of disjoint cubes, 
$\card( \mathcal I_2)<\card( \mathcal I_1)$, and  $\mu(A_2)=\mu(A_1\setminus I'_1)>\frac{a}{2}$.

If $\mu(A_2)\leq a$, we set $A=A_2$ and the result is proved. If $\mu(A_2)>a$, we repeat previous argument. 
If necessary, the process is iterated up to $r=\epsilon^{-n}-2=2^{sn}-2$ times. This generates a sequence of collections of cubes 
$\mathcal I_r\subset \mathcal I_{r-1}\subset \ldots \subset \mathcal I_{1}\subset \mathcal I_{0}$, each collection strictly included in the previous one, with $\card( \mathcal I_i)=\card( \mathcal I_{i-1})-1$, and the corresponding sets $A_i=\cup_{I'\in \mathcal I_i'}I'$. 
We then have $\card(\mathcal I_r)=\card(\mathcal I_{0})-r=\epsilon^{-n}-(\epsilon^{-n}-2)=2$, 
and $\mu(A_i)>\frac{a}{2}$ for all $i\in \{0, \ldots, r\}$. 

Now we reason as follows. Since $\mu(I')\leq \frac{a}{2}$ for each $I'\in \mathcal I_r$, we have 
$$
\mu(A_r)=\sum_{I'\in \mathcal I_r}\mu(I')\leq \card(\mathcal I_r)\frac{a}{2}=a.
$$
Then we set $A=A_r$ and the result is completely proved.

\end{proof}

%For a cube $Q$, let $Q^*$ be the cube such that $c(Q^*)=c(Q)$ and $\ell(Q^*)=5\ell(Q)$. 

\section{Two Haar wavelet systems}
We continue with the definition and properties of two Haar wavelet systems that will be used throughout the proof of the main result. We start by modifying  the measure used to define such systems of wavelets. 

\begin{remark}
The kernel representation \eqref{kernelrep}  implies that for all functions $f,g$ bounded with compact disjoint supports, we have
\begin{equation*}
%\label{kernelrep2}
\langle Tf, g\rangle  =\int_{\mathbb R^{n}}\int_{\mathbb R^{n}} f(t) g(x) K(t,x)\, d\mu(t) d\mu(x). 
\end{equation*}
Moreover, 
if $Q'\in \mathcal D$ is such that $\sup f \cup \sup g\subset Q'$, we 
consider the measure $\tilde \mu$ defined and follows: 
for $A\in \mathcal D(Q')$, $\tilde \mu(A)=\mu(A)$, while 
for $A\in \mathcal D(\mathbb R^n\setminus Q')$, $\tilde \mu(A)=C\ell(A)^\alpha$, 
%measures $\mu|_{Q'}, m|_{{Q'}^c}$ 
%%$\mu_Q$ 
%defined for each $A\in \mathcal D$ by $\mu|_{Q'}(A)=\mu(A\cap Q')$ 
% and $m|_{{Q'}^c}(A)=C\ell(A\cap (\mathbb R^n\setminus Q'))^\alpha $, 
 where $C$ is the implicit constant in Definition \ref{powergrowth} of power growth, 
 and then extended by additivity and continuity. 
 %with $m$ the Lebesgue measure of $\mathbb R^n$. 
%Let $\tilde \mu=\mu|_{Q'}+m|_{{Q'}^c}$. 
We then have
\begin{equation*}
%\label{kernelrep3}
\langle Tf, g\rangle  =\int_{Q'}\int_{Q'} f(t) g(x) K(t,x)\, d\tilde \mu (t) d\tilde \mu (x). 
\end{equation*}
%where $\tilde \mu=\mu|_{Q'}+m|_{\mathbb R^n\setminus Q'}$, that is, $\tilde \mu(A)=\mu(A\cap Q')+m(A\cap (\mathbb R^n\setminus Q'))$, with $m$ the Lebesgue measure of $\mathbb R^n$. 
The choice of this measure is to ensure the technical property 
$\lim_{\ell(I)\rightarrow \infty}\langle f\rangle_{I}=0$, which is used 
in Lemma %s \ref{QvsQ} and 
\ref{densityinL2-1}.
%%For this reason, we can assume that $\mu $ is zero outside the support of the argument functions, meaning that $\mu (A)=0$ for each set $A$ disjoint with $Q$.  
\end{remark}

We are planing to define the Haar wavelets by using the measure $\tilde \mu$. 
This will lead to a sparse bilinear form constructed with the measure $\tilde \mu$, 
$\Lambda_{\mathcal S, \tilde \mu}(f,g)$, 
instead of 
$\mu$, $
\Lambda_{\mathcal S, \tilde \mu}(f,g)
$, as we stated. But this is fine since the former is estimated above by the latter as we see. 

\begin{lemma}\label{L1-L2} Let $\mu$ and $\tilde \mu$ as defined before. 
Then $\mu\leq \tilde \mu$. 

Let 
$\langle f\rangle_{S,\mu}=\frac{1}{\mu(S)}\int_{S}f(x)d\mu(x)$ and 
$\langle f\rangle_{S,\tilde \mu}=\frac{1}{\tilde \mu(S)}\int_{S}f(x)d\tilde \mu(x)$. Let $\mathcal S$ be a sparse collection of cubes. For $f,g$ positive functions, let 
$$
\Lambda_{\mathcal S, \tilde \mu}(f,g)=\sum_{S\in \mathcal S}a_S
\langle f\rangle_{S,\tilde \mu}\langle g\rangle_{S,\tilde \mu}\tilde \mu(S), 
$$
and similar for $\Lambda_{\mathcal S, \mu}(f,g)$. Then
$$
\Lambda_{\mathcal S, \tilde \mu}(f,g)\leq \Lambda_{\mathcal S, \mu}(f,g)
$$
\end{lemma}
\begin{proof}
We first check that $\mu(I)\leq \tilde \mu (I)$ for $I\in \mathcal D$. 
\begin{itemize}
\item If $I\subseteq Q'$ then $\mu(I)=\tilde \mu(I)$, by definition.
\item If  $I\cap Q'=\emptyset $ then $\mu(I)\leq C\ell(I)^\alpha =\tilde \mu(I)$, by definition. 
\item If $Q'\subset I$, then 
%As we saw $\tilde \mu(Q)=\mu(Q')+C\cdot m(Q\setminus Q')=\mu(Q')+C(2^{nq}-1)\ell(Q')^n$. On the other hand, 
we can write $I=\cup_{i=1}^{2^{rn}}Q_i$ with $2^r=\frac{\ell(I)}{\ell(Q')}$, $Q_i\in \mathcal D$ pairwise disjoint, $Q_0=Q'$ and $\ell(Q_i)=\ell(Q')$. Then %since $\ell(Q')>1$ and $0<\alpha \leq n$, we get  
$$
\mu(I)=\mu(Q')+
\sum_{i=1}^{2^{nr}-1}\mu(Q_i)
\leq C\ell(Q')^\alpha +
C\sum_{i=1}^{2^{nr}-1}\ell(Q_i)^{\alpha}= %\mu(Q')+C(2^{nq}-1)\ell(Q')^{\alpha}= 
\tilde \mu(Q). 
$$
Moreover we have also shown that in the last case %$\tilde \mu(I)=\mu(Q')+C(2^{nr}-1)\ell(Q')^\alpha$.
$\tilde \mu(I)=C2^{\alpha r}\ell(Q')^\alpha$. 
\end{itemize}

Now we prove that $
\Lambda_{\mathcal S, \tilde \mu}(f,g)\leq \Lambda_{\mathcal S, \mu}(f,g)$ by showing the estimate term by term. 

Since $\tilde \mu|_{Q'}=\mu$ and $f, g$ are supported on $Q'$, we have for $S\in \mathcal S(Q')$, 
$$
\langle f\rangle_{S,\tilde\mu}\langle g\rangle_{S,\tilde\mu}\tilde \mu(S)
=\frac{1}{\tilde \mu(S)}\int_{S}f(x)d\mu(x) \int_{S}g(x)d\mu(x)
= \langle f\rangle_{S,\mu} \langle g\rangle_{S, \mu}\mu(S). 
$$

Since $\mu\leq \tilde \mu$, and $f, g$ are supported on $Q'$, we have for $Q'\subset S\in \mathcal S$, 
\begin{align*}
\langle f\rangle_{S,\tilde\mu}\langle g\rangle_{S,\tilde\mu}\tilde \mu(S)
&= \frac{1}{\tilde \mu(S)}\int_{Q'}f(x)d\mu(x) \int_{Q'}g(x)d\mu(x)
\\
&
\leq \frac{1}{\mu(S)}\int_{Q'}f(x)d\mu(x) \int_{Q'}g(x)d\mu(x)
= \langle f\rangle_{S,\mu} \langle g\rangle_{S, \mu}\mu(S). 
\end{align*}
\end{proof}

For this reason and to simplify notation, we write the new measure $\tilde \mu$ simply by $\mu$, and we assume that it satisfies 
$\lim_{\ell(I)\rightarrow \infty}\langle f\rangle_{I, \mu}=0$. We hope this license will not cause problems.

\begin{definition}\label{Haarbasis}
We define 
the adapted Haar functions %adapted to $I$ with respect to the measure $\mu$ 
as follows: 
for $I \in \mathcal{D}$, 
if $\mu(I)\neq 0$ $$
    h_{I}=\mu(I)^{\frac{1}{2}}\Big(
    \frac{1}{\mu(I)}\mathbbm{1}_{I}-\frac{1}{\mu(\widehat{I}\,)}\mathbbm{1}_{\widehat{I}}\Big).
$$
If $\mu(I)=0$, we define $h_{I}\equiv 0$.
With this notation, each function $h_{I}$ is supported on $\widehat{I}$, it is constant on $I$, and also on $\widehat{I}\setminus I$. 
\end{definition}

We now define an alternative Haar wavelet system, 
%and show that the analog of lemma \ref{adaptedwavelets} 
%and \ref{densityinL2} 
%holds.
which will be used to deal with the paraproducts.
\begin{definition}\label{fullwavelet}

Let $J\in \mathcal{D}$, $Q\in \mathcal C$. For $I\in \mathcal D$ with 
$\mu(I)\neq 0$, we define
$$
h_{I}^{J, Q}=\mu(I)^{\frac{1}{2}}\Big(
\frac{\mathbbm{1}_{I}(c(J))}{\mu(I)} %\chi_{I}(t)
-\frac{\mathbbm{1}_{\widehat{I}}(c(J))}{\mu(\widehat{I}\,)} %\chi_{I_{p}}(t)
\Big)\mathbbm{1}_{Q}. 
$$
If $\mu(I)=0$, we define $h_{I}^{J, Q}\equiv 0$. 
We note that $h_{I}^{J, Q}=0$ if $J\cap \widehat{I}=\emptyset $, and that 
$h_{I}^{J, Q}\mathbbm{1}_{I}=h_I\mathbbm{1}_{I}$ when  $J\subseteq I\subset Q$.
\end{definition} 

\begin{definition}
We define the averaging operators 
$$
E_{I}f=\langle f\rangle _{I}\mathbbm{1}_{I},
\hspace{2cm}
E_{I}^{J,Q}(f)=\langle f\rangle _{I}\mathbbm{1}_{I}(c(J))\mathbbm{1}_{Q}.
$$

\end{definition}

\begin{definition}\label{projection}
A family of cubes $\mathcal S\subset \mathcal D$ is convex if whenever $I,J\in \mathcal S$ and $K\in \mathcal D$ with 
$J\subset K\subset I$ we have $K\in \mathcal S$.

Let $\mathcal{S}\subset \mathcal{D}$ either finite or convex. We define the \emph{projection operator} $P_{\mathcal{S}}$ by $$P_{\mathcal{S}}f=\sum_{I\in \mathcal{S}}\langle f,h_I \rangle h_I$$ 
%and also
%$$
%    P_{\mathcal{S}}^{\perp}f:= (I-P_{\mathcal{S}})f=\sum_{Q \in %\mathcal D\setminus \mathcal{S}}\langle f,h_Q\rangle h_Q
%$$ 
with pointwise almost everywhere convergence (or in $L^2(\mu)$). 

Given $Q\in \mathcal C$, we remind that $\mathcal{D}(Q)_{\geq N}=\{ I\in \mathcal{D}(Q)/ \ell(I)\geq 2^{-N}\ell(Q)\}$. Then we write the projection to $Q$ up to scale $N$ as 
$$
P_{Q}^Nf=\sum_{I\in \mathcal{D}(Q)_{\geq N}}\langle f,h_I \rangle h_I.
$$

\end{definition}

As shown in \cite{V2022}, we have the following representation results:
\begin{lemma}\label{densityinL1}
For $R\in \mathcal{D}\cup \tilde{\mathcal D}$ and $f$ 
locally integrable we have
%bounded and compactly supported
\begin{align}\label{Deltaincoord1}
\sum_{I\in \ch(R)}\langle f, h_{I}\rangle 
h_{I}
%&=\sum_{I\in \child(R)}\big(\langle f\rangle_{I}
%-\langle f\rangle _{R}\big)\mathbbm{1}_{I}
=\Big(\sum_{I\in \child(R)}E_{I}f\Big)-E_{R}f
\end{align}
and 

\begin{align}\label{Deltaincoord2}
\sum_{I\in \ch(R)}\langle f, h_{I}\rangle 
h_{I}^{\widehat{J}, Q}
&=\Big(\sum_{I\in \child(R)}E_{I}^{\widehat{J}, Q}f\Big)-E_{R}^{\widehat{J}, Q}f
%\\
%\nonumber
%&=\Big(\sum_{I\in \child(R)}\langle f\rangle_{I} \mathbbm{1}_{I}(c(\widehat{J}\,))%\chi_{I}
%\Big)\mathbbm{1}_{S}
%-\langle f\rangle_{R}\mathbbm{1}_{R}(c(\widehat{J}\,))\mathbbm{1}_{S}
\end{align}
%for $f$ bounded, compactly supported and with mean zero 
almost everywhere with respect to $\mu$. 
\end{lemma}

\begin{lemma}\label{diftheo}  
For $f\in L^2(\mu)$ supported on $Q'\in \mathcal D$ and $N>0$, we have 
\begin{align*}
\|E_{N}f\|_{L^{2}(\mu)}^2&\leq  \| f\|_{L^{2}(\mu)}^2,
\end{align*}
and 
\begin{align}\label{limit2}
\lim_{N\rightarrow \infty }\| f-E_{-N}f\|_{L^{2}(\mu)}^2=0.
%, (I need to prove)
\end{align}
\end{lemma}
\begin{proof} We assume $f$ positive. 
%If $\langle |f|\rangle_I'\neq 0$ then $I\cap Q'\neq \emptyset $. 
%If $I\subseteq Q'$ then $\tilde \mu (I)=\mu(I)$. If $Q'\subsetneq I$ then $1\leq \ell(Q')\leq \ell(I)$ and $I\setminus Q'$ consists of 
%$2^n-1$ cubes $I_i'$ with $\ell(I_i')\geq \ell(Q')\geq 1$. Moreover, 
%$m(I_i')=\ell(I_i')^n\geq \ell(I_i')^\alpha \geq C^{-1}\mu(I_i')$. 
%Then 
%$\tilde \mu(I)=\mu(Q')+C\sum_{i=1}^{2^n-1}m(I_i')\geq \mu(Q')+\sum_{i=1}^{2^n-1}\mu(I_i')\geq  \mu(I)$
For $N\in \mathbb N$, and $I\in \mathcal D$ with $\ell(I)=2^{N}$, since $\mu(I)\leq \tilde \mu(I)$, we have 
\begin{align*}
\|E_{N}f\|_{L^{2}(\mu)}^2&\leq \sum_{\substack{I\in \mathcal D\\\ell(I)=2^{M}}}\langle f\rangle_I^2\mu(I)\leq \sum_{\substack{I\in \mathcal D\\\ell(I)=2^{N}}}\langle f^2\rangle_I\mu(I)
\\&
=\sum_{\substack{I\in \mathcal D\\\ell(I)=2^{N}}}\int_{I}f(x)^2d\tilde \mu(x)\frac{\mu(I)}{\tilde \mu(I)}\leq \| f\|_{L^{2}(\mu)}^2. 
\end{align*}

For $I\in \mathcal D$ with $\ell(I)=2^{-N}<\ell(Q')$, since $\sup{f}\subset Q'$, we have $\int_{I} f(x)d\tilde \mu(x)\neq 0$ if and only if 
$I\subset Q'$. In that case, $E_{-M}f$ is supported on $Q'$, $E_{M}f= \sum_{I\subset Q': \ell(I)=2^{-N}}\langle f\rangle_I\mathbbm{1}_{I}$,  and for such cubes $\tilde \mu(I)= \mu(I)$. Then, by the classical use of Lebesgue's Differentiation Theorem and Lebesgue's Domination Theorem, we can prove \ref{limit2}. 

\end{proof}

The following Plancherel's inequality was also proved in \cite{V2022}:
\begin{lemma}\label{Planche} Let $f\in L^2(\mu)$ compactly supported on $Q\mathcal D$. Then $ \| P_{Q}^N\|_{L^2(\mu)}
\lesssim \|f\|_{L^2(\mu)}$ and 
$$
\sum_{I\in \mathcal D}
|\langle f,h_{I}\rangle |^{2}
\lesssim  \| f\|_{L^{2}(\mu)}^{2}.
$$
\end{lemma}
The proof of Lemma \ref{Planche} requires three properties. First, cancellation and orthogonality properties of the Haar wavelets, both of which are consequence of the actual expression in the definition of the wavelets and are independent of the measure used in their definition. It also requires the property that $\lim_{\ell(I)\rightarrow 0}E_{I}f=f$, which is ensured by Lemma \ref{diftheo}. This is also valid for $\tilde \mu$ since for each cube $I\subset \supp{f}\subset Q'$ we have
$\tilde \mu(I)= \mu(I)$. 

%The classical decomposition 
%$
%f=\lim_{N\rightarrow \infty }P_{Q^N}(f)+E_{Q}(f)
%%=\lim_{M\rightarrow \infty }
%%\sum_{\substack{I\in \mathcal D(Q)\\ 2^{-M}\leq \ell(I)}}\langle f, h_{I}\rangle 
%%h_{I}
%%%\sum_{I\in \mathcal D}\langle f, \tilde{\psi}_{I}\rangle \psi_{I}+
%%%\langle f,\tilde{\uppsi}_{I}\rangle \uppsi_{I}
%%+\langle f\rangle_Q\mathbbm{1}_{Q}, 
%$
%either with $\mu$-a.e. pointwise or $L^2(\mu)$ convergence, was used in \cite{} to estimate the $L^2(\mu)$ norm of the operator. However, this decomposition does not help in the current project since it leads to terms of the form
%$$
%\sum_{I\in \mathcal D(Q)_{\geq N}}
%\langle f\rangle_Q \langle g, h_{J}\rangle \langle T\mathbbm{1}_{Q}, h_J\rangle .
%$$
%Such terms can be successfully treated as a partial paraproduct when calculating the $L^2(\mu)$ norm, but when searching for a sparse domination, it yields terms for the form
%$\langle f\rangle_Q  \langle g\rangle_S \mu(S)$, which are not satisfactory. Instead, we search for a decomposition of the form
%$
%f=\lim_{N\rightarrow \infty }P_{Q}^N(f) 
%$. The closest to that type of decomposition we have been able to obtain is contained in the following proposition. 

\begin{lemma}\label{clasdecomp} Let $f $ bounded, and compactly supported.
Let $Q'\in \mathcal D$ with $\sup{f}\subset Q'$. 
For each cube $Q\in \mathcal Q$ with $Q'\subset Q$, and $\epsilon >0$ there exists $N\geq 1$ such that
%$f=\lim_{N\rightarrow \infty }P_{N,Q}f+E_Qf$,  
\begin{equation*}
D_{Q}^Nf=f-(P_{Q}^Nf+E_Qf)
\end{equation*}
satisfies that $D_{Q}^Nf$ supported on $Q$ and $\| D_{Q}^Nf\|_{L^{2}(\mu)}<\epsilon $. 
\end{lemma}

\begin{lemma}\label{QvsQ}
Let $f, g$ bounded, and compactly supported.
Let $Q'\in \mathcal D$ with $\sup{f}\cup \sup{g}\subset Q'$. Then 
there exist $Q\in \mathcal D$ such that $Q'\subset Q$, and 
\begin{equation}\label{Q-Q}
\tilde \mu(Q)^{\frac{1}{2}}\geq \| T\|\Big(\frac{\| f\|_{L^{2}(\mu)}}{1+\langle f\rangle_{Q'}}+\frac{\| g\|_{L^{2}(\mu)}}{1+\langle g\rangle_{Q'}}\Big) . 
\end{equation}
%and the stopping time argument \eqref{} starting at $Q$ results in a sparse collection of cubes $\mathcal S$ such that $Q'\in \mathcal S$. 
\end{lemma}
\begin{proof}
%We first assume the $\| \mu\|=\infty $. 
Let $q\in \mathbb N$ be large enough so that the only cube $Q\in \mathcal D$ such that $Q'\subset Q$,  and 
$\ell(Q)=q\ell(Q')$, satisfies 
$$
\ell(Q)\geq C^{-1}\Big(%C\ell(Q')^\alpha+ 
\| T\|^{2}\big(\frac{\| f\|_{L^{2}(\mu)}}{1+\langle f\rangle_{Q'}}+\frac{\| g\|_{L^{2}(\mu)}}{1+\langle g\rangle_{Q'}}\big)^2\Big)^{\frac{1}{\alpha}}. 
$$
Since %$\tilde \mu(Q)=\mu(Q')+m(Q\setminus Q')\geq m(Q\setminus Q')\geq \ell(Q)^\alpha-\ell(Q')^\alpha$, 
$\tilde \mu(Q)=C\ell(Q)^\alpha$, 
we get \eqref{Q-Q}. 

\end{proof}

\begin{lemma}\label{densityinL2-1}
Let $f, g$ bounded, and compactly supported.
Let $Q'\in \mathcal D$ with $\sup{f}\cup \sup{g}\subset Q'$. Then 
there exist $Q\in \mathcal D$ such that $Q'\subset Q$ and 
\begin{equation}\label{representationoff2}
|\langle Tf, g\rangle  |\lesssim \lim_{N\rightarrow \infty }\Big|
\sum_{I,J\in \mathcal D(Q)_{\geq N}}\langle f, h_{I}\rangle 
\langle g, h_{J}\rangle \langle Th_{I}, h_{J}\rangle \Big| +\Lambda_{\mathcal S}(f,g). 
\end{equation}
%holds  
%almost everywhere 
\end{lemma}
\begin{proof}
By noting that 
$$
\sum_{I,J\in \mathcal D(Q)_{\geq N}}\langle f, h_{I}\rangle 
\langle g, h_{J}\rangle \langle Th_{I}, h_{J}\rangle =\langle T(P_{Q}^Nf), P_{Q}^Ng\rangle , 
$$
we plan to prove that there exist $N\in \mathbb N$ and $Q\in \mathcal D$ such that 
\begin{equation*}
|\langle Tf, g\rangle -\langle T(P_{Q}^Nf), P_{Q}^Ng\rangle |\lesssim \Lambda_{\mathcal S}(f,g). 
\end{equation*}

We choose $Q\in \mathcal D$ the cube given by Lemma \ref{QvsQ}. Then, by Lemma \ref{clasdecomp}
%First, since $f=\lim_{N\rightarrow \infty }P_{N,Q}f+E_Qf$, 
there is $N\geq 1$ such that 
$f=P_{Q}^Nf+E_Qf+D_{Q}^Nf$, with $D_{Q}^Nf$ supported on $Q$ and  
\begin{equation}\label{DN}
\| D_{Q}^Nf\|_{L^{2}(\mu)}+\| D_{Q}^Ng\|_{L^{2}(\mu)}
\leq 
(\| T\|(\| f\|_{L^{2}(\mu)}+|\| g\|_{L^{2}(\mu)}))^{-1}\Lambda_{\mathcal S}(f,g).
\end{equation}
Then 
\begin{align*}
\langle Tf, g\rangle -\langle T(P_{Q}^Nf), P_{Q}^Ng\rangle 
&=\langle g\rangle_Q\langle TP_{Q}^Nf, \mathbbm{1}_Q\rangle
+\langle T(P_{Q}^Nf), D_{Q}^Ng\rangle 
\\
&
+\langle f\rangle_Q\langle T\mathbbm{1}_Q, g\rangle
+\langle T(D_{Q}^Nf), g\rangle . 
\end{align*}
Now, since $\sup{f}\subset Q'$, $\tilde \mu|_{Q'}=\mu$, and $\| P_{Q}^Nf\|_{L^{2}(\mu)}\leq \| f\|_{L^{2}(\mu)}$, we have for the first term  
$$
\langle g\rangle_Q|\langle TP_{Q}^Nf, \mathbbm{1}_Q\rangle |
\leq 
\frac{1}{\tilde \mu(Q)}\int_{Q'}g(x)d\mu(x)\|T\|| \| f\|_{L^{2}(\mu)}\mu(Q)^{\frac{1}{2}}. 
$$

From the estimate $\mu\leq \tilde \mu$ shown in Lemma \ref{L1-L2}, inequality \eqref{Q-Q}, and the fact that $Q'\in \mathcal S$, we get 
\begin{align*}
\langle g\rangle_Q|\langle TP_{Q}^Nf, \mathbbm{1}_Q\rangle |
&
\leq \frac{1}{\tilde \mu(Q)^{\frac{1}{2}}}\| T\|\| f\|_{L^{2}(\mu)}  \int_{Q'}g(x)d\mu(x) 
\leq \langle f\rangle_{Q'} \int_{Q'}g(x)d\mu(x) 
\\
&
=\langle f\rangle_{Q'}\langle g\rangle_{Q'}\mu(Q')
\leq \Lambda_{\mathcal S}(f,g). 
\end{align*}

By the choice of $N$ in \eqref{DN}, we have for the second term 
$$
|\langle T(P_{Q}^Nf), D_{Q}^Ng\rangle |
\leq \|T\||\| f\|_{L^{2}(\mu)}\| D_{Q}^Ng\|_{L^{2}(\mu)}
\leq \Lambda_{\mathcal S}(f,g). 
$$

The third term can be treated as we did with the first one by only changing the roles played by $f$ and $g$, while the last term can be treated as we did with the second one.

\end{proof}

\section{Construction of a sparse collection of stopping cubes}

We use a stopping-time argument to select an appropriate sparse collection of cubes $\mathcal S$. Later this collection  will be used to estimate the dual pair of a singular integral operator by a sparse sub-bilinear form supported on $\mathcal S$. 

The cubes in this collection either belong to one of $k$ previously selected dyadic grids $\mathcal{D}_i$ or $\mathcal{\tilde D}_i$, or they are translations of cubes in those grids. 
%or even sub-collections of cubes in those grids. 
We remind that $k=n-\lfloor\alpha\rfloor+\delta(\alpha -\lfloor\alpha\rfloor)$ as given in Theorem \ref{mainresult}.
But, to simplify notation, we omit such dependence for now.

Given $f,g$ bounded and compactly supported, we fix $Q'\in \mathcal D$ such that 
$\sup{f}\cup \sup{g}\subset 2^{-1}Q'$, and $Q\in \mathcal D$ such that $Q'\subset Q$, $\ell(Q')=2^q\ell(Q')$, and satisfies \ref{Q-Q} of Lemma \ref{QvsQ}. 
%We may take $Q=[-2^N, 2^N]^n$ for $N$ sufficiently large.  
We then fix $C>0$ so that $0<\tau=\frac{k}{C/4-1}<2^{-\alpha q}<1$. 

For each fixed dyadic grid $\mathcal D=\mathcal D_s$ with $s\in \{1, \ldots, k\}$, we define families of cubes $\mathcal S^{i}$ for $i\geq 0$ by induction. 
Let first $\mathcal S^{-1}=\{R\in \mathcal D: Q'\subsetneq R\subset Q\}$, and $\mathcal S^{0}=\{Q'\}$. 
 
We assume that $\mathcal S^{j}$ for all $0\leq j<i$ have been defined. To define $\mathcal S^i$, we proceed as follows. 

For $I\in \mathcal S^{i-1}$, 
%that is minimal with respect to the inclusion, 
%let $\varphi_1=f$, $\varphi_2=g$, $\varphi_3=T\mathbbm1_{I}$, and $\varphi_4=T^*\mathbbm1_{I}$. %and let $\tilde I\in I^{\rm sp}$. 
%We define $\mathcal S^i_0$ as the collection containing $\child(I)$ and the cubes $R\in \child_2(I)$ such that 
%$R\not\subset 2^{-1}I$. 
%Fixed $r\in \{1,2,3,4\}$, we define
let $\tilde{\mathcal S}^{i}(I)$ be the collection of cubes 
$J\in \mathcal D(2^{-1}I)$ that are maximal with respect to the inclusion and such they satisfy the stopping condition
\begin{equation}\label{selectiondef}
\langle |f|\rangle_{J}>C \langle |f |\rangle_{I}  
\hskip20pt \textrm{ or } \hskip20pt 
\langle |g|\rangle_{J}>C \langle |g |\rangle_{I} . 
\end{equation}% \end{equation}
We now add the offspring of each selected cube by defining
$\mathcal S^{i}(I)=\tilde{\mathcal S}^{i}(I)\cup \big(\bigcup_{J\in \tilde{\mathcal S}^{i}(I)}\{ K\in \child_1(J)\cup \child_2(J) \}\big)$. 

Now we consecutively define the collection of cubes 
%$\mathcal S^{i}(I)=\bigcup_{r=1}^4\mathcal S^{i}_r(I)$, %\cup \tilde{\mathcal S}^i_r(I)$.
%,  and $\tilde{\mathcal S}^i(I)=\bigcup_{r=0}^4\tilde{\mathcal S}^{i}_r(I)$. 
%Next we set 
$\mathcal S^i=\bigcup_{I\in \mathcal S^{i-1}}\mathcal S^i(I)$, 
%$\tilde{\mathcal S}^i=\bigcup_{I\in \mathcal S^{i-1}}\tilde{\mathcal S}^i(I)$
and 
$\mathcal S_{\mathcal D_s}=\bigcup_{i=0}^\infty \mathcal S^i$. 
Finally, let $\mathcal S=\bigcup_{s=1}^k \mathcal S_{\mathcal D_s}$.

We denote by $\mathcal U(S)$ all cubes $J\in \mathcal J(S)\setminus \mathcal S(S)$ as the collection of the cubes that are not selected.

%-------
%
%%let $\varphi_1=f$, $\varphi_2=g$, $\varphi_3=T\mathbbm1_{I}$, and $\varphi_4=T^*\mathbbm1_{I}$. 
%%Fixed $r\in \{1,2,3,4\}$, 
%we define two collections of cubes: 
%$\mathcal S^{i}_1(S)$ as 
%$I\in \mathcal D(2^{-1}S)$ and $J\in \mathcal D(2^{-1}S)$  with $J_p\in ?$
%such that 
%\begin{equation}\label{selection3} 
%\langle |g|\rangle_{J}>C \langle |g|\rangle_{S}
% \hskip40pt {\rm or } \hskip40pt 
%\langle |f|\rangle_{I}>C \langle |f|\rangle_{S} 
%\end{equation}
%and both are maximal in their corresponding collections with respect to the inclusion. 

\begin{proposition}
The family $\mathcal S$ is sparse. 
\end{proposition}
\begin{proof} 
%\begin{remark} 
We divide the proof in two parts: the cubes in $\mathcal S^{i}$ for $i\geq 0$ and the cubes in $\mathcal S^{-1}$.

1) For each $I\in \mathcal S^{i-1}$, the family of cubes %$S^{i}_r(I)$ and 
$S^{i}(I)$ has been defined only for the single index $i$. 
For this proof, we need some additional collections of cubes. 

We first extend that definition to larger indexes %larger than $i$
: for $j>i$, we define $S^{j}(I)%=S^{j}\cap \mathcal S(I)
=\{ J\in S^{j}: J\subset I\}$, that is, the family of cubes selected in the $j$-th step that are contained in $I$. 

%In a similar way, we also define 
%$\mathcal S^i_r=\bigcup_{I\in \mathcal S^{i-1}}\mathcal S^i_r(I)$ and 
%$S^{j}_r(I)%=S^{j}_k\cap \mathcal S(I)
%=\{ J\in S_r^{j}: J\subset I\}$ as the family of cubes selected with condition $r$ in the $j$-th step that are contained in $I$.
%%\end{remark}
%%Similar definitions apply for $\tilde{\mathcal S}^{j}(I)$. 

Now, we start the actual proof by fixing a dyadic grid $\mathcal D=\mathcal D_s$. 
Let $I\in \mathcal S\cap \mathcal D_s$ be fixed and let $i\geq 0$ such that $I\in \mathcal S^i$. 
%For $j>i+2$, we define $\mathcal S^{j}(I)$ the family of cubes $J\in \mathcal S$ such that  $\widehat{J}\subseteq I$, and 
%we denote by 
%$\mathcal S(I)=\bigcup_{j>i}\mathcal S^{j}(I)$
%the family of all cubes 
%$J\in \mathcal S$ such that  $\widehat{J}\subseteq I$. 
The cubes $J\in \tilde{\mathcal S}^{i+1}(I)$ are included in $I$, and they are 
%can be divided into two groups: those $J\subset 2^{-1}I$ are 
pairwise disjoint by maximality and satisfy the inequality in \eqref{selectiondef}.
%; those $J\in \child(I\setminus 2^{-1}I)$ are pairwise disjoint and disjoint with any cube in previous group by definition. Moreover, 
%$\sum_{J\in \child(I\setminus 2^{-1}I)}\mu(J)\leq \mu(I)$. 
%We enumerate the functions $\varphi_k\in \{f, g, T\mathbbm1_{I}, T^*\mathbbm1_{I}\}$ and the corresponding families of chosen cubes
%$\mathcal S^{i}_k(I)$ for $k=1,2,3,4$. Then 
%%We only show two cases: when 
%%$\langle |f|\rangle_{J}>C\langle |f|\rangle_{I}$, and when 
%%$\langle |T\mathbbm1_{I}|\rangle_{J}>C\langle |T\mathbbm1_{I}|\rangle_I $. The other cases can be proved using similar ideas. 
%%In the first case we have that 
%Then
%\begin{align*}
%\sum_{J\in \mathcal S^{i+1}(I)}\mu(J)
%&\leq  \sum_{r=1}^{4}\sum_{J\in \mathcal S^{i+1}_r(I)}\mu(J)
%%+\sum_{J\in \child(I\setminus 2^{-1}I)}\mu(J))
%\leq \sum_{r=1}^{4}C^{-1}\langle |\varphi_r |\rangle_{I}^{-1}\sum_{J\in \mathcal S^{i+1}_r(I)}\int_{J}|\varphi_r |d\mu 
%%+4\mu(I)
%\\
%&
%\leq C^{-1}\sum_{r=1}^{4}\langle |\varphi_r| \rangle_{I}^{-1}\int_{I}|\varphi_r |d\mu
%%+4\mu(I)
%\leq 4C^{-1}\mu(I). 
%\end{align*}
Then
\begin{align*}
\sum_{J\in \tilde{\mathcal S}^{i+1}(I)}\mu(J)
&\leq C^{-1}\Big(\langle |f |\rangle_{I}^{-1}\sum_{J\in \tilde{\mathcal S}^{i+1}(I)}\int_{J}|f |d\mu 
+\langle |g |\rangle_{I}^{-1}\sum_{J\in \tilde{\mathcal S}^{i+1}(I)}\int_{J}|g |d\mu \Big)
%+4\mu(I)
\\
&
\leq C^{-1}\Big(\langle |f |\rangle_{I}^{-1}\int_{I}|f |d\mu 
+\langle |g |\rangle_{I}^{-1}\int_{I}|g |d\mu \Big)
%+4\mu(I)
\leq 2C^{-1}\mu(I). 
\end{align*}
Moreover, for each $J\in \tilde{S}^{i+1}(I)$ we clearly have  
$
\sum_{K\in \child_1(J)\cup \child_2(J)}\mu(K)=2\mu(J)
$
and so, 
\begin{align*}
\sum_{J\in \mathcal S^{i+1}(I)}\mu(J)
&\leq 4C^{-1}\mu(I). 
\end{align*}

By repeating previous argument twice, we have  
\begin{align*}
\sum_{j=i+1}^{i+2}\sum_{K\in \mathcal S^{j}(I)}\mu(K)
&=\sum_{J\in \mathcal S^{i+1}(I)}\mu(J)+\sum_{J\in \mathcal S^{i+1}(I)}\sum_{K\in \mathcal S^{i+2}(J)}\mu(K)
\\
&
\leq \sum_{J\in \mathcal S^{i+1}(I)}\mu(J)+\sum_{J\in \mathcal S^{i+1}(I)}(C/4)^{-1}\mu(J)
\leq \big((C/4)^{-1}+(C/4)^{-2}\big)\mu(I). 
\end{align*}
Now, by reiteration, we get 
$$
\sum_{\substack{J\in \mathcal S(I)\\ J\in \mathcal D_s}}\mu(J)
=\sum_{j>i}\sum_{J\in \mathcal S^j(I)}\mu(J)
\leq \sum_{j>i}
(C/4)^{-(j-i)}\mu(I)
\leq \frac{(C/4)^{-1}}{1-(C/4)^{-1}}\mu(I)
=\frac{1}{C/4-1}\mu(I). 
%=\frac{\tau }{k}\mu(I)
$$ 
Finally, by the choice of the constant $C$ we have that $\tau= \frac{k}{C/4-1}$ satisfies $0<\tau <1$, and so 
$$
\sum_{J\in \mathcal S(I)}\mu(J)
=\sum_{s=1}^k\sum_{\substack{J\in \mathcal S(I)\\ I\in \mathcal D_s}}\mu(J)
\leq \frac{k}{C/4-1}\mu(I)=\tau \mu(I).
$$

2) On the other hand, the cubes $R\in S^{-1}$ form chain that can be parametrized by side-length: $(R_j)_{j=1}^q$ such that $Q'\subset R_i\subset Q$ with $R_0=Q'$, $\ell(R_j)=2^j\ell(Q')$, $R_q=Q$, and $R_j\in \child(R_{j-1})$.
Moreover,  we saw in the proof of Lemma \ref{L1-L2} that 
%$
%m(R_j\setminus Q')=C(\ell(R_j)^\alpha -\ell(Q')^\alpha)=C(2^{\alpha j}-1)\ell(Q')^\alpha 
%$. 
%Hence 
$\tilde \mu(R_j)=C2^{\alpha j}\ell(Q')^\alpha $. 

With this, for $0\leq j<k$, we have $Q'\subsetneq R_j\subsetneq R_{k}\subset Q$, and 
%$$
%\frac{\tilde \mu(R_{k})}{\tilde \mu(R_j)}
%=\frac{\mu(Q')+C(2^{\alpha k}-1)\ell(Q')^\alpha }{\mu(Q')+C(2^{\alpha j}-1)\ell(Q')^\alpha }
%=\frac{\frac{\mu(Q')}{C\ell(Q')^\alpha }+2^{\alpha k}-1}{\frac{\mu(Q')}{C\ell(Q')^\alpha }+2^{\alpha j}-1}. 
%$$
%We can assume $\ell(Q')>1$ and $\mu(Q')>0$.  
%%Since $0<\frac{\mu(Q')}{C\ell(Q')^n}\leq \frac{\ell(Q')^\alpha}{\ell(Q')^n}\leq 1$, $r_2\geq 1$, 
%%%$2^{nr_2}\geq 2C\geq 2$, 
%%$2^{nr_1}-1\geq \frac{1}{2}\cdot 2^{nr_1}$, and $r_2\geq r_1+1$ we have 
%%$$
%%\frac{\tilde \mu(R_2)}{\tilde \mu(R_1)}
%%<\frac{2^{nr_2}}{2^{nr_1}-1}
%%%\leq \frac{2^{nr_1}}{2^{nr_2}-1}+\frac{1}{2}\leq 2\cdot 2^{n(r_1-r_2)}+\frac{1}{2}
%%\leq 2\cdot 2^{n(r_2-r_1)}. 
%%$$
%%Meanwhile 
%Since $0< \frac{\mu(Q')}{C\ell(Q')^\alpha }\leq 1$,  $k\geq j+1\geq 2$, and so $2^{\alpha k}-1\geq \frac{2}{3}2^{\alpha k}$. Then
%%$2^{nr_2}+C-1\leq 2^{nr_2}+C\leq 2\cdot 2^{nr_2}$, 
%we have 
$$
\frac{\tilde \mu(R_{k})}{\tilde \mu(R_j)}=
%\frac{2^{\alpha k}-1}{2^{\alpha j}} \geq \frac{2}{3}
2^{\alpha (k-j)}.
$$
%For $j'=j+1$, we have 
%$\tilde \mu(R_j)< \frac{3}{2^{1+n}}\tilde \mu(R_{j+1})\leq \frac{3}{4} \mu(R_{j+1})$. 
%That is, $\mu(R_j)=2^{-\alpha (k-j)}\mu(R_k)$ for $0\leq j< k$. 

With this and the previous case 1), we have for each $R_k\in \mathcal S^{-1}$ 
\begin{align*}
\sum_{S\in \mathcal S(R_k)}\tilde \mu(S)&=\sum_{S\in \mathcal S(Q')}\mu(S)
+\sum_{j=0}^{k-1}\tilde \mu(R_j)
\leq \tau \mu(Q')+2^{-\alpha k}\sum_{j=0}^{k-1}2^{\alpha j}\tilde \mu(R_k)
\\
&
\leq  \tau \tilde \mu(Q')+2^{-\alpha k}(2^{\alpha k}-1)\tilde \mu(R_k)
\leq \tau\tilde \mu(R_k)+(1-2^{-\alpha k})\tilde \mu(R_k)
\\
&\leq (\tau+1-2^{-\alpha q})\tilde \mu(R_k). 
\end{align*}
%since $n, k\geq 1$ and so, $1-2^{-\alpha k}\leq 2^{-1}$. 
By the choice of the constant 
$C$ we have 
$\tau<2^{-\alpha q}$ and so, $0<\tau+1-2^{-\alpha q}<1$. 
This shows that $\mathcal S^{-1}$, the collection of cubes above $Q'$, is also sparse. 

%On the other hand, since $J\in \tilde{\mathcal S}^{i+1}(I)$ implies  $J\in \mathcal S^{i+1}(I)$, we have 
%\begin{align*}
%\sum_{J\in \tilde{\mathcal S}^{i+1}(I)}\mu(J)\leq \sum_{J\in \mathcal S}^{i+1}(I)\mu(J)\leq 4C^{-1}\mu(I)
%\end{align*}
%
%
%Finally, we note that for $I\in \tilde{\mathcal S}^{i}$, we have 
%\begin{align*}
%\sum_{J\in \tilde{\mathcal S}^{i+1}(I)}\mu(J)\leq \sum_{J\in \tilde{\mathcal S}^{i+1}(I)}\mu(J)
%\end{align*}
% $$
% \sum_{J\in \mathcal S(I)}\mu(J)
% =\sum_{j>i}\sum_{J\in \mathcal S^i(I)}\mu(J)
% \leq \sum_{j>i}
% (j-i)
% C^{-(j-i)}\mu(I)
% \leq \sum_{j>i}
% C^{-\frac{j-i}{2}}\mu(I)
% \leq \frac{C^{-\frac{1}{2}}}{1-C^{-\frac{1}{2}}}\mu(I)
% <\frac{2}{C^{\frac{1}{2}}}\mu(I)
% $$ 
% for $C>16$.

% On the other hand, for the second case we get
% $$
% \sum_{J\in \mathcal S^{i+1}(I)}\mu(J)
% \leq C^{-1}\langle |T\mathbbm1_{I}|\rangle_I ^{-1}\sum_{J\in \mathcal S^{i+1}(I)}\int_{J}|T\mathbbm1_{I}|d\mu 
% \leq C^{-1}\langle |T\mathbbm1_{I}|\rangle_{I}^{-1}\int_{I}|T\mathbbm1_{I}|d\mu
% =C^{-1}\mu(I)
% $$

\end{proof}

\begin{definition}
For any cube in $I\in \mathcal D$, we have that either $I\in \mathcal{S}^i$ or there are cubes $S\in \mathcal{S}^i$, $S'\in \mathcal{S}^{i+1}$ such that $S'\subsetneq I\subsetneq S$. 
Given $S\in \mathcal S^i$, 
we define 
$$
{\mathcal U}(S)=\mathcal D(S)\setminus \bigcup_{S'\in \mathcal S}\mathcal D(S'), 
$$
that is, the collection of dyadic cubes $I\in \mathcal D$ such that $I\subset S$ and $I\not\subset S'$ for any other selected cube $S'\in \mathcal S^j(S)$ with $j>i$.  This way the sets ${\mathcal U}(S)$ when $S$ varies in $\mathcal S$ form a partition of $\mathcal D$.
\end{definition}

\section{Bump estimates, Paraproducts, and Square functions}

The following result was proved in \cite{V2022}.
\begin{proposition}\label{twobumplemma1} 
Let $T$
be a linear operator with a C-Z kernel 
and parameter $0<\delta <1$. 
Let $I, J\in \mathcal D$ and 
$\theta=\frac{\alpha }{\alpha +\frac{\delta}{2}}\in (0,1)$. 

\vskip5pt
1) When $\dist (\widehat{I},\widehat{J}\,)> 2\ell(I\smlor J)$,
\begin{align}\label{bump1}
|\langle Th_{I},h_{J}\rangle |
&\lesssim \frac{\ec(I,J)^{\delta }}{
(1+\rdist(I,J))^{\alpha+\delta}}
\frac{\mu(I)^{\frac{1}{2}}\mu(J)^{\frac{1}{2}}}
{\ell(I\smlor J)^{\alpha}}
F_{1}(I,J),
\end{align}
with
$F_{1}(I, J)
=L(\dist(\widehat{I},\widehat{J}\,))
S(\ell(\widehat{I}\smland \widehat{J}\,))D(1+\rdist (\langle \widehat{I},\widehat{J}\,\rangle, \mathbb B))
$ if $\dist (\widehat{I},\widehat{J}\,)> 0$ and zero elsewhere.

\vskip10pt
2) When
$0<\dist (\widehat{I},\widehat{J}\,)\leq  2\ell(I\smlor J)$
\begin{align}\label{bump2}
|\langle Th_{I},h_{J}\rangle |
&\lesssim (1+\inrdist(I,J))^{-(\alpha+\delta) }
\frac{\mu(I)^{\frac{1}{2}}\mu(J)^{\frac{1}{2}}}
{\ell(I\smland J)^{\alpha}}
F_{1}(I,J).
\end{align}

\vskip10pt
3) When 
%$\rdist (I_{p},J_{p})=1$  
%and $\inrdist(I_{p},J_{p})>1$
$\dist(\widehat{I}, \widehat{J}\,)=0$, 
$\lambda_{\theta}(\widehat{I},\widehat{J}\, )=\ec(I,J)^{\theta}(1+\inrdist(\widehat{I},\widehat{J}\,))> 1$, and $S\in \mathcal D$ with $I,J\subset S$, we have 
\begin{align}\label{bump3}
|\langle T(h_{I}-{h}^{\hat{J},S}_{I}),h_{J}\rangle |
&\lesssim
(1+\inrdist(I,J))^{-\delta}
\sum_{R\in \{I,\widehat{I}\}}
\Big(\frac{\mu(R\cap J)}{\mu(R)}\Big)^{\frac{1}{2}}
F_{2}(I,J)
\\
\nonumber
&+(1+\inrdist(I,J))^{-(\alpha+\delta)}
\frac{\mu(I)^{\frac{1}{2}}\mu(J)^{\frac{1}{2}}}{\ell(
I\smland J)^{\alpha }}\mathbbm{1}_{\widehat{I}\setminus I}(c(\widehat{J}\,))F_{3}(I,J). 
\end{align}
The functions $F_{2}(I, J)$, $F_{3}(I, J)$ are defined as follows: for $\dist(\widehat{I}, \widehat{J}\,)=0$, $\lambda_{\theta}(\widehat{I},\widehat{J}\,)> 1$,
\begin{align*}
F_{2}(I, J)&=
L(\ell(I\smland J))S(\ell(I\smland J))
\sum_{k\geq 0}2^{-k\delta }\frac{\mu(2^kK)}{\ell(2^kK)}D(1+\rdist(2^kK,\mathbb B)), 
\end{align*} 
and
\begin{align*}
F_{3}(I, J)&=
L(\ell(I\smland J))S(\ell(I\smland J))
\sum_{k\geq 0}2^{-k\delta }D(1+\rdist(2^kK,\mathbb B)), 
\end{align*} 
with 
$K=(1+\inrdist(\widehat{I},\widehat{J}\,))(I\smland J)$, and zero elsewhere.

From now we denote $F(I,J)=F_1(I,J)+F_{2}(I,J)+F_3(I,J)$.

% \vskip10pt
% 3)nope When $\rdist (I_{p},J_{p})\leq 3$ and $\inrdist(I_{p},J_{p})=1$, 
% \begin{align*}
% |\langle Th_{I},h_{J}\rangle |
% &\lesssim
% \ec(I,J)^{\frac{n}{2}}(F_{2}(I,J)+F_{3}(I,J))
% \end{align*}
% where $F_{2}$ 
% is as before and, when $I\neq J$,
% \begin{align*}
% F_{3}(I, J)=\tilde F_{K}(I\smland J) 
% \end{align*}
% while when $I=J$,
% \begin{align*}
% F_{3}(I, J)\! =\! \tilde F_{K}(I)\! +\! F_{W}(I)\! 
% \end{align*}
\end{proposition}

\begin{remark}\label{defBF} We note that $1+\rdist(\widehat{I},\widehat{J})\approx 1+\rdist(I,J)$ and thus, we often use the latter expression to classify cubes. 
% and similar for $\inrdist$. 
%% We note that $F$ dominates all terms 
%% $B_{i}\cdot F_{i}$ in the statement of Proposition \ref{twobumplemma1}. 
%% In fact,  
%% $F_{i}(I,J)\lesssim \tilde F_{K}(I\smland J, I\smland J,  \langle I,J\rangle )+F_{W}(I;M_{T,\epsilon })$. 
%We note that $2\rdist(I_{p},J_{p})-1\leq 1+ \rdist(I,J)\leq 2(\rdist(I_{p},J_{p})+1)$ and 
%$2\inrdist(I_{p},J_{p})\leq \inrdist(I,J)\leq 2\inrdist(I_{p},J_{p})+1$.
\end{remark}

\begin{definition}[Paraproducts]\label{paraproducts}
%Let $\mu, \mu_1$ Radon measures such that 
%$\mu_1\leq \mu$. 
Let $Q\in \mathcal C$ and $\theta \in (0,1)$ be fixed. 
We define the following bilinear forms: 
for $f,g$ bounded functions with $\supp f\cup \supp g\subset Q$
% and mean zero,
\begin{align*}
\Pi_1^i (f,g)
&=\sum_{I\in \mathcal D_i(Q)}
\sum_{\substack{J\in \mathcal D(\widehat{I}\,)\\
%\inrdist(\widehat{I},\widehat{J}\,)>\lambda_\theta
\lambda_{\theta}(\widehat{I},\widehat{J}\,)>1}
%\\ \ell(J)\leq \ell(I), J_p\subset I_p^*
}
\langle f, h_{I}\rangle \langle g, h_{J}\rangle \langle Th_{I}^{\widehat{J},S},h_{J}\rangle, 
%+\langle f, \tilde\psi_{3,I}\rangle \langle g,\tilde \psi_{J}\rangle \langle T\psi_{3,I}^{{\rm full}},\psi_{J}\rangle
\end{align*}
%and $\sum_{I:\ell(J)\leq \ell(I), J\subset 5I}\langle f, \tilde\psi_{I}\rangle\tilde\psi_{I}^{1},$ 
%does not telescope. Then, 
%where
%$$
%\psi_{I}^{{\rm full}}(t)=\mu(I)^{\frac{1}{2}}(\alpha_{I,J}b^{1}_{I^{a}}(t)-\alpha_{I_{p},J}b^{1}_{I_{p}^{a}}(t))
%$$
and 
\begin{align*}
\Pi_2^i (f,g)
&=\sum_{J\in \mathcal D_i(Q)}
\sum_{\substack{I\in \mathcal D(\widehat{J}\,)\\
%\inrdist(\widehat{I},\widehat{J}\,)>\lambda_\theta
\lambda_{\theta}(\widehat{I},\widehat{J}\,)>1}
%\\ \ell(J)\leq \ell(I), J_p\subset I_p^*
}
\langle f, h_{I}\rangle \langle g, h_{J}\rangle \langle Th_{I},h_{J}^{\widehat{I},S}\rangle .
%+\langle f, \tilde\psi_{3,I}\rangle \langle g,\tilde \psi_{J}\rangle \langle T\psi_{3,I}^{{\rm full}},\psi_{J}\rangle
\end{align*}
%with $\lambda_\theta =1+\ec(I,J)^{-\theta}$.

% Then 
% given $\epsilon >0$ there exist 
% $M_0\in \mathbb N $ independent of $Q$ and the functions $f, g$ such that for all $M>M_0$
% $$
% |\Pi (P_M^\perp f,P_M^\perp g) |+|\Pi' (P_M^\perp f,P_M^\perp g) |\lesssim 
% \epsilon \| f\|_{L^2(\mu)}\| g\|_{L^2(\mu)}.
% $$
\end{definition}

\begin{definition}[Square functions]\label{square}
We define the following square functions:
\begin{align*}
S_{1,i}^1f(x)
&=\Big( \sum_{\substack{m\in \mathbb Z\\ m\geq 2}}\frac{1}{m^{\alpha +\delta}}\sum_{I\in \mathcal D_i}
a_{I}
|\langle f,h_{I}\rangle |^2
\frac{1}{\ell(I)^{\alpha}}
\sum_{J \in I_{-e,m}}\frac{\mu(J)}{\mu_{\epsilon }(J_I')}\mathbbm{1}_{J_I'}(x)
\Big)^{\frac{1}{2}}
\\
S_{1,i}^{-1}f(x)
&=\Big( \sum_{\substack{m\in \mathbb Z\\ m\geq 2}}\frac{1}{m^{\alpha +\delta}}\sum_{J\in \mathcal D_i}
b_{J}
|\langle f,h_{J}\rangle |^2 
\frac{1}{2^{e\alpha}\ell(J)^{\alpha}}\sum_{I \in J_{e,m}}\frac{\mu(I)}{\mu_{\epsilon }(J_I')}\mathbbm{1}_{J_I'}(x)
\Big)^{\frac{1}{2}}
%\end{align*}
%$$
%S_{1}^if(x)=\sum_{m\geq 1}\frac{1}{m^{\alpha +\delta -\frac{n-1}{2}}}\Big( \sum_{I\in \mathcal{D}_i}
%a_{I,1}
%|\langle f,h_{I}\rangle |^2
%\frac{\mu(I)}
%{\ell(I)^{2\alpha}}
%\mathbbm{1}_{I_m}(x)\Big)^{\frac{1}{2}}
%$$
%\begin{align*}
\\
S_{2,i}^1f(x)&=\Big( \sum_{I\in \mathcal{D}_i}
a_{I}
|\langle f,h_{I}\rangle |^2
\frac{\mu(I)}
{\ell(I)^{2\alpha}}
\mathbbm{1}_{3I}(x)\Big)^{\frac{1}{2}}
%\geq S_1f
\\
S_{2,i}^{-1}f(x)&=\Big( \sum_{J\in \mathcal{D}_i}
b_{J}\rho(2^eJ)
|\langle f,h_{J}\rangle |^2
\frac{\mathbbm{1}_{J}(x)}{\mu(J)}
\Big)^{\frac{1}{2}}
%\end{align*}
%$$
%S_3^if(x)=\Big( \sum_{I\in \mathcal{T}_i\mathcal{D}}
%a_{I,1}
%|\langle f,h_{I}\rangle |^2
%\sum_{R\in \{I,\hat{I} \}}
%\frac{1}{\mu(R)}\mathbbm{1}_{R}(x)
%\Big)^{\frac{1}{2}}
%$$
%\begin{align*}
\\
S_{3,i}^1f(x)&=\Big( \sum_{I\in {\mathcal D}_i}
a_{I}
|\langle f,h_{I}\rangle |^2
%\frac{\mathbbm{1}_{I}(x)}{\mu(I)}
\sum_{R\in \{I,\hat{I}\}}
\sum_{k=2^{\theta e}}^{2^{e}}\sum_{J\in I_{-e,1,k}}
\frac{\mu(R\cap J)}{\mu(R)}
\frac{\mathbbm{1}_{J'_I}(x)}{\mu_{\epsilon }(J_I')}
\Big)^{\frac{1}{2}}
\\
S_{3,i}^{-1}f(x)&=\Big( \sum_{I\in \mathcal D_i}b_{I}
|\langle f,h_{I}\rangle |^2
\frac{\mathbbm{1}_{I'}(x)}{\mu_{\epsilon} (I')}
\Big)^{\frac{1}{2}}
\end{align*}
where $e\in \mathbb Z$, $m\in \mathbb N$, 
$a_{I}=\sup_{J\in I_{e,m}}F(I,J)$, $b_{J}=\sup_{I\in J_{-e,m}}F(I,J)$.
%and, for $\vec{m}\in \mathbb Z^n$, 
%we denote by $I_m$ the only cube in $\mathcal{T}_i\mathcal D$ such that 
%$\ell(I_m)=\ell(I)$ and 
%$\vrdist{(\widehat{I},\widehat{I_m}\,)}=\vec{m}$.
%and $I_m=mI\setminus (m-1)I$. 
Moreover, for $I\in \mathcal D$, let 
$\tilde I\in \mathcal D$ be such that 
$\ell(\tilde I)=\ell(I)$, $\dist(\tilde I,I)=2^{|e|}\ell(I)$. Now, if  $\mu(\tilde I)\leq \mu(I)$ we define 
$I'=\tilde I$. Otherwise,  
$I'\subset \tilde I$ is a union of a finite collection of cubes in $\mathcal D$ such that 
$\mu(I)/2<\mu(I')\leq \mu(I)$. The existence of such set $I'$ is guaranteed by Lemma \ref{subsets}. Finally, we write
$\mu_{\epsilon} (I)=\mu(I)+\epsilon$ for fixed $\epsilon >0$, which is just a notational tool, since $\mu_{\epsilon}$ is not a measure.

\end{definition}

\begin{remark}
Since the quantity $\mu(I)$ appears in the denominator of the expression defining $S_{2,i}^{-1}$, it is implicitly assumed that the sum runs over the convex family of cubes $I$ such that $\mu(I)\neq 0$.

We also note that, due to the factors $a_{I}$, all square functions depend on the parameters $e$ and $i$. %$\vec{m}$. 
But we often omit that dependence in the notation. 
\end{remark}

\section{Estimates of the dual form \\ by bilinear forms acting on square functions}
In this section we start the proof of the main result Theorem \ref{mainresult}. 
In fact, we prove a result that is a bit more technical than what is actually stated. One does not always know a-priori that a singular integral operator $T$ is bounded. The classical way to proceed is to define appropriate truncations of the operator, which are bounded but with non-convergent bounds, and then prove that the dual pair associated with such truncations can be estimated by a sparse operator with constant independent of the truncations.

We define the following truncation of an operator with a Calder\'on-Zygmund kernel. 
\begin{definition}\label{TgammaQ}
Let $\phi$ be a smooth function such that 
$0\leq \phi(x)\leq 1$, $\sup \phi\subset [-2,2]$,  
$\phi(x)=1$ for all $|x|< 1$ and $|\phi'(x)|\leq 2$.
Let $Q=[-2^N,2^N]^n$ with $N\geq 0$
and $0<\gamma \leq 1$. 
We define the kernel 
\begin{align*}
K_{\gamma,Q}(t,x)&=%\Big(
K(t,x)
(1-\phi(\frac{|t-x|}{\gamma}))
\phi(\frac{4|t|}{\ell(Q)})
\phi(\frac{4|x|}{\ell(Q)}).
\end{align*}
Let 
$T_{\gamma, Q}$ the operator with kernel $K_{\gamma,Q}$ defined via a similar equality to \eqref{kernelrep}.
\end{definition}

The following result was proved in \cite{V2022}:
%we prove that the truncated operators are bounded, have a Calder\'on-Zygmund kernel and satisfy the testing condition uniformly on $\gamma $ and $Q$. 
\begin{lemma}\label{truncaT}
The operator $T_{\gamma, Q}$ is bounded on $L^2(\mu)$ with bounds  depending on $\gamma $ and $Q$. 

On the other hand,  
$K_{\gamma,Q}$ is a Calder\'on-Zygmund kernel with parameter $0<\delta\leq 1$ and constant independent of $\gamma $ and $Q$. Moreover,
the operator $T_{\gamma, Q}$ satisfies the testing condition \ref{restrictcompact1-2bddness} with bounds independent of $\gamma $ and $Q$. 
\end{lemma}

The next proposition shows how to estimate the dual pair associated with the truncated operators by a combination of dual pairs containing all previously defined square functions plus two remaining terms, which will be estimated in other ways. For this result, we set the notation $s_0=1$, $s_e=\sign{(e)}$ for $e\neq 0$. 

%For $\vec{m}\in \mathbb Z^n$, we denote the $l^\infty(\mathbb Z^n)$ norm of $\vec{m}$ by $|\vec{m}|$. 

\begin{proposition}\label{dualbysquare} 
Let $\mu$ be a positive Radon measure on $\mathbb R^n$ with power growth $0<\alpha \leq n$. 
Let $T$
be a linear operator associated with a C-Z kernel 
and parameter $0<\delta \leq 1$. Let $k=n-[\alpha ]+\delta(\alpha -[\alpha])$, and 
$\theta=\frac{\alpha }{\alpha +\frac{\delta}{2}}\in (0,1)$. 
For $Q\in \mathcal C$ and $0<\gamma <1$,  
let $T_{\gamma, Q}$ be the truncated operator of Definition \ref{TgammaQ}. 
Then, for $f,g$ bounded with compact support on $Q$,  
%with mean zero with respect to $\mu$, we have 
\begin{align}\label{estsquare}
|\langle T_{\gamma, Q}f,g\rangle|
%=|\langle Tf,g\rangle|
&\lesssim \sum_{i=1}^{k} 
B_i(f,g)
%\sum_{e\in \mathbb Z}
%\hskip5pt 2^{-|e|\frac{\theta \delta }{2}}\Big( 
%%2^{-|e|\delta }
%%\sum_{\substack{m\in \mathbb Z\\ m\geq 2}}\frac{1}{m^{\beta}}
%\langle S_{1,i}^{s_e}f,S_{1,i}^{-s_e}g\rangle 
%+%2^{-|e|(\theta (\alpha +\delta)-\alpha)}
%%2^{-|e|\frac{\theta \delta }{2}}
%\langle S_{2,i}^{s_e}f,S_{2,i}^{-s_e}g\rangle 
%%+\sum_{i=1}^{k}\sum_{\substack{e\in \mathbb Z\\ e\leq 0}}2^{-|e|\delta }
%%\sum_{\substack{m\in \mathbb Z\\ m\geq 2}}\frac{1}{m^{\beta}}
%%\langle S_{1,i}'f,S_{1,i}g\rangle 
%%\\
%%\nonumber
%%&
%%\hskip50pt 
%+%2^{-|e|\theta \delta }
%\langle S_{3,i}^{s_e}f,S_{3,i}^{-s_e}g\rangle \Big)
%%\\
%%\nonumber
%%&
%%+\sum_{i=1}^{k}\langle S_2^if,S_0^ig\rangle +\langle S_0^if,S_2^ig\rangle 
%%+\langle S_3^if,S_0^ig\rangle +\langle S_0^if,S_3^ig\rangle
%\\
%\nonumber
+ |\Pi_1^i(f,g)|+|\Pi_2^i(f,g)|
\\
&
\nonumber
+\sum_{\substack{I, J\in \mathcal D(Q)\\ \lambda(\widehat{I},\widehat{J}\,)\leq 1}}
|\langle f,h_{I}\rangle |
|\langle g, h_{J}\rangle |
|\langle Th_I,h_J\rangle |, 
%+\langle |f|\rangle_Q\langle |g|\rangle_Q\mu(Q)
\end{align}
%where $\beta=\alpha +\delta-\frac{n-1}{2}>1$.
%where $B_i(f,g)=\langle S_if, S_ig\rangle$.
where 
\begin{equation}\label{bilinearsquare}
B_i(f,g)=\sum_{e\in \mathbb Z}
\hskip5pt 2^{-|e|\frac{\theta \delta }{2}}\Big( 
%2^{-|e|\delta }
%\sum_{\substack{m\in \mathbb Z\\ m\geq 2}}\frac{1}{m^{\beta}}
\langle S_{1,i}^{s_e}f,S_{1,i}^{-s_e}g\rangle 
+%2^{-|e|(\theta (\alpha +\delta)-\alpha)}
%2^{-|e|\frac{\theta \delta }{2}}
\langle S_{2,i}^{s_e}f,S_{2,i}^{-s_e}g\rangle
+%2^{-|e|\theta \delta }
\langle S_{3,i}^{s_e}f,S_{3,i}^{-s_e}g\rangle \Big) , 
\end{equation}
$S_{j,i}^k$ are the square functions of Definition \ref{square}, and $\Pi_j^i(f,g)$ are the paraproducts of Definition 
\ref{paraproducts}. 
\end{proposition}
\proof Let $f,g$ positive bounded compactly supported functions. 
%Otherwise we can express each function as the complex linear combination of four positive functions. 
%We remind that the parameter $\lambda_i$ in \eqref{grids} determines the translation of the grid $\mathcal D_i$. We choose 
%all $\lambda_i$ with sufficiently large absolute value so that the supports of $f$ and $g$ are both included in the same quadrant of the $k$ grids we plan to use. In other words, we assume that  the supports of $f$ and $g$  are contained in the first quadrant of all the grids we use. (or just  $Q\in \mathcal C$ works better). 
Let $Q'\in \mathcal D$ with $\sup{f}\cup \sup{g}\subset Q'$. 
Without loss of generality, we assume %($c(Q)=0$), 
$\mu(Q')\neq 0$, and $\ell(Q')>2$. 
Let  $Q\in \mathcal D$ such that $Q'\subset Q$ and satisfies \ref{Q-Q} of Lemma \ref{QvsQ}. 
We fix   
%such that $F_\mu(Q)<\epsilon$, 
%$0<\gamma <\ell(Q)$ 
%such that $S(\gamma )<\epsilon $ (?) 
%and 
$N_0=\log \frac{6\ell(Q)^{\alpha+2}}{\gamma^{\alpha+1}}$. 
%We fix $f, g$ 
%in the unit ball of $L^{2}(\mu)$ bounded, supported on $Q$, and with mean zero with respect to $\mu$. 
% and satisfying 
% $\|P_{M}^{\perp}T_{\gamma, Q} P_{2M}^{\perp}\|_2\leq 2
% |\langle P_{M}^{\perp}T_{\gamma, Q} P_{2M}^{\perp}f,g\rangle |$.

We start by considering the dyadic grid $\mathcal D=\mathcal D_1$ as denoted in Notation \ref{grids}.
Let $(h_{I})_{I\in {\mathcal D}}$ 
%and $(\psi_{I})_{I\in {\tilde{\mathcal D}}}$ 
be the Haar wavelets frame of Definition \ref{Haarbasis}.  
% %such that 
% %$I\in \mathcal D_1$
% %if $\ell(I)\geq 2^{-N_1}\ell(Q)$, while $I\in \mathcal D_1$ if 
% %$\ell(I)< 2^{-N_1}\ell(Q)$. 
% and $P_{M}$ 
% %and $\tilde P_{M}$ 
% be the lagom projection operators related to that system.  %respectively.
% % we denote both projections simply by $P_{M}$.
% %By Theorem \ref{charofcompact}, 
%We also fix 
%the parameter $\theta=\frac{\alpha }{\alpha +\delta/2} \in (0,1)$. 
%Let $M_0$ such that for all $M>M_0$ we have 
%$M^{-\frac{\delta }{8}}
%+M^{-\alpha (\frac{\alpha +\delta}{\alpha+\delta/2}-1)}
%+M^{-\frac{\alpha \delta}{\alpha +\delta/2}}
%%+3M2^{-\frac{M}{4}\frac{\delta }{2\alpha +\delta}}
%<\epsilon $ and 
%$$
%%\sum_{I\in {\mathcal D}_{M}^{c}}
%\sup_{\substack{J\in  \mathcal D_{M}^{c}\\
%(I,J) \in \mathcal F_{M}}}F_\mu(I,J)
%+F_\mu(J,I)
%<\epsilon , 
%$$
%where $\mathcal D_M$ is given in Definition \ref{FM}.
%Since $k(t)=\int K $. 
%We decompose the dual pair by using this Haar wavelet system: 
%\begin{align}\label{orthopro}
%\langle Tf,g\rangle
%=\hspace{-.3cm}
%\sum_{I\in {\mathcal D}(Q)}\sum_{J\in {\mathcal D}(Q)}
%\langle f,h_{I}\rangle 
%\langle g, h_{J}\rangle
%\langle Th_I,h_J\rangle .
%\end{align}

We remind the notation of Lemma \ref{zeroborder}:
$\mathcal{D}(Q)_{\geq N_0}=\{ I\in \mathcal{D}(Q)/ \ell(I)\geq 2^{-N_0}\ell(Q)\}$; for 
$I\in \mathcal{D}(Q)_{\geq N_0}$ and 
$r>N_0$, $C_r(I)$ denotes the union of the open cubes 
$R\in \mathcal{\tilde D}(Q)$ such that $\ell(R)=2^{-r}\ell(Q)\leq \ell(I)$, and 
$\inrdist(R,I)\leq \ec(I,R)^{-\theta}$; let 
$C_r=\bigcup_{I\in \mathcal{D}(Q)_{\geq N_0}}
C_r(I)$. 
%for some $I\in \mathcal{D}(Q)_{\geq N_0}$.  

By Lemma \ref{zeroborder} with $\epsilon =\Big(\frac{\langle |f|\rangle_{Q}\langle |g|\rangle_{Q}\mu(Q)}{1+\| f\|_{L^{\infty }(\mu)}\| g\|_{L^{\infty }(\mu)}2^{N_0(n+3)}}\Big)^2$ and Lemma \ref{densityinL2-1}, 
%and the implicit limit in the equality at \eqref{orthopro}, 
we can choose $N_1>N_0$ so that for all $N>N_1$,
\begin{align}\label{Small}
\mu(C_N)^{\frac{1}{2}}\| f\|_{L^{\infty }(\mu)}\| g\|_{L^{\infty }(\mu)}
%\mu(Q)^{\frac{1}{2}}
%\sup_{i,k}F_\mu(J_{i,k})
%N_0
2^{N_0(n+3)}<\langle |f|\rangle_{Q}\langle |g|\rangle_{Q}\mu(Q),
\end{align}
and 
\begin{align*}
|\langle Tf,g\rangle |
&\lesssim\Big|\sum_{I,J\in {\mathcal D}(Q)_{\geq N}}
 \langle f,h_{I}\rangle 
 \langle g, h_{J}\rangle
 \langle Th_I,h_J\rangle \Big|+\Lambda_{\mathcal S}(f,g)
 =|\langle TP_{Q}^Nf,P_{Q}^Ng\rangle |+\Lambda_{\mathcal S}(f,g),
 %\lesssim \epsilon
 \end{align*}
 %uniformly in the finite family $\mathcal F_M$.
 where %$M_N=N-\log \ell(Q)$ and 
 $P_{Q}^N$ is the projection operator at the end of Definition \ref{projection}. 
%We note that $P_{M}^{\perp}P_{M_N}$ is not the zero operator since $N>N_0+M>\log \ell(Q)+M$ implies that $M_N>M$. 
%To simplify notation, we stop writing the conditions $I\in {\mathcal D}(Q)_{\geq N}$, $J\in {\mathcal D}(Q)_{\geq N}$. We will recover this notation whenever needed. 
%Previous inequality implies we can assume that $f$ and $g$ are constant on small cubes with zero mean. 

%Let $N>N_1$ fixed but to be chosen later (so that $\mu(C)<\epsilon )$. 

Fixed such $N>N_1$,  
we denote by $\partial \mathcal{D}(Q)$ the union of 
$\partial I$ for all $I\in \mathcal{D}(Q)_{\geq N}$. Then we
decompose the argument functions as 
$P_{Q}^Nf=f_1+f_{1,\partial}$, $P_{Q}^Ng=g_1+g_{1,\partial}$ where 
$f_{1,\partial}=(P_{Q}^Nf)\mathbbm{1}_{\partial D(Q)}$ and $g_{1,\partial}=(P_{Q}^Ng)\mathbbm{1}_{\partial D(Q)}$. 
% Since $f_1$ is supported on the interior of the cubes 
% $I\in \mathcal D(Q)_{\geq N}$ we have that 
% $\langle f_1, \psi_I\rangle = \langle f_1, \psi_{\tilde I}\rangle$ for any cubes 
% $I\in \mathcal D $ and $I\in \tilde{\mathcal D}$.
% Then
% $P_Mf_1=\tilde{P}_Mf_1$,  $P_M^{\perp}f_1=\tilde{P}_M^{\perp}f_1$ and similar for $g_1$. 
With this, 
\begin{align}\label{fdecom}
\langle TP_{Q}^Nf,P_{Q}^Ng\rangle 
&=\langle TP_{Q}^Nf,g_1\rangle 
+\langle Tf_1,g_{1,\partial}\rangle 
+\langle Tf_{1,\partial},g_{1,\partial}\rangle 
\end{align}
%We can assume there exists a dyadic cube $Q\in \mathcal D$ %such that $\supp f\cup \supp g\subset Q$. 

We note that 
$$
g_1=\sum_{J\in {\mathcal D}(Q)_{\geq N}}
 \langle g,h_{J}\rangle h_{\tilde J}
$$
where $\tilde J$ is the interior of $J$ and so, it is an open cube. Similar for $f_1$. 
%Therefore, when we deal with 
%$$
%\langle TP_{M_N}f,g_1\rangle
%=\sum_{I\in {\mathcal D}(Q)_{\geq N}}\sum_{J\in {\mathcal D}(Q)_{\geq N}}
% \langle f,h_{I}\rangle 
% \langle g, h_{J}\rangle
% \langle Th_I,h_{\tilde J}\rangle 
%$$
%we have that the cubes $\tilde J\in \tilde{\mathcal D}$ are all open. Similarly, when we later deal with 
%$$
%\langle Tf_1,g_{1,\partial}\rangle 
%=\sum_{I\in {\mathcal D}(Q)_{\geq N}}\sum_{J\in {\mathcal D}(Q)_{\geq N}}
% \langle f,h_{I}\rangle 
% \langle g, h_{J}\rangle
% \langle Th_{\tilde I},h_{J}\rangle 
%$$
%we will have that $\tilde I\in \tilde{\mathcal D}$ are all open cubes. \

However, we will start our work without reflecting this distinction in the notation since it is only useful at the end of the argument. 
%Instead we will highlight the difference later. 
%Although 
Namely, 
the work to prove \eqref{estsquare} starts with the first term $\langle TP_{Q}^Nf,g_1\rangle$. But since the same reasoning will also work for the second term 
$\langle Tf_1,g_{1,\partial}\rangle $,
we write each term simply by 
$\langle Tf,g\rangle$.  
%we keep denoting the argument functions $f$ and $g$
We will make distinctions only at the end of the proof. We hope this license will not cause any confusion.  

% Given $\mathcal{S}\subset \mathcal{D}$, 
% we will perform the following decomposition:
% \begin{align*}
% |\langle P_{\mathcal{S}}(Tf),g\rangle | & \lesssim 
% \sum_{e\in \mathbb Z}\sum_{m\in \mathbb N}2^{-|e|(\frac{d}{2}+\delta )}m^{-(d+\delta )}
% \sum_{\substack{J\in \mathcal S}}\sum_{I\in J_{e,m}}
% F(I,J) 
% |\langle f,\psi_{I}\rangle | |\langle g,\psi_{J}\rangle |
% \end{align*}

For fixed $e\in \mathbb Z$, $m\in \mathbb N$, $m>1$,  and every given cube $J\in \mathcal D$,
we define the family
$$
J_{e,m}=\{ I\in {\mathcal D}(Q):\ell(I)=2^{e}\ell(J), m-1\leq \rdist(\widehat{I},\widehat{J}\,)< m \} .
$$
For $m=1$ and $1\leq k\leq 2^{-\min(e,0)}$, we also define 
$$
J_{e,1, k}=J_{e,1}\cap \{ I\in {\mathcal D}(Q):k-1\leq \inrdist(\widehat{I},\widehat{J}\, )< k \} .
$$
The cardinality of 
$J_{e,m}$ is comparable to $2^{-\min(e,0)n}nm^{n-1}$, while the cardinality of 
$J_{e,1,k}$ is comparable to $n(2^{-\min(e,0)}-k)^{(n-1)
\frac{\min(e,0)}{e}}$.
By symmetry, 
we have $I\in J_{e,m}$ if and only if $J\in I_{-e,m}$
and, similarly,  $I\in J_{e,1,k}$ if and only if $J\in I_{-e,1,k}$.

%Let $\theta \in (0,1)$ such that $\theta >\frac{\alpha }{\alpha +\delta}$.
%Given $\mathcal{S}\subset \mathcal{D}$, 
%we will perform the following decomposition:
We now arrange the collections of cubes $I, J$ in previous sums in five different groups according to the positions and sizes of $I$ and $J$: distant cubes, close cubes, nested cubes, paraproducts, and essentially attached cubes.
We use the following notation: 
\begin{itemize}
\item distant: $I\! \Leftrightarrow \!J$, if $\rdist(I,J)\geq 1$, 
\item close: $I\sim J$, if  $0< \rdist(I,J)<1$, $\lambda(\widehat{I},\widehat{J}\, )>1$, 
\item nested: $I\asymp J$, if  $\rdist(I,J)=0$,  $\lambda(\widehat{I},\widehat{J}\, )>1$, 
\item essentially attached: $I\parallel J$, if $0\leq  \rdist(I,J)< 1$, $\lambda(\widehat{I},\widehat{J}\, )\leq 1$.
\end{itemize}
We note that $I\asymp J$ implies that either $\widehat{J}\subseteq \widehat{I}$ or $\widehat{I}\subseteq \widehat{J}$. Then
\begin{align*}
%\label{moduloin}
%\nonumber
%|\langle P_{\mathcal{S}}(Tf),g\rangle |=
\langle Tf,g\rangle 
&=\Big(\sum_{I \Leftrightarrow J}
%\langle f,h_{I}\rangle 
%\langle g, h_{J}\rangle
%\langle Th_I,h_J\rangle 
+\sum_{I\sim J}\Big)\langle f,h_{I}\rangle 
\langle g, h_{J}\rangle
\langle Th_I,h_J\rangle 
\\
&+%\sum_{S\in \mathcal S}
\sum_{\substack{I\asymp J\\ \widehat{J}\subseteq \widehat{I}\subseteq Q}}
\langle f,h_{I}\rangle 
\langle g, h_{J}\rangle
\langle T(h_I-h_{I,\widehat{J}}^{Q}),h_J\rangle 
\\
&+%\sum_{S\in \mathcal S}
\sum_{\substack{I\asymp  J\\ \widehat{I}\subset \widehat{J}\subseteq Q}}
\langle f,h_{I}\rangle 
\langle g, h_{J}\rangle
\langle Th_I,h_J-h_{J,\widehat{I}}^{Q}\rangle 
\\
&
+\Pi_1(f,g)
+\Pi_2(f,g)
+\sum_{I\parallel J}
\langle f,h_{I}\rangle 
\langle g, h_{J}\rangle
\langle Th_I,h_J\rangle
%\\
%&=\sum_{e\in \mathbb Z}\sum_{m\geq 2}
%\sum_{J}
%\sum_{I\in J_{e,m}}
%\langle f,h_{I}\rangle 
%\langle g, h_{J}\rangle
%\langle Th_I,h_J\rangle 
%\\
%&
%+\sum_{e\in \mathbb Z}
%\sum_{k=2^{\theta |e|}+1}^{2^{|e|}}
%\sum_{J}
%\sum_{\substack{I\in J_{e,1,k}\\ \dist(\widehat{I},\widehat{J}\,)>0}}
%\langle f,h_{I}\rangle 
%\langle g, h_{J}\rangle
%\langle Th_I,h_J\rangle 
%\\
%&+\sum_{S\in \mathcal S}\sum_{e\geq 0}\sum_{k=2^{\theta |e|}+1}^{2^{|e|-2}}
%\sum_{I}
%\sum_{\substack{\widehat{J}\subset \widehat{I}\\ J\in I_{e,1,k}\\J\equiv S}}
%\langle f,h_{I}\rangle 
%\langle g, h_{J}\rangle
%\langle T(h_I-h_{I,\widehat{J}}^{S}),h_J\rangle 
%\\
%&+\sum_{S\in \mathcal S}
%\sum_{e<0}\sum_{k=2^{\theta |e|}+1}^{2^{|e|-2}}
%\sum_{J}
%\sum_{\substack{\widehat{I}\subset \widehat{J}\\ I\in J_{-e,1,k}
%\\I\approx S}}
%\langle f,h_{I}\rangle 
%\langle g, h_{J}\rangle
%\langle Th_I,h_J-h_{J,\widehat{I}}^{S}\rangle 
%\\
%&
%+\Pi_1(f,g)
%+\Pi_2(f,g)
%\\
%&+\sum_{e\in \mathbb Z}\sum_{k=1}^{2^{\theta |e|}}
%\Big(\sum_{I}
%\sum_{J\in I_{e,1,k}}
%+\sum_{J}
%\sum_{I\in J_{-e,1,k}}
%\Big)
%\langle f,h_{I}\rangle 
%\langle g, h_{J}\rangle
%\langle Th_I,h_J\rangle
\\
&=D+C+N_1+N_2+P_1+P_2+A.
\end{align*}
We aim to control each of the terms $D$, $C$, $N_1$, $N_2$ by a different bilinear form acting on square functions, which will later be estimated by a sub-bilinear sparse form. % acting on the modulus of the argument functions. 
The remaining terms, $P_1$, $P_2$, and $A$, will be directly estimated by a sparse form, without square functions involved. 

%\begin{definition}
% Given $\mathcal{S}\subset \mathcal{D}$, a sequence $(\varepsilon_{I,J})$, 
% $e\in \mathbb Z$,  $m,k\in \mathbb N$,
% and $i\in \{1,\ldots, 8\}$, 
% we define the bilinear form
% \begin{align*}
% B_i(f,g)=
% B_{\mathcal{S},e,i}(f,g)=\sum_{m,k}a_{m,k}\sum_{\substack{J\in \mathcal{S}}}\sum_{I\in J_{e,m,k}}
% \varepsilon^i_{I,J}
% |\langle f, h_{I}\rangle | |\langle g, h_{J}\rangle |
% \end{align*}

{\bf Distant cubes.} We start by estimating the term $D$.  
By \eqref{bump1}, and the fact that 
$\ec{(I,J)}=2^{-e}$ and 
$1+\rdist(I,J)\approx m$, 
we have 
\begin{align*}
|\langle Th_I,h_J\rangle |
%=|\langle T_{b}(h_I^{b_{1},i}),h_J^{b_{2},j}\rangle |
&\lesssim \frac{2^{-|e|\delta}}{m^{\alpha +\delta}}
\frac{\mu(I)^{\frac{1}{2}}
\mu(J)^{\frac{1}{2}}}{\ell(I\smlor J)^\alpha }
F(I,J).
\end{align*}

%We remind that for $\vec{m}\in \mathbb Z^n$ we had defined 
% $$
% J_{e,\vec{m}}=\{ I\in {\mathcal D}(Q):\ell(I)=2^{e}\ell(J), |\vrdist(\widehat{I},\widehat{J})-\vec{m}|< 1 \} .
% $$
% The cardinality of $J_{e,\vec{m}}$ is $n2^{-\min(e,0)n}$ 
% and $I\in J_{e,\vec{m}}$ if and only if $J\in I_{-e,\vec{m}}$. Moreover, 
% $J_{e,m}$ is the disjoint union of at most  
% $2n m^{n-1}$ sets $J_{e,\vec{m}}$. 

By symmetry, we assume that $\ell(J)\leq \ell(I)$, that is, $e\geq 0$, and $I\smlor J=I$. Then the corresponding part of $D$, which we denote by $D^+$, can be estimated as follows: 
\begin{align*}%\label{moduloin2}
|D^+|
&
=\Big|\sum_{\substack{I \Leftrightarrow J\\ \ell(J)\leq \ell(I)}}
\langle f,h_{I}\rangle 
\langle g, h_{J}\rangle
\langle Th_I,h_J\rangle \Big|
\leq \sum_{e\geq 0}\sum_{m\geq 2}
\sum_{J}
\sum_{I\in J_{e,m}}
|\langle f,h_{I}\rangle |
|\langle g, h_{J}\rangle |
|\langle Th_I,h_J\rangle |
\\
&\lesssim
\sum_{e\geq 0}2^{-e\delta}
\sum_{\substack{m\in \mathbb Z\\ m\geq 2}}\frac{1}{m^{\alpha +\delta}}
\sum_{J}\sum_{I\in J_{e,\vec{m}}}
\frac{\mu(I)^{\frac{1}{2}}
\mu(J)^{\frac{1}{2}}}{\ell(I)^\alpha }
|\langle f,h_{I}\rangle | |\langle g,h_{J}\rangle |
F(I,J)
\\
&=
\sum_{e\geq 0}2^{-e\delta}
%\sum_{\substack{m\in \mathbb Z\\ m\geq 2}}\frac{1}{m^{\alpha +\delta}}
B_{1,e}(f,g). 
\end{align*}

For each pair $(I,J)$ %let $\lambda_{I,J}=2^N\ell(Q)$ and 
let $J_{I}\in \mathcal D(J)_{\geq N}$ (for example $J_I\in \child(J)$). %or $J_{I}\subset J$). 
If $\mu(J_{I})\leq \frac{\mu(I)\mu(J)}{\mu(Q)}$, we define $J'_{I}=J_{I}$. 
If $\mu(J_{I})> \frac{\mu(I)\mu(J)}{\mu(Q)}=a>0$, by Lemma \ref{subsets}, there is a subset $J_I'\subset J_I$ consisting of a union of a finite collection of dyadic cubes such that $a/2<\mu(J_I')\leq a$.  

We now denote $\tilde \epsilon =\min\{\mu(K): K\in \mathcal D(Q)_{\geq N}, \mu(K)>0\}>0$. 
Since $1\leq \tilde \epsilon+\tilde \epsilon^{-1}$, we have $\epsilon=\frac{1}{(1+\tilde \epsilon^{-2})(\mu(Q)+1)}\leq \frac{1}{1+\tilde \epsilon^{-2}}\leq \tilde \epsilon$.  
With this, $\mu(J_I)\geq \tilde \epsilon \geq \epsilon $ and so
$\frac{\mu(J_I)}{\mu(J_I)+\epsilon }\geq \frac{1}{2}$. 

If $\mu(J_I)\leq a$, we have $J_I'=J_I$ and so, 
$2\int\frac{\mathbbm{1}_{J_I'}(x)}{\mu(J_I')+\epsilon }d\mu(x)=2\frac{\mu(J_I)}{\mu(J_I)+\epsilon }\geq 1$.

On the other hand, if $\mu(J_I)>a$ we have by construction 
$0<a/2<\mu(J_I')\leq a$ and so, $4\int\frac{\mathbbm{1}_{J_I'}(x)}{\mu(J_I')+\epsilon}d\mu(x)= 4\frac{\mu(J_I')}{\mu(J_I')+\epsilon}\geq 2 \frac{a}{a+\epsilon}\geq 1$. For the last inequality, we reason as follows. 
Since $I,J\in \mathcal D(Q)_{\geq N}$, we have $\mu(I)\geq \tilde \epsilon$ and $\mu(J)\geq \tilde \epsilon$.  
%Moreover, $\lambda_{I,J}I\subset Q$. 
Then 
%, since $1\leq \tilde \epsilon+\tilde \epsilon^{-1}$, 
$$
a=\frac{\mu(I)\mu(J)}{\mu(Q)}\geq \frac{\tilde \epsilon^2}{\mu(Q)+1}\geq \frac{1}{(1+\tilde \epsilon^{-2})(\mu(Q)+1)}=\epsilon .
$$
Therefore, $1\geq \epsilon a^{-1}$  and so,
$$
\frac{a}{a+\epsilon}=\frac{1}{1+\epsilon a^{-1}}\geq \frac{1}{2}.
$$

%Since $J\subset \widehat{I}$ and thus, also $\mu(J)\geq \tilde \epsilon$.
To simplify notation, we write $\mu_{\epsilon}(K)=\mu(K)+\epsilon$. 
With this, we have
\begin{align*}
B_{1,e}(f,g)&\leq 4\sum_{\substack{m\in \mathbb Z\\ m\geq 2}}\frac{1}{m^{\alpha +\delta}}
\int  
\sum_{J\in \mathcal D}\sum_{I\in J_{e,m}}
F(I,J)
\frac{\mu(I)^{\frac{1}{2}}\mu(J)^{\frac{1}{2}}}{\ell(I)^\alpha }
|\langle f,h_{I}\rangle ||\langle g,h_{J}\rangle | 
\frac{\mathbbm{1}_{J_I'}(x)}{\mu_{\epsilon}(J_I')}d\mu(x)
\\
&
\lesssim 
\int \Big( \sum_{\substack{m\in \mathbb Z\\ m\geq 2}}\frac{1}{m^{\alpha +\delta}}\sum_{I\in \mathcal D}
a_{I}
|\langle f,h_{I}\rangle |^2
\frac{1}{\ell(I)^{\alpha}}
\sum_{J \in I_{-e,m}}\frac{\mu(J)}{\mu_{\epsilon}(J_I')}\mathbbm{1}_{J_I'}(x)
\Big)^{\frac{1}{2}}
\\
&
\hskip 30pt \Big( \sum_{\substack{m\in \mathbb Z\\ m\geq 2}}\frac{1}{m^{\alpha +\delta}}\sum_{J\in \mathcal D}
b_{J}
|\langle g,h_{J}\rangle |^2 
\frac{1}{2^{e\alpha}\ell(J)^{\alpha}}\sum_{I \in J_{e,m}}\frac{\mu(I)}{\mu_{\epsilon}(J_I')}\mathbbm{1}_{J_I'}(x)
\Big)^{\frac{1}{2}}d\mu(x)
\\
&
=\int  S_{1}^1f(x) S_{1}^{-1}g(x)d\mu(x), 
\end{align*}
where we define $\displaystyle{a_{I}=
\sup_{J\in  I_{-e,m}}F(I,J)}$ and 
$\displaystyle{b_{J}=
\sup_{I\in  J_{e,m}}F(I,J)}$.

With this, 
\begin{align*}%\label{moduloin2}
|D|
&\lesssim 
\sum_{\substack{e\in \mathbb Z\\ e\geq 0}}2^{-e\delta}
%\sum_{\substack{m\in \mathbb Z\\ m\geq 2}}\frac{1}{m^{\alpha +\delta-\frac{n-1}{2}}}
\langle S_1^1f, S_1^{-1}g\rangle
+\sum_{\substack{e\in \mathbb Z\\ e<0}}2^{e\delta}
%\sum_{\substack{m\in \mathbb Z\\ m\geq 2}}\frac{1}{m^{\alpha +\delta-\frac{n-1}{2}}}
\langle S_1^{-1}f, S_1^1g\rangle. 
\end{align*}

\vskip10pt
{\bf Close cubes.} To estimate the term $C$ we note that $m=1$,  $\ec(I,J)=2^{-|e|}$,   
$1+\inrdist(I,J)\approx k$ and so, the condition $1<\lambda(I,J)=\ec(I,J)^{\theta}(1+\inrdist(I,J))$ implies $2^{|e|\theta}\leq k\leq 2^{|e|}$. Then
by \eqref{bump2}
\begin{align*}
|\langle Th_I,h_J\rangle |
&\lesssim \frac{1}{k^{\alpha +\delta}}
\frac{\mu(I)^{\frac{1}{2}}
\mu(J)^{\frac{1}{2}}}{\ell(I\smland J)^\alpha }
F(I,J). 
\end{align*}
As before, we assume by symmetry that $\ell(J)\leq \ell(I)$, namely $e\geq 0$, and $I\smland J=J$. This way $C^+$, the corresponding part of $C$, can be bounded as follows:
\begin{align}\label{moduloin2-2}
\nonumber
|C^+|
&=\Big|\hspace{-.2cm}\sum_{\substack{I\sim J\\ \ell(J)\leq \ell(I)}}\langle f,h_{I}\rangle 
\langle g, h_{J}\rangle
\langle Th_I,h_J\rangle \Big|
\\
\nonumber
&
\leq \sum_{e\geq 0}
\sum_{k=2^{\theta e}}^{2^{e}}
\sum_{J}\hspace{-.1cm}
\sum_{I\in J_{e,1,k}
%\substack{I\in J_{e,1,k}\\ \dist(\widehat{I},\widehat{J}\,)>0}
}\hspace{-.1cm}
|\langle f,h_{I}\rangle |
|\langle g, h_{J}\rangle |
|\langle Th_I,h_J\rangle |
\\
&\lesssim
\sum_{e\geq 0}\sum_{k=2^{e\theta}}^{2^{e}}
\frac{1}{k^{\alpha +\delta}}
\sum_{J}\hspace{-.2cm}\sum_{I\in J_{e,1,k}}\hspace{-.2cm}
\frac{\mu(I)^{\frac{1}{2}}
\mu(J)^{\frac{1}{2}}}{\ell(J)^\alpha }
|\langle f,h_{I}\rangle | |\langle g,h_{J}\rangle |
F(I,J)
=\sum_{e\geq 0}
B_{2,e}(f,g).
\end{align}

We can assume that all cubes $I\in \mathcal D$ in the definition of $B_{2,e}(f,g)$ satisfy 
$\mu(J)\neq 0$, since otherwise their corresponding terms are zero and they do not need to be considered. 
Now, by defining $a_I$ and $b_J$ as before and using that $\ell(I)=2^{e}\ell(J)$, we have 
\begin{align*}
B_{2,e}(f,g)
&=\int 
\sum_{k=2^{e\theta}}^{2^{e}}
\frac{1}{k^{\alpha +\delta}}
\sum_{J\in \mathcal D}\sum_{I\in J_{e,1,k}}
F(I,J)
\frac{\mu(I)^{\frac{1}{2}}\mu(J)^{\frac{1}{2}}}
{\ell(I)^\alpha 2^{-e\alpha}}
|\langle f,h_{I}\rangle | |\langle g,h_{J}\rangle |
\frac{\mathbbm{1}_{J}(x)}{\mu(J)}d\mu(x)
\\
&\leq \int 
\Big( \sum_{I\in \mathcal D}
a_{I}
|\langle f,h_{I}\rangle |^2
\frac{\mu(I)2^{2e\alpha }}{\ell(I)^{2\alpha}}
\sum_{k=2^{e\theta}}^{2^{e}}
\frac{1}{k^{\alpha +\delta}}
\sum_{J \in I_{-e,1,k}}
\mathbbm{1}_{J}(x)\Big)^{\frac{1}{2}}
\\
&
\hskip30pt 
\Big( \sum_{J\in \mathcal D}
b_{J}
|\langle g,h_{J}\rangle |^2 \frac{\mathbbm{1}_{J}(x)}{\mu(J)}
\sum_{k=2^{e\theta}}^{2^{e}}
\frac{1}{k^{\alpha +\delta}}
\sum_{I \in J_{e,1,k}}
1
\Big)^{\frac{1}{2}}d\mu(x).
\end{align*}
For the first factor, we note that  
% Now we denote by $I_k\in \mathcal C$ the cube 
% such that $c(I_k)=c(I')$ and $\ell(I_k)=(2^{e}-k)\ell(J)\leq 2^{e}\ell(J)=\ell(I)$. With this, 
% $$
% I_{-e,1,k}\subset \{ t\in 3I : k\ell(J)<\dist(t,I)\leq (k+1)\ell(J)\}
% \subset I_k\setminus I_{k+1}.
% $$
%Now, 
the cubes $J$ in $\cup_{k=2^{e\theta}}^{2^{e}}I_{-e,1,k}$ are all pairwise disjoint and included in $(3I)\setminus I$. Then
% $$
% \sum_{J \in I_{-e,1,k}}\hspace{-.2cm}
% \mu(J)\lesssim \sum_{I'\in I^{\rm fr}}\mu(I_k\setminus I_{k+1}).
% $$
%\begin{align}\label{21k}
$$
\sum_{k=2^{e\theta}}^{2^{e}}
\frac{1}{k^{\alpha +\delta}}
\sum_{J \in I_{-e,1,k}}
\mathbbm{1}_{J}(x)
\leq 
\frac{1}{2^{e\theta(\alpha +\delta)}}
\sum_{k=2^{e\theta}}^{2^{e}}
\sum_{J \in I_{-e,1,k}}
\mathbbm{1}_{J}(x)
\leq \frac{1}{2^{e\theta(\alpha +\delta)}}
\mathbbm{1}_{3I}(x).
$$
%\end{align}
For the second factor, we show that we can consider that $k$ is actually determined by $J$.
This is because, for each $J$ there are only a quantity comparable to $3^n-1$ cubes $I$ such that $m=1$ and there 
is only one $k\geq 2^{e\theta}$
such that $J_{e,1,k}$ is not empty.  We denote such index $k$ by $k_{J}$. 
In addition, the cardinality of such $J_{e,1,k_J}$  
is comparable to $n2^{-\min(e,0)(n-1)\frac{\min(e,0)}{e}}=n\approx 1$.
Then
\begin{align}\label{21k}
\sum_{k=2^{e\theta}}^{2^{e}}
\frac{1}{k^{\alpha +\delta}}
\sum_{I \in J_{e,1,k}}1
\lesssim \frac{1}{k_J^{\alpha +\delta}}
\leq \frac{1}{2^{e\theta(\alpha +\delta)}}.
\end{align}
With both things we get
\begin{align*}
B_{2,e}(f,g)
&\leq \frac{2^{e\alpha}}{2^{e\theta (\alpha +\delta)}}\hspace{-.1cm}
\int \hspace{-.2cm}\Big( \sum_{I\in \mathcal D}
a_{I}
|\langle f,h_{I}\rangle |^2
\frac{\mu(I)}
{\ell(I)^{2\alpha}}
\mathbbm{1}_{3I}(x)\Big)^{\frac{1}{2}}
\Big( \sum_{J\in \mathcal D}
b_{J}
|\langle g,h_{J}\rangle |^2
\frac{\mathbbm{1}_{J}(x)}{\mu(J)}
\Big)^{\frac{1}{2}}d\mu(x)
\\
&
=2^{-e(\theta(\alpha +\delta)-\alpha )}\int  S_{2}^1f(x) S_{2}^{-1}g(x)d\mu(x).
\end{align*}

Since $\theta= \frac{\alpha}{\alpha +\delta/2}> \frac{\alpha}{\alpha +\delta}$, we have $\theta(\alpha +\delta)-\alpha =\frac{\alpha \delta}{2(\alpha +\delta/2)}=\frac{\theta \delta}{2}$ and so 
\begin{align*}
|C|
&\lesssim
\sum_{e\in \mathbb Z}
B_{2,e}(f,g)
\lesssim \sum_{e\geq 0}2^{-e\frac{\theta \delta}{2}}
\langle S_{2}^1f, S_{2}^{-1}g\rangle
+\sum_{e\leq 0}2^{e\frac{\theta \delta}{2}}
\langle S_{2}^{-1}f, S_{2}^1g\rangle . 
\end{align*}

{\bf Nested cubes}. By definition, the cubes in the term $N$ satisfy $\ell(J)\leq \ell(I)$. Then 
\begin{align*}
|N|&
=\Big|\sum_{S\in \mathcal S}\sum_{\substack{I\asymp  J\\ J\subset I\subset S}}
\langle f,h_{I}\rangle 
\langle g, h_{J}\rangle
\langle T(h_I-h_{I,\widehat{J}}^{S}),h_J\rangle \Big|
\\
&
\leq 
\sum_{S\in \mathcal S}\sum_{e\geq 0}\sum_{k=2^{\theta e}}^{2^{e}}
\sum_{I}
\sum_{\substack{J\in I_{e,1,k}\\\widehat{J}\subset \widehat{I}%\\J\equiv S
}}
|\langle f,h_{I}\rangle |
|\langle g, h_{J}\rangle |
|\langle T(h_I-h_{I,\widehat{J}}^{S}),h_J\rangle |.
\end{align*}
Now, by 
\eqref{bump3}, we have
\begin{align*}
|\langle T(h_{I}-{h}^{S}_{I,J}),h_{J}\rangle |
&\lesssim
(1+\inrdist(I,J))^{-\delta}
\sum_{R\in \{I,\widehat{I}\}}
\Big(\frac{\mu(R\cap J)}{\mu(R)}\Big)^{\frac{1}{2}}
F(I,J)
\\
\nonumber
&+(1+\inrdist(I,J))^{-(\alpha+\delta)}
\frac{\mu(I)^{\frac{1}{2}}\mu(J)^{\frac{1}{2}}}{\ell(
I\smland J)^{\alpha }}\mathbbm{1}_{\widehat{I}\setminus I}(c(\widehat{J}\,))F(I,J). 
\end{align*}
We can deal with the second term exactly as we did in the case of close cubes. So, we only need to work with the first term. 
Since this estimate does not depend on $S\in \mathcal S$, we can sum over all cubes $I,J\in D_i$ such that 
$I\asymp  J$ and $J\subset I$. 
Then the corresponding expression, which we denote by $N^{-}$ can be estimated as follows:
\begin{align*}
|N^{-}|
&\lesssim 
\sum_{e\geq 0}\sum_{k=2^{\theta e}}^{2^{e}}
\sum_{\substack{J\\ \widehat{J}\subset \widehat{I}}}
\sum_{I\in J_{e,1,k}}
|\langle f,h_{I}\rangle |
|\langle g,h_{J}\rangle |
\sum_{R\in \{I,\hat{I}\}}\Big(\frac{\mu(R\cap J)}{\mu(R)}\Big)^{\frac{1}{2}}k^{-\delta}
F(I,J) 
\\
&=\sum_{e\geq 0}B_3(f,g).
\end{align*}
Now we reason in a similar way we did for distant cubes. For each $I\in \mathcal D$ and each 
$J\in I_{-e,1,k}$, we define $J_I=J+\lambda \ell(I)e_i$ where $e_i=(0, \ldots, 0,1,0, \ldots, 0)$ and $\lambda \in \{-1, 0, 1\}$ are the only unit basis vector and scalar such that  $J_I\subset I$. We note that $\ell(J_I)=\ell(J)$ and that, if $J\subset I$, we have $J_I=J$.

If $\mu(J_I)\leq \mu(J)$, we define $J_I'=J_I$. Otherwise, 
from Lemma \ref{subsets}, for each $J_I$ with $\mu(J_I)>\mu(J)=a$ there is a subset $J_I'\subset J_I$ consisting of a union of a finite collection of dyadic cubes such that $\mu(J)/2<\mu(J_I')\leq \mu(J)$.  

We note that, once this process has been repeated for all cubes $I\in \mathcal D$, then for each cube in 
$J\in \mathcal D$ we find up to $2n$ disjoint sets $J_I'\subset I$ (one per each cube $I\in \mathcal D$ with 
$\dist(J,I)\approx k\ell(J)$) 
 such that $\diam(J_{I}')\leq \ell(J)$, 
$\dist(J_I',J)\approx 2^{e}\ell(J)$ and either $J_I'$ is a dyadic cube with $\mu(J_I')\leq \mu(J)$ or it is the union of a finite collection of dyadic cubes with $\mu(J)/2<\mu(J_I')\leq \mu(J)$. Therefore, we may consider that 
$J_I'$ is uniquely determined by $J$ and we can just write such set as $J'$. 

Finally, we denote $\epsilon =\min\{\mu(K): K\in \mathcal D(\hat{I}), \ell(K)=2^{-e}\ell(I), \mu(K)>0\}>0$. We then have $\mu(J_I)\geq \epsilon$ and so
$\frac{\mu(J_I)}{\mu(J_I)+\epsilon }\geq \frac{1}{2}$. With this, if $\mu(J_I)\leq \mu(J)$, we have $J_I'=J_I$ and so, 
$2\int\frac{\mathbbm{1}_{J_I'}(x)}{\mu(J_I')+\epsilon }d\mu(x)=2\frac{\mu(J_I)}{\mu(J_I)+\epsilon }\geq 1$. On the other hand, if $\mu(J_I)>\mu(J)$ we have by construction 
$0<\mu(J)/2<\mu(J_I')\leq \mu(J)$ and so, $4\int\frac{\mathbbm{1}_{J_I'}(x)}{\mu(J_I')+\epsilon}d\mu(x)= 4\frac{\mu(J_I')}{\mu(J_I')+\epsilon}\geq 2 \frac{\mu(J)}{\mu(J)+\epsilon}\geq 1$ since $J\subset \widehat{I}$ and thus, also $\mu(J)\geq \epsilon$.
To simplify notation, we write $\mu_{\epsilon}(K)=\mu(K)+\epsilon$. 
With all this, we now write
\begin{align}\label{decompnested}
\nonumber
B_3(f,g)
&\leq 4\hspace{-.1cm}\int  
\hspace{-.1cm}\sum_{k=2^{\theta e}}^{2^{e}}
\sum_{J}
\sum_{\substack{I\in J_{e,1,k}\\\widehat{J}\subset \widehat{I}}}
\hspace{-.1cm}
|\langle f,h_{I}\rangle |
|\langle g, h_{J}\rangle |
\hspace{-.1cm}
\sum_{R\in \{I,\hat{I}\}}\hspace{-.1cm}
\Big(\frac{\mu(R\cap J)}{\mu(R)}\Big)^{\frac{1}{2}}k^{-\delta}
F(I,J) \frac{\mathbbm{1}_{J_I'}(x)}{\mu_{\epsilon}(J_I')}d\mu(x) 
\\
&\lesssim \int 
\Big( \sum_{I\in {\mathcal D}}
a_{I}
|\langle f,h_{I}\rangle |^2
%\frac{\mathbbm{1}_{I}(x)}{\mu(I)}
\sum_{k=2^{\theta e}}^{2^{e}}k^{-\delta}
\sum_{R\in \{I,\hat{I}\}}
\sum_{J\in I_{-e,1,k}}
\frac{\mu(R\cap J)}{\mu(R)}
\frac{\mathbbm{1}_{J'_I}(x)}{\mu_{\epsilon}(J_I')}
\Big)^{\frac{1}{2}}
\\
&
\nonumber
\hskip30pt \Big( \sum_{J\in \mathcal D}b_{J}
|\langle g,h_{J}\rangle |^2
\sum_{k=2^{\theta e}}^{2^{e}}k^{-\delta}
\sum_{I\in J_{e,1,k}}\frac{\mathbbm{1}_{J'_I}(x)}{\mu_{\epsilon}(J'_I)}
%\hspace{-.1cm}  
%\sum_{k'=2^{\theta \min(e,0)}}^{2^{\min(e,0)}}
%\sum_{\tiny \begin{array}{c}I\! \in \! J_{e,k'}'\end{array}}{(k')}^{-2(\delta +\frac{n}{r_{1}'})}
\Big)^{\frac{1}{2}}d\mu(x).
\end{align}

For the first factor, we 
reason as follows: 
since $k\geq 2^{\theta e}$,  and $\mu(R)\geq \mu(I)$, we have
\begin{align*}
\sum_{k=2^{\theta e}}^{2^{e}}k^{-\delta}
&\sum_{R\in \{I,\hat{I}\}}
\sum_{J\in I_{-e,1,k}}
\frac{\mu(R\cap J)}{\mu(R)}
\frac{\mathbbm{1}_{J_I'}(x)}{\mu(J_I')}
\\
&\leq 2^{-\theta \delta e}\frac{1}{\mu(I)}
\sum_{R\in \{I,\hat{I}\}}
\sum_{k=2^{\theta e}}^{2^{e}}
\sum_{J\in I_{-e,1,k}}
\frac{\mu(R\cap J)}{\mu(J_I')}\mathbbm{1}_{J_I'}(x)
=2^{-\theta \delta e}\frac{1}{\mu(I)}\Phi_I(x).
\end{align*}
with the obvious definition for $\Phi_I$. 
Moreover, we note that $\mu(R\cap J)\neq 0$ only if 
$J\subset R\subset \widehat{I}$.

For the second factor in \eqref{decompnested}, 
we first note that for each cube $J\in \cup_{k=2^{\theta e}}^{2^{e}}I_{-e,1,k}$ with $J\subset I$ there are up to $2n$ cubes $I'\in \mathcal D$ such that $J_{I'}=J$.  Moreover, $\mu_{\epsilon}(J_I')\geq \mu(J_I')=\mu(J)$. 
Then 
we use the reasoning previous to the calculation in \eqref{21k} 
and the fact that $I$ is essentially determined by $J$ to eliminate the sums in $k$ and $I$:
\begin{align*}
%\sum_{J}a_{J,2}|\langle g,h_{J}\rangle |^2
\sum_{k=2^{\theta e}}^{2^{e}}k^{-\delta}
\sum_{I\in J_{e,1,k}}\frac{\mathbbm{1}_{J'_I}(x)}{\mu_{\epsilon}(J'_I)}
\lesssim 
%\sum_{J}a_{J,2}|\langle g,h_{J}\rangle |^2
k_J^{-\delta}
\frac{\mathbbm{1}_{J'}(x)}{\mu(J')}  %(?)
\end{align*}

With both things and $k_J\geq 2^{\theta e}\geq 2^{\theta \delta e}$, 
we have
\begin{align*}
B_3(f,g)
&\lesssim 
2^{-\theta \delta e}\int 
\Big( \sum_{I\in {\mathcal D}}
a_{I}
|\langle f,h_{I}\rangle |^2
\frac{1}{\mu(I)}\Phi_{I}(x)
\Big)^{\frac{1}{2}}
\Big( \sum_{J\in \mathcal D}b_{J}
|\langle g,h_{J}\rangle |^2
\frac{\mathbbm{1}_{J'}(x)}{\mu(J')}
\Big)^{\frac{1}{2}}d\mu(x)
\\
&=2^{-\theta \delta e}\int  S_{3}f(x)S_{3}^{-1}g(x)d\mu(x).
\end{align*}

Similar estimate holds for $N_2$ and so, 
\begin{align*}
N_1+N_2&\leq \sum_{e\in \mathbb Z}B_3(f,g)
\lesssim  \sum_{e\geq 0}2^{-\theta \delta e}\langle S_{3}^1f, S_{3}^{-1}g\rangle 
+\sum_{e\leq 0}2^{\theta \delta e}\langle S_{3}^{-1}f, S_{3}^1g\rangle . 
%\\
%&
%\lesssim \langle S_{0}f, S_{3}g\rangle +\langle S_{3}f, S_{0}g\rangle .(depends on e)
\end{align*}

We stop the proof of Proposition \ref{dualbysquare} here even though it is not completed. On the one hand, we still need to deal with two more terms:

{\bf Paraproducts (or the constant part of nested cubes)}.
The paraproducts $\Pi_i (f,g)$, are not estimated by square functions. Instead, we prove in the next section that they can be directly estimated by a sparse form without previous estimate by a square form. 

{\bf Essentially attached cubes.} The last term in \eqref{estsquare}, which is associated with close cubes with large eccentricity, is also directly estimated by a sparse form. This is done in the next section. 

In addition, we note that all the work done so far only involves the use of one dyadic grid and, as we will later see, only applies to the first two terms in \eqref{fdecom}: $\langle TP_{Q}^Nf,g_1\rangle $ and $\langle Tf_1,g_{1,\partial}\rangle $.
To deal with the third term in \eqref{fdecom}, namely, $\langle Tf_{1,\partial},g_{1,\partial}\rangle $, 
we need an iteration process that requires up to $k$ different grids. 
Only after that work, the proof is completed. 

\section{Paraproducts and essentially attached cubes.}

\subsection {Paraproducts}  
The following proposition shows how to estimate the paraproducts by a sparse form. 
%, which we directly estimate by a sparse form without previous estimate by a square form. 
\begin{proposition}\label{estimate4paraproducts} For $i\in \{ 1, \cdots, k\}$ and $j\in \{1,2\}$, let $\Pi_j^i (f,g)$ be the paraproducts  as given in Definition \ref{paraproducts}. Then 
$$
|\Pi_j^i (f,g)|\lesssim \Lambda_{S}(|f|,|g|). 
$$
\end{proposition}
\begin{proof}
By symmetry we only work with $\Pi_1^i (f,g)$, which we simply denote by $\Pi (f,g)$ and rewrite 
as follows: since each cube $J\in \mathcal D(Q)$ satisfies that $J\in \mathcal U(S)$ for some $S\in \mathcal S$, we have 
%Remember that given $S\in \mathcal S$, we defined 
% $\mathcal U(S)=\mathcal{T}_i\mathcal{D}(S)\setminus \bigcup_{S'\in \mathcal S(S)}\mathcal T_i\mathcal D(S')$, that is, the collection of cubes $I\in \mathcal{T}_i\mathcal{D}$ such that $S$ is the minimal cube in $\mathcal S$ with $I\subset S$. 
% %$J\subset S$ and $J\not\subset S'$ for any other selected cube $S'\in \mathcal S(S)$. 
% The collection  $(\mathcal U(S))_{S\in \mathcal S}$ is a partition of $\mathcal T_i \mathcal D(Q)$ (?). Then 
\begin{align}\label{para2}
\nonumber
%P_1=
\Pi(f,g)
&=%\sum_{S\in \mathcal S}
\sum_{I\in \mathcal D(Q)}
%\sum_{I\in \mathcal D(S')}
\sum_{\substack{J\in \mathcal D\\ \widehat{J}\subset \widehat{I}, 
%\inrdist(\widehat{I},\widehat{J}\,)>\lambda_\theta
\lambda(\widehat{I},\widehat{J}\, )>1}
}
\langle f, h_{I}\rangle \langle g, h_{J}\rangle \langle Th_{I}^{\widehat{J},Q},h_{J}\rangle
\\
&
=\sum_{S\in \mathcal S}
\sum_{J\in \tilde{\mathcal{U}}(S)\cap \mathcal D} \langle g,h_{J}\rangle
\Big\langle T\big( \hspace{-.5cm}
\sum_{\substack{I\in \mathcal{D}(Q)\\ %\ell(J_p)\leq \ell(I_p), 
\widehat{J}\subset \widehat{I}, 
\lambda(\widehat{I},\widehat{J}\, )>1}}
\langle f, h_{I}\rangle h_{I}^{\widehat{J}, Q}
%+\langle f, \tilde\psi_{3,I}\rangle \psi_{3, I}^{{\rm full}}
),h_{J}\Big\rangle .
%\\
%&
%=\sum_{S\in \mathcal S}\Pi_S ( f,g).
\end{align}
We have used that conditions $I\in \mathcal D(S)$, and $\widehat{J}\subset \widehat{I}$ with 
$\lambda(\widehat{I},\widehat{J}\, )>1$ imply that $J_j\subset \widehat{I}\subset S$ and so, 
$J\in {\mathcal{U}}(S)$. 

We can assume that the first sum in \eqref{para2} only contains terms for which $\mu(J)\neq 0$ since otherwise $h_J\equiv 0$. 

Moreover, $\widehat{J}\subset \widehat{I}$ implies  
$\ell(J)<\ell(I)$.
More importantly, fixed $J\in \mathcal D$, for each cube $I$ satisfying the condition $\widehat{J}\subset \widehat{I}$, all cubes $I'\in \child{(\widehat{I}\,)}$ also satisfy the same condition. In other words, for each cube $I$ in the sum, all its brothers are also in the sum. The same holds for the cubes $J$, but we will not make use of such property. 

With both facts, 
by \eqref{Deltaincoord2} in Lemma \ref{densityinL1},   
we get 
$$
\sum_{I\in \child(\hat{I})}
\langle f,h_{I}\rangle 
h_{I}^{\hat{J},Q}
%=\hat{\Delta}_{I_p}(f)
=\Big(\sum_{I\in \child(\hat{I})}E_{I}^{\hat{J},Q}f\Big)-E_{\widehat{I}}^{\hat{J},Q}f,
$$
where we remind that 
$E_{I}^{\widehat{J},Q}f=\langle f\rangle _{I}\mathbbm{1}_{I}(c(\widehat{J}\,))\mathbbm{1}_{Q}$.

The inner sum in \eqref{para2} takes place under the condition 
$\lambda(\widehat{I},\widehat{J}\,)>1$. 
Then, 
for $\widehat{J}\in \mathcal{D}(Q)$, 
let 
%$\lambda $ be the smallest integer such that 
%$\inrdist(\widehat{I},\widehat{J})>\lambda$. Let also 
$J_{\lambda}\in \mathcal{D}(Q)$ 
be the smallest cube such that 
$\widehat{J}\subset J_{\lambda}$ and $\lambda(\widehat{I},\widehat{J_\lambda }\,)>1$.  
If such cube does not exist, we then define 
$J_{\lambda}=\emptyset $. 

We now add and subtract the following term: 
\begin{align*}
G=
&\sum_{\substack{I\in \mathcal D(Q)\\ 
%\ell(J)\leq \ell(I)\\ 
\widehat{J}\subset \widehat{I},\lambda(\widehat{I},\widehat{J}\,)\leq 1}}
\langle f, h_{I}\rangle 
h_{I}^{\widehat{J}, Q}, 
\end{align*}
which we call a partial paraproduct, 
to obtain 
\begin{align*}
\sum_{\substack{I\in \mathcal D(Q)\\ 
%\ell(J)\leq \ell(I)\\ 
\widehat{J}\subset \widehat{I}, \lambda(\widehat{I},\widehat{J}\,)>1}}
\langle f, h_{I}\rangle 
h_{I}^{\widehat{J}, Q}
=\sum_{\substack{I\in \mathcal D(Q)\\ 
%\ell(J)\leq \ell(I)\\ 
\widehat{J}\subseteq \widehat{I}}}
\langle f, h_{I}\rangle 
h_{I}^{\widehat{J}, Q}-G. 
\end{align*}
The second term $-G$, together with a symmetric expression containing cubes such that $\widehat{I}\subset \widehat{J}$ will be estimated in the next subsection and so, we now focus on the first term. 

By summing a telescopic sum, we have for $J\in {\mathcal D}(Q)$ fixed and $M\in \mathbb N$ such that $2^M>\ell(Q)$, 
\begin{align*}
\sum_{\substack{I\in \mathcal D(Q)\\ 
%\ell(J)\leq \ell(I)\\ 
\widehat{J}\subseteq \widehat{I}}}
\langle f, h_{I}\rangle 
h_{I}^{\widehat{J}, Q}
%+\langle f, \tilde{\psi}_{3,I}\rangle
%\psi_{3,I}^{{\rm full}}
&=\sum_{\substack{\widehat{I}\in \mathcal D(Q)\\ 
%\ell(J)\leq \ell(I)\\ 
\widehat{J}\subseteq \widehat{I}}}
\sum_{I'\in \child(\hat{I})}
\langle f, h_{I'}\rangle 
h_{I'}^{\widehat{J}, Q}
=\sum_{\substack{\widehat{I}\in \mathcal D(Q)\\ 
%\ell(J)\leq \ell(I)\\ 
\widehat{J}\subseteq \widehat{I}}}
\Big(\sum_{I\in \child(\hat{I})}\bar E_{I}f\Big)-\bar E_{\widehat{I}}f
%\hat\Delta_{I_p}(f)
\\
&=\sum_{R\in \child(\widehat{J}\, )}
\bar E_{R}f
%\hat E_{J}f
-\bar E_{Q}f
%-\sum_{R\in J^{\rm fr}\setminus \{J\}}\hat E_{R_J}f
%\\
%&=\sum_{R\in \child(J^{\lambda_\theta})\cup \{Q\}}
%\alpha_R\hat E_{R_J}f
=\sum_{R\in \child(\widehat{J}\,)}
\alpha_{R}\langle f\rangle _{R}
%\chi_{R_J}(c_{J_{p}})
\mathbbm{1}_{Q},
%\chi_{I}
%-\langle f\rangle _{(Q_{0})_{p}}
%\chi_{(Q_0)_{p}}(c_{J_{p}})%\chi_{(Q_0)_p}
%\chi_{S}
\end{align*}
where $\alpha_{R}=1$ if $R\in \child(\widehat{J}\,)$ and 
$\alpha_{R}=-1$ if $R=Q$. 
% where $Q'\in \mathcal D$ such that 
% $\sup f\subset Q'$ and so, 
% $\langle f\rangle_{Q'}=0$. 
%and 
%where 
%$\mathcal{J}=\child(J^{\lambda_\theta})$. 
%$\alpha_R=1$ for $R\neq Q$ and 
%$\alpha_{Q}=-1$.  
% $$
% +\sum_{k=1}^{\log \frac{\ell(Q_{0})}{\ell(J)}}\sum_{%\tiny \begin{array}{c}
% I\in J_{k,3}}
% %\\ \ec(I,J)=k, \rdist(I,J)=3\end{array}}
% \frac{\langle f\rangle _{I}}{\langle b^{1}_{I^{a}}\rangle _{I}}b^{1}_{I^{a}}\chi_{I}(c(J_{p})) (\rm maybe not needed????)
% $$
%We denote $\mathcal{J}=J^{\rm fr}\cup \{Q\}$.
The cardinality of $\child(\widehat{J}\,)$ is $2^n$ and so, 
we can enumerate the family in a uniform way accordingly with their position inside $\widehat{J}\,$: 
%for all $J_p\in \mathcal{D}_M^c(Q)$: 
$\child(\widehat{J}\,)=\{ J_j\}_{j=0}^{2^n}$, with $J_0=Q$. We denote by $\alpha_j\in \{ 1,-1\}$ the corresponding coefficient 
$\alpha_{R_j}$.
%with 
%$J_0=J$ and 
%$J_{2^n+1}=Q$.
With this
\begin{align*}
\Pi (f,g)
&=\sum_{J\in \mathcal D(Q)} \langle g,h_{J}\rangle
\langle T(\sum_{R\in \child(J_{\lambda})\cup Q}
\alpha_{R}\bar E_{R}f)
,h_{J}\rangle 
%-\langle T(\hat E_{Q_0}f),\psi_{J}\rangle \Big)
\\
&=\sum_{j=0}^{2^n}\alpha_{j}
\sum_{J\in \mathcal D(Q)} 
\langle f\rangle_{J_j}
\langle g,h_{J}\rangle 
\langle T\mathbbm{1}_{Q},h_{J}\rangle 
% \\
% &=\sum_{k=0,1}(-1)^k\sum_{j=1}^{2^n}
% \sum_{J\in \mathcal D_M^c(Q)} 
% \langle f\rangle _{J_j}
% \langle g\rangle_{J^k}
% \mu(J)^{\frac{1}{2}}\langle T\chi_{Q},\psi_{J}\rangle 
%=\sum_{j=0}^{2^n}\alpha_{j}
%\langle T\mathbbm{1}_{S}, 
%\sum_{J\in \mathcal D(S)}\langle f\rangle_{J_j}
%\langle g,h_{J}\rangle h_{J}\rangle 
=\sum_{j=0}^{2^n}\alpha_{j}\Pi_j (f,g). 
\end{align*}
%where we used the definition of the Haar wavelets 
%in the last equality with 
%$J^0=J$ and $J^1=\widehat{J}$. 

Then the estimate for $\Pi (f,g)$ follows once we uniformly estimate $\Pi_j (f,g)$.
The case when $j=0$, namely, $J_j=Q$, can be treated by summing a telescoping sum and using arguments similar to the ones used in Lemma \ref{densityinL2-1}. That is, by using that $Q$ can be chosen so that the average $\langle f\rangle_{Q}$ is as small as needed. Then we focus on the terms 
$\Pi_j (f,g)$ for $j\neq0$.
% and deal with the case $j=0$ at the very end. 
 
We remind that given $S\in \mathcal S^1$, we define  
 $\mathcal U^1(S)=\mathcal{D}_i(S)\setminus \bigcup_{S'\in \mathcal S^1(S)}\mathcal D_i(S')$, where  $\mathcal S^1(S)$ is the collection of cubes that belong to $\mathcal S^1$ and are included in $S$. In other words, $\mathcal U^1(S)$ is the collection of cubes $I\in \mathcal{D}_i$ such that $S$ is the minimal cube in $\mathcal S^1$ with $I\subset S$. 
 %$J\subset S$ and $J\not\subset S'$ for any other selected cube $S'\in \mathcal S(S)$. 
 The collection  $(\mathcal U^1(S)\cap Q)_{S\in \mathcal S}$ is a partition of $\mathcal D_i(Q)$. %(?). 
 Then 
%Since $\{ \mathcal U(S)\cap Q\}_{S\in \mathcal S}$ form a partition of $\mathcal D(Q)$, 
we have 
$$
\Pi_j (f,g)
%=\langle T\mathbbm{1}_{Q}, 
%\sum_{J\in \mathcal D(Q)}\langle f\rangle_{J_j}
%\langle g,h_{J}\rangle h_{J}\rangle 
=\sum_{S\in \mathcal S^1}\sum_{J\in \mathcal U^1(S)} 
\langle f\rangle_{J_j}
\langle g,h_{J}\rangle 
\langle T\mathbbm{1}_{S'},h_{J}\rangle
=\sum_{S\in \mathcal S^1}\Pi_{j,S} (f,g). 
$$ 
%We further decompose $\mathcal U^1(S)$ as follows. Let $\mathcal S'$ the collection of cubes in $\mathcal S$
%such that $\ell(S')=\ell(S)$ and $c(S)$ is either a vertex of $S$ or the midpoint of each $k$-dimensional face of $S$ for all $k\in \{1, \ldots, n\}$. Then the cardinality of $\mathcal S'$ is 
%exactly $3^n+1$, being $S$ itself one of such cubes. It is clear that 
%each cube $J\in \mathcal U^1(S)$ is included in exactly one cube $S'\in \mathcal S'$ satisfying  
%$J\subset 2^{-1}S'$. 
%Than by considering $3^n+1$ different families, $J\in \mathcal U^1(S')$ with $S'\in \mathcal S'$, we can assume without loss of generality that 
%all cubes $J\in \mathcal U^1(S)$ satisfy $J\subset 2^{-1}S$.

We remind that  for all $J\in \mathcal U^1(S)$ we have $J\subset 2^{-1}S$. We now fix $S$ and estimate $\Pi_{j,S} (f,g)$
by decomposing the operator in two parts:
\begin{align}\label{2terms}
\Pi_{j,S} (f,g)
&=\sum_{J\in \mathcal U^1(S)} 
\langle f\rangle_{J_j}
\langle g,h_{J}\rangle 
\langle T\mathbbm{1}_{S},h_{J}\rangle
+\sum_{J\in \mathcal U^1(S)} 
\langle f\rangle_{J_j}
\langle g,h_{J}\rangle 
\langle T\mathbbm{1}_{Q\setminus S},h_{J}\rangle
\\
&
\nonumber
=T_1+T_2, 
\end{align}
and we estimate each term separately. 

To bound $T_1$ we note that 
%some properties. 
%First, 
%by the stopping condition \ref{selection}
%we have 
%$|\langle g\rangle_{J}|\leq C|\langle g\rangle_{S}|$ for all $J\in \mathcal U^1(S)$. Furthermore, when $J$ varies in $\mathcal U^1(S)$, the 
%cubes $J_j$ vary in a finite sequence of collections of cubes $\mathcal U^1(S_k)$, for $k\in \{ r_{j,1}, \ldots, r_{j,m_j}\}$. We denote by 
%$S'=S_{r_{j,m_j}}$ the smallest collection, which is completely determined by $S$ and $j$. Then
%if say $J_j \in  \mathcal U^1(S_k)$, we have by the stopping condition \ref{selection} again and the fact that $C>1$, 
%$$
%|\langle f\rangle_{J_j}|\leq C|\langle f\rangle_{S_{k}}|\leq |\langle f\rangle_{S_{k-1}}|
%\leq C^{-1}|\langle f\rangle_{S_{k-2}}|
%\leq  C^{-(k-r_{j,m_j}-1)}|\langle f\rangle_{S_{r_{j,m_j}}}|\leq 
%|\langle f\rangle_{S'}|\leq |\langle f\rangle_{S}|
%$$ 
%where 
%$S_{J_j}'=J_j^a$. Note that since $J\subseteq J_j$ we have $S\subseteq S_{J_j}'$. 
%Moreover, since the family of cubes $J_j$ is finite, there is $S'\in {\mathcal S}$ uniquely determined by $S$ such that 
%$S\subseteq S'$ and
%$$
%|\langle g\rangle_{S'}|=\sup_{\substack{S_{J_j}'\in {\mathcal S}\\ S_{J_j}'=J_j^a \textrm{ with } J\in \mathcal U^1(S)}}|\langle g\rangle_{S_{J_j}'}|
%$$ 
%Then, 
for all $J\in \mathcal U^1(S)$, by the testing condition \ref{selectiondef} we have 
$\langle g\rangle_{J}\leq C\langle g\rangle_{S}$ and $\langle f\rangle_{J_j}\leq C\langle f\rangle_{S}$ for each 
$J_j\in \child(\widehat{J}\,)$.
%Moreover, since by definition, $J\subset J_j$ we have two cases. If $J\subset J_j\subset S$ then we also get
%$|\langle f\rangle_{J_j}|\leq C|\langle f\rangle_{S}|$. Otherwise, if $J\subset S\subset J_j$ we denote by  
%by $S_j^r$ the chain of cubes such that 
%$S=S_j^0\subset S_j^1\subset \ldots \subset S_j^{k-1}\subset J_j\subset S_j^k$ where $S_j^k$ the is smallest cube in... such that $S\subset J_j\subset S_j^k$. By definition we have $|\langle f\rangle_{J_j}|\leq C|\langle f\rangle_{S_j^k}|$
%and each cube $S_j^r$ was selected so that $|\langle f\rangle_{S_j^{r-1}}|>C|\langle f\rangle_{S_j^{r}}|$. Then, since $C>1$, we have 
%$$|\langle f\rangle_{J_j}|\leq C|\langle f\rangle_{S_j^k}|\leq C^k|\langle f\rangle_{S_j^0}|\leq C|\langle f\rangle_{S}|$$ 
In addition,  
for $J\in \mathcal U(S)$, we have $\sup h_{J}\subset \widehat{J}\subset S$. 
With all this, 
we can rewrite and bound $T_1$ as follows:
\begin{align*}
%|\Pi_S (f,g)|
% &=|\sum_{S\in \mathcal S}
% \sum_{k=0,1}(-1)^k\sum_{j=1}^{2^n}
% \langle T\chi_{S},
% \sum_{J\in \mathcal D_M^c(S)}\langle f\rangle_{J_j}
% \langle g\rangle_{J^k}
% \mu(J)^{\frac{1}{2}}\psi_{J}\rangle |
% \\
% &\lesssim \sup_{j=1,\ldots, 2^n}
% |\langle T\chi_{S},
% \sum_{J\in \mathcal D_M^c(S)}\langle f\rangle_{J_j}
% \langle g\rangle_{J^k}
% \mu(J)^{\frac{1}{2}}\psi_{J}\rangle |
% \\
|T_1|&=|\langle T\mathbbm{1}_{S}, 
\Big(\sum_{J\in \mathcal U(S)}\langle f\rangle_{J_j}
\langle g,h_{J}\rangle h_{J}\Big)\mathbbm{1}_{S}\rangle |
\\
&\lesssim \sup_{j=0,1, \ldots, 2^n}
\|\mathbbm{1}_{S}T\mathbbm{1}_{S}\|_{L^2(\mu)}
\Big\|\sum_{J\in \tilde{\mathcal D}(S)\cap \mathcal D}\langle f\rangle_{J_j}
\langle g,h_{J}\rangle h_{J}\Big\|_{L^2(\mu)}
\\
&\lesssim \sup_{j=0,1, \ldots, 2^n}
\mu(S)^{\frac{1}{2}}F(S)
\Big(\sum_{J\in {\mathcal U}(S)}\langle f\rangle_{J_j}^2
|\langle g,h_{J}\rangle|^2\Big)^{\frac{1}{2}}
% \\
% &=\sup_{j=0,1, \ldots, 2^n}
% \mu(S)^{\frac{1}{2}}F(S)
% \Big(\sum_{S'\in \mathcal S(S)}\sum_{J\in \mathcal D_M^c(S')}|\langle f\rangle_{J_j}|^2
% |\langle g,h_{J}\rangle|^2\Big)^{\frac{1}{2}}
\\
&\lesssim 
\mu(S)^{\frac{1}{2}}F(S)\langle f\rangle_{S} 
\Big(\sum_{J\in {\mathcal U}(S)}
|\langle g,h_{J}\rangle|^2\Big)^{\frac{1}{2}}
\\
&\lesssim 
\mu(S)^{\frac{1}{2}}F(S)\langle f\rangle_{S} 
\Big\| \sum_{J\in {\mathcal U}(S)}
\langle g,h_{J}\rangle h_{J}\Big\|_{L^2(\mu)}. 
% \\
% &\lesssim \sup_{j=1, \ldots, 2^n}
% \langle f\rangle _{S}
% \langle g\rangle_{S}
% \mu(J)^{\frac{1}{2}}
% \mu(J)^{\frac{1}{2}}
% \langle T\chi_S, 
% \sum_{J\in \mathcal D_M^c(Q)}\frac{\langle f\rangle _{J_j}}{\langle f\rangle _{S}}
% \psi_{J}\rangle
\end{align*} 
%since by the selection process we have 
%$|\langle f\rangle_{J_j}|\leq C|\langle f\rangle_{S}|$.  
We want to estimate the $L^2$-norm. 
Since $\mathcal U(S)$ is convex, we should use \eqref{Deltaincoord1} to collapse the sum. But the sum does not contain cubes that belong to $\mathcal S(S)\setminus \{S\}$ and so, for any cube in the sum there is no guarantee that its siblings are also included in the sum. For that reason, 
we add and subtract a term containing all cubes $S'\in \mathcal S(S)$. This way $\mathcal U(S)\cup \mathcal S(S)$ is convex, and contains all siblings of each cube included in the union. Moreover, the maximum cube in the collection is $S$, while the minimum cubes are either selected cubes $S'\in \mathcal S(S)$ or cubes $M\in \mathcal M(S)=\mathcal D_{N}(S)\cap \, \mathcal U(S)$, the collection of cubes with minimal side-length $2^{-N}$ that are not included in any selected cube $S'\in {\mathcal S(S)}$. Then 
we sum the telescoping sum as follows: 
\begin{align}\label{allS'}
\nonumber
\sum_{J\in {\mathcal U}(S)}
\langle g,h_{J}\rangle h_{J}
&=\sum_{J\in {\mathcal U}(S)\cap \mathcal S(S)}
\langle g,h_{J}\rangle h_{J}
-\sum_{S'\in {\mathcal S(S)}}
\langle g,h_{S'}\rangle h_{S'}
\\
&=\sum_{S'\in {\mathcal S(S)}}
\langle g\rangle_{S'} \mathbbm{1}_{S'}+\sum_{M\in {\mathcal M(S)}}
\langle g\rangle_{M} \mathbbm{1}_{M}-\langle g\rangle_{S} \mathbbm{1}_{S}
-\sum_{S'\in {\mathcal S(S)}}
\langle g,h_{S'}\rangle h_{S'}. 
\end{align}
%where ${\mathcal M(S)}$ is the collection all cubes in ${\mathcal U}(S)$ with side-length $2^{-N}$ that are not included in any selected cube $S'\in {\mathcal S(S)}$. 
Now we study the last term. We note that 
\begin{align}\label{oneS'}
\nonumber
\langle g, h_{S'}\rangle h_{S'}
&=\mu(S')^{\frac{1}{2}}(\langle g\rangle_{S'}-\langle g\rangle_{\widehat{S'}})
\mu(S')^{\frac{1}{2}}(\frac{\mathbbm{1}_{S'}}{\mu(S')}-\frac{\mathbbm{1}_{\widehat{S'}}}{\mu(\widehat{S'}\, )})
%\\
%&=\langle g\rangle_{S'}\mathbbm{1}_{S'}-\langle g\rangle_{\widehat{S'}}\mathbbm{1}_{S'}
%-\frac{\mu(S')}{\mu(\widehat{S'})}\langle g\rangle_{S'}\mathbbm{1}_{\widehat{S'}}
%+\frac{\mu(S')}{\mu(\widehat{S'})}\langle g\rangle_{\widehat{S'}}\mathbbm{1}_{\widehat{S'}}
\\
&=\langle g\rangle_{S'}\mathbbm{1}_{S'}-\langle g\rangle_{\widehat{S'}}\mathbbm{1}_{S'}
-\langle g\mathbbm{1}_{S'}\rangle_{\widehat{S'}}\mathbbm{1}_{\widehat{S'}}
+\frac{\mu(S')}{\mu(\widehat{S'})}\langle g\rangle_{\widehat{S'}}\mathbbm{1}_{\widehat{S'}}, 
\end{align}
where we have used that 
\begin{equation}\label{wideningS'}
\frac{\mu(S')}{\mu(\widehat{S'})}\langle g\rangle_{S'}
=\frac{1}{\mu(\widehat{S'})}\int_{S'}g(x)d\mu(x)=\langle g\mathbbm{1}_{S'}\rangle_{\widehat{S'}}. 
\end{equation}
If we now substitute \eqref{oneS'} in the last term of \eqref{allS'}, we see that the corresponding first terms in these two numbered equations cancel out. This yields  
\begin{align*}%\label{togetherS'}
\sum_{J\in {\mathcal U}(S)}
\langle g,h_{J}\rangle h_{J}
&=\sum_{M\in {\mathcal M(S)}}
\langle g\rangle_{M} \mathbbm{1}_{M}-\langle g\rangle_{S} \mathbbm{1}_{S}
\\
&
+\sum_{S'\in {\mathcal S(S)}}\langle g\rangle_{\widehat{S'}}\mathbbm{1}_{S'}
+\sum_{S'\in {\mathcal S(S)}}\langle g\mathbbm{1}_{S'}\rangle_{\widehat{S'}}\mathbbm{1}_{\widehat{S'}}
-\sum_{S'\in {\mathcal S(S)}}\frac{\mu(S')}{\mu(\widehat{S'})}\langle g\rangle_{\widehat{S'}}\mathbbm{1}_{\widehat{S'}}=\sum_{i=1}^5h_i\, .  
\end{align*}

To estimate the norm of the function on the left hand side of previous equality, we bound the 
norm of each of the five resulting functions, which we denote by $h_i$ respectively. First, 
since $M\in {\mathcal M}(S)\subset \mathcal U(S)$, by the selection condition \eqref{selectiondef} we have  
$\langle g\rangle_{M}\lesssim \langle g\rangle_{S}$. Moreover, all cubes $M\in \mathcal M(S)$ have the same side-length, and are pairwise disjoint with union inside $S$. Then
\begin{align*}
\| h_1\|_{L^2(\mu)}^2&\lesssim \langle g\rangle_{S}^2
\sum_{M\in \mathcal M(S)}
\mu(M)
\leq \langle g\rangle_{S}^2\mu(S). 
\end{align*}

For the second term, we simply have 
%since the cubes 
%$S'\in \mathcal S(S)$ are pairwise disjoint by maximality and they form a sparse collection, we have 
$
\| h_2\|_{L^2(\mu)}^2=
\langle g\rangle_{S}^2\mu(S) 
$. 

To deal with the next term we note that, for each $S'\in \mathcal S(S)$, we have that $\widehat{S'}\in {\mathcal U}(S)$ and so, by the selection condition \eqref{selectiondef}, 
$\langle g\rangle_{\widehat{S'}}\lesssim \langle g\rangle_{S}$. Moreover, the collection $\mathcal S(S)$ is sparse with all its cubes inside $S$.
Then
\begin{align*}
\| h_3\|_{L^2(\mu)}^2&\lesssim \langle g\rangle_{S}^2
\sum_{S'\in \mathcal S(S)}\mu(S)
\lesssim \langle g\rangle_{S}^2\mu(S). 
\end{align*}
%where in the last inequality we used that $\mathcal S$ is sparse. 

For the fourth term, the reasoning is a bit longer. We enumerate the collection $\mathcal S(S)=\{ S_i\}_{i=1}^m$. Then
$$
\| h_4\|_{L^2(\mu)}^2
=\| \sum_{S'\in \mathcal S(S)}\langle g\mathbbm{1}_{S'}\rangle_{\widehat{S'}} 
\mathbbm{1}_{\widehat{S'}}
\|_{L^2(\mu)}^2
=\sum_{i,j=1}^m\langle g\mathbbm{1}_{S_i}\rangle_{\widehat{S_i}}
\langle g\mathbbm{1}_{S_j}\rangle_{\widehat{S_j}}\mu(\widehat{S_i}\cap \widehat{S_j}).
$$
We note that $\widehat{S_i}\cap \widehat{S_j}\in \{
\widehat{S_i}, \widehat{S_j}, \emptyset \}$. Then, by symmetry, we have 
$$
\| h_4\|_{L^2(\mu)}^2 \leq 
2\sum_{i=1}^{m}\langle g\mathbbm{1}_{S_i}\rangle_{\widehat{S_i}}
\sum_{\substack{j=1\\ \widehat{S_j}\subseteq \widehat{S_i}}}^m
\langle g\mathbbm{1}_{S_j}\rangle_{\widehat{S_j}}\mu(\widehat{S_j})
\lesssim \sum_{i=1}^m\langle g\mathbbm{1}_{S_i}\rangle_{\widehat{S_i}}
\sum_{\substack{j=1\\ \widehat{S_j}\subseteq \widehat{S_i}}}^m
\int_{S_j} g(x)d\mu(x).
$$
Since the cubes $S_j\in \mathcal S(S)$ are pairwise disjoint 
by maximality, and $S_j\subsetneq \widehat{S_j}\subseteq \widehat{S_i}$, 
we get 
\begin{align*}
\| h_4\|_{L^2(\mu)}^2
&\leq 
\sum_{i=1}^m\langle g\mathbbm{1}_{S_i}\rangle_{\widehat{S_i}}
\int_{\widehat{S_i}} g(x)d\mu (x) 
=\sum_{i=1}^m\langle g\rangle_{\widehat{S_i}}
\int_{S_i} g(x)d\mu(x) 
\\
&
\lesssim  \langle g\rangle_{S}\sum_{i=1}^m\int_{S_i} g(x)d\mu(x) 
\leq \langle g\rangle_{S}\int_{S} g(x)d\mu(x) 
=\langle g\rangle_{S}^2\mu(S).
\end{align*}
In the second line we used that since $\widehat{S_i}\in \mathcal U(S)$, by the selection condition \eqref{selectiondef} we have  
$\langle g\rangle_{\widehat{S_i}}\lesssim \langle g\rangle_{S}$, and that the cubes $S_i$ are pairwise disjoint 
by maximality. 

We now show that we can pointwise estimate the last term $h_5$ by $h_4$. By the selection condition \eqref{selectiondef} applied first to the non-selected cube $\widehat{S'}\in \mathcal U(S)$ and then to the selected cube $S'\in \mathcal S(S)$, we have 
$\langle g\rangle_{\widehat{S'}}\leq C\langle g\rangle_{S}<\langle g\rangle_{S'}$.  Then,  
\begin{align*}
h_5=\sum_{S'\in \mathcal S(S)}\frac{\mu(S') }{\mu(\widehat{S'}\,)}
\langle g\rangle_{\widehat{S'}}
\mathbbm{1}_{\widehat{S'}}
&\leq \sum_{S'\in \mathcal S(S)}\frac{\mu(S') }{\mu(\widehat{S'}\,)}
\langle g\rangle_{S'}
\mathbbm{1}_{\widehat{S'}}
=\sum_{S'\in \mathcal S(S)}
\langle g\mathbbm{1}_{S'}\rangle_{\widehat{S'}}
\mathbbm{1}_{\widehat{S'}}=h_4\, ,
%\\
%&\lesssim 
%\frac{1}{\mu(\widehat{S}\,)}
%\Big(\sum_{i=0}^r\mu(S_i) \langle g\rangle_{S_i}\Big)
%\mathbbm{1}_{\widehat{S}}
%=\Big(\sum_{i=0}^r\langle g\mathbbm{1}_{S_{i}}\rangle_{\widehat{S}}\Big)
%\mathbbm{1}_{\widehat{S}}=h_4\, .
\end{align*}
where we used again the equality \eqref{wideningS'} in the second last equality. 

With all this, we have 
\begin{align}\label{2ndterm}
\Big\| \sum_{J\in {\mathcal U}(S)}
\langle g, h_{J}\rangle h_{J}\Big\|_{L^2(\mu)}
&\lesssim \langle g\rangle_{S}\mu(S)^{\frac{1}{2}}
\end{align}
and so, 
we finally get that the first term in \eqref{2terms} satisfies 
\begin{align*}
|T_1|&\lesssim 
%F(2S)|\langle f\rangle_{S'}|\mu(2S)^{\frac{1}{2}}
%|\langle g\rangle_{S}|
%\mu(S)^{\frac{1}{2}} 
%%|\langle g\rangle_{S}|\mu(S)
%\\
%&=
\langle f\rangle_{S}\langle g\rangle_{S} 
\mu(S)F(S). 
\end{align*}

For the second term, we reason as follows. 
\begin{align*}
|T_2|&\leq \sum_{J\in \mathcal U^1(S)} 
\langle f\rangle_{J_j}
|\langle g,h_{J}\rangle |
|\langle T\mathbbm{1}_{Q\setminus S},h_{J}\rangle |
\\
&\leq \langle f\rangle_{S} 
\Big(\sum_{J\in \mathcal U^1(S)} 
|\langle g,h_{J}\rangle |^2\Big)^{\frac{1}{2}}
\Big(\sum_{J\in \mathcal U^1(S)} 
|\langle T\mathbbm{1}_{Q\setminus S},h_{J}\rangle |^2\Big)^{\frac{1}{2}}
\\
&\lesssim |\langle f\rangle_{S}||\langle g\rangle_{S}|\mu(S)^{\frac{1}{2}}
\Big(\sum_{J\in \mathcal U^1(S)} 
|\langle T\mathbbm{1}_{Q\setminus S},h_{J}\rangle |^2\Big)^{\frac{1}{2}}, 
\end{align*}
where the last inequality follows from \eqref{2ndterm}. 

We now prove that 
$$
\sum_{J\in \mathcal U^1(S)} 
|\langle T\mathbbm{1}_{Q\setminus S},h_{J}\rangle |^2\lesssim a_S^2\mu(S)
$$

%Let $S\in \mathcal{W}(R)$, $J\in \mathcal D_M^c(S)$ and 
Let $J\in \mathcal U^1(S)$ be fixed. 
Since $\widehat{J}\subset S$ and $h_{J}$ has mean zero, we can write
\begin{align*}
|\langle T(\mathbbm{1}_{S'\setminus S}), h_{J}\rangle |
&=|\langle T(\mathbbm{1}_{Q\setminus S})-T(\mathbbm{1}_{Q\setminus S})(c(\widehat{J})), h_{J}\rangle |
\\
&\leq \int_{Q\setminus S}\int_{\widehat{J}}
|K(t,x)-K(t,c(\widehat{J'}))||h_{J}(x)|d\mu(t)d\mu(x)
\\
&\leq \int_{Q\setminus S}
\int_{\widehat{J}}\frac{|x-c(\widehat{J})|^\delta }{|t-x|^{\alpha +\delta }}F(t,x)|h_{J}(x)|d\mu(t) d\mu(x),
\end{align*}
where $F(t,x)=L(|t-x|)S(|x-c(\widehat{J})|)D(1+\frac{|t+c(\widehat{J})|}{1+|t-c(\widehat{J'})|})$.

Since $J\subset 2^{-1}S$, we have  
for $t\in Q\setminus S$ that  
$\dist(t,J)\geq \dist(J,S)\geq \ell(S)/4$. 
%\ec(J,S)^{-\theta}\ell(J)=\Big(\frac{\ell(J)}{\ell(S)}\Big)^{1-\theta}\ell(S)=r\ell(S)$. 
Then we decompose 
$$
Q\setminus S=\bigcup_{i=0}^{\log \frac{\ell(Q)}{\ell(S)}}
S_i, 
$$
where 
$S_i=\{ t\in Q\setminus S : 2^{i-1}\ell(S)<
|t-c_{\widehat{J}}|\leq 2^{i}\ell(S)
%\dist(t,J')\leq 2^{i+1}\ell(S)
\}\subset 2^{i+1}S$. 
As shown in \cite{V2022},  
we have for $(t,x)\in S_i$ 
$$
F(t,x)\lesssim L(\ell(S))S(\ell(S))D(\rdist (2^irS, \mathbb B)).
$$
With this, 
\begin{align*}
|\langle T&(\mathbbm{1}_{Q\setminus S}),h_{J}\rangle | 
\\
&\leq L(\ell(S))S(\ell(S))\int_{\widehat{J}}|h_{J}(x)|
d\mu(x)
\sum_{i\geq 1}
D(\rdist (2^iS, \mathbb B))
\int_{S_i}\frac{\ell(J)^{\delta }}{|t-c_{\widehat{J}}|^{\alpha +\delta }}
d\mu(t)
\\
&\lesssim 
\ell(J)^{\delta }\mu(\widehat{J}\,)^{\frac{1}{2}}
L(\ell(S))S(\ell(S))\sum_{i\geq 1}
D(\rdist (2^iS, \mathbb B))
\frac{\mu(S_i)}{(2^{i}\ell(S))^{\alpha +\delta }}.
% \\
% &\leq 
% \frac{\ell(J)^{\delta }}{\dist(I_p,J_p)^{\alpha +\delta }}\mu(R)\mu(S)
% F([I,J], J,[ I,J] ).
\end{align*}
Moreover, the sum can be bounded by
\begin{align*}
\sum_{i\geq 1}
\frac{1}{(2^{i}\ell(S))^{\delta }}
\frac{\mu(2^{i+1}S)}{(2^{i+1}\ell(S))^{\alpha }}
D(\rdist (2^iS, \mathbb B))
&
\lesssim \frac{1}{\ell(S)^{\delta}}\tilde D(S)\rho(S).
%\rho_{\rm out}(S)
\end{align*}
Then, %since $r=\Big(\frac{\ell(J)}{\ell(S)}\Big)^{1-\theta}$, 
we have
\begin{align*}
|\langle T(\mathbbm{1}_{Q\setminus S}), h_{J}\rangle |
&\lesssim %\frac{1}{r^\delta}
\Big(\frac{\ell(J)}{\ell(S)}\Big)^{\delta }\mu(\widehat{J}\, )^{\frac{1}{2}}
L(\ell(S))S(\ell(S))%\rho_{\rm out}(S)
\tilde D(S)\rho(S)
\\
&\leq  \Big(\frac{\ell(J)}{\ell(S)}\Big)^{\delta }\mu(\widehat{J}\, )^{\frac{1}{2}}F(S)\rho(S)
=
a_{S}\Big(\frac{\ell(J)}{\ell(S)}\Big)^{\delta }\mu(\widehat{J}\, )^{\frac{1}{2}}.
\end{align*}
%The last inequality is due to the fact that, 
%%since $S\subset R\in \mathcal D_M^c(Q)$, we have 
%as we saw before, 
%$S\in \mathcal D_M^c(Q)$, and then 
%$F(S)\leq \sup \{F(K)\rho(K): K\in \mathcal D_{\geq N}(Q)\}=a_{S} $. 
Now, we parametrize the cubes $J', J$ according with their relative size with respect to $S$: $\ell(J)=2^{-k}\ell(S)$.
%and $\ell(J')=2^{-k'}\ell(J)$, which imply
%$\ell(J')=2^{-(k+k')}\ell(S)$. 
%Note that for each $k$, the cubes can be arranged by the condition $J\subset S+j\ell(S)$ with $j=(j_1, \ldots , j_n)$ and $|j_i|\in \{ 0, \ldots, 2^k\}$.  
We sum in $J$ %and $J'$ 
and use that the cubes with fixed side length are disjoint, to get
\begin{align*}
\sum_{J\in \mathcal U^1(S)} 
&%\sum_{J'\in \mathcal J(J)}
|\langle T(\mathbbm{1}_{Q\setminus S}), h_{J}\rangle |^{2}
\lesssim a_S^2
\sum_{J\in \mathcal U^1(S)} 
%\sum_{J'\in \mathcal J(J)}
\Big(\frac{\ell(J)}{\ell(S)}\Big)^{2 \delta }\mu(J)
%\\
%&\lesssim a_S^2
%\sum_{k\geq 1}2^{-k2\delta}
%%\sum_{\substack{j\in \mathbb Z^n\\ |j|\leq 2^k}}
%\sum_{\substack{J\in {\mathcal D}(S)\\ \ell(J)=2^{-k}\ell(S)}}
%%\sum_{k'\geq 1}2^{-k'2\delta}
%%\sum_{\substack{J\in \mathcal J(J)\\ 
%%\ell(J)=2^{-k'}\ell(S')}}
%\mu(J)
%\\
%&\leq a_S^2
%\sum_{k\geq 1}2^{-k2\delta}
%\sum_{\substack{J\in {\mathcal D}(S)\\ \ell(J)=2^{-k}\ell(S)}}
%\sum_{k'\geq 1}2^{-k'2\delta}
%\mu(J)
\\
&\lesssim a_S^2
\sum_{k\geq 1}2^{-k2\delta}
\sum_{\substack{J\in {\mathcal D}(2^{-1}S)\\ \ell(J)=2^{-k}\ell(S)}}
\mu(J)
\\
&
\leq a_S^2
\sum_{k\geq 1}2^{-k2\delta}\mu(S)
\lesssim \mu(S)
a_S^2.
\end{align*}
%Summing now over the cubes $S$ in $\mathcal{W}(R)$, we finally get
%\begin{align*}
%\sum_{S\in \mathcal{W}(R)}
%\sum_{J\in \mathcal D_M^c(S)} 
%\sum_{J'\in \mathcal J(J)}
%|\langle T(\chi_{S'\setminus 2S}), h_{J'}\rangle |^{2}
%&\lesssim a_R^2\sum_{S\in \mathcal{W}(R)}
% \mu(S)
%\lesssim a_R^2\mu(R).
%\end{align*}

With this we get 
$
|T_2|\lesssim 
%|\langle f\rangle_{S'}|
%|\langle g\rangle_{S}|\mu(S)^{\frac{1}{2}}
%a_S\mu(S)^{\frac{1}{2}}
%%|\langle g\rangle_{S}|\mu(S)
%\\&=
a_S \langle f\rangle_{S} \langle g\rangle_{S} 
$,  
%We now deal with $j=0$, that is, when $J_j=Q$
and so, finally
%\begin{align*}
%|\Pi_1 (f,g)|
%&\lesssim \sum_{S\in \mathcal S}
%F(S)\Big(\sum_{S'\in \mathcal S(S)}|\langle f\rangle_{S'}|^2
%\mu(S')\Big)^{\frac{1}{2}}
%\Big(\sum_{S'\in \mathcal S(S)}|\langle g\rangle_{S'}|^2
%\mu(S')\Big)^{\frac{1}{2}}. (?)
%%|\langle f\rangle_{S}||\langle g\rangle_{S}|
%%\mu(S)
%\end{align*}
\begin{align*}
|\Pi (f,g)|
&\lesssim \sum_{S\in \mathcal S}
a_S \langle f\rangle_{S} 
\langle g\rangle_{S}. 
%+|\langle f\rangle_{S}||\langle g\rangle_{S'}|)\mu(S). 
%|\langle f\rangle_{S}||\langle g\rangle_{S}|
%\mu(S)
\end{align*}

\end{proof}

\subsection{Essentially attached cubes and fragments of the two paraproducts}

We work now with the expression $A-B_1-B_2$, where $A$ contains all attached cubes $I,J\in \mathcal D(Q)$ such that $\inrdist(I,J)\leq \ec(I,J)^{-\theta}$, that is, $1\leq k\leq  2^{\theta |e|}$, $B_1$ contains cubes $I,J\in \mathcal D(Q)$, $\widehat{J}\subset \widehat{I}$ and , while $B_2$ is similar with the only change that $\widehat{I}\subset \widehat{J}$. 
For this, we remind the notation $I\parallel J$, if  $0\leq \rdist(I,J)<1$,  $\lambda(\widehat{I},\widehat{J}\, )\leq 1$, and 
we now use a new symbol: $I\Bumpeq J$, if  $\rdist(I,J)=0$,  $\lambda(\widehat{I},\widehat{J}\, )\leq 1$. Then the terms to be estimated are 
%For the terms corresponding to the case 
%$1\leq k\leq 2^{\theta e}$ we proceed in a different way.
%Then we can rewrite these terms as 
\begin{align*}
A
&=\sum_{I\parallel J}
\langle f,h_{I}\rangle 
\langle g, h_{J}\rangle
\langle Th_I,h_J\rangle , 
%&=\sum_{e\in \mathbb Z}\sum_{k=1}^{2^{\theta |e|}}
%\Big(\sum_{I}
%\sum_{J\in I_{e,1,k}}
%+\sum_{J}
%\sum_{I\in J_{-e,1,k}}
%\Big)
%\langle f,h_{I}\rangle 
%\langle g, h_{J}\rangle
%\langle Th_I,h_J\rangle
%\\
%=\sum_{I\in \mathcal{D}(Q)_{\geq N}}\sum_{J\in I_{\theta}}
%\langle f,h_I\rangle \langle g,h_J\rangle
%\langle Th_I,h_J \rangle , 
\end{align*}
\begin{align*}
G_1=
&\sum_{\substack{I\Bumpeq J\\ 
%\ell(J)\leq \ell(I)\\ 
\widehat{J}\subset \widehat{I}}}
\langle f, h_{I}\rangle \langle g, h_{J}\rangle
\langle Th_{I}^{\widehat{J}, Q}, h_{J}\rangle ,
\hskip30pt 
G_2=
&\sum_{\substack{I\Bumpeq J\\ 
%\ell(J)\leq \ell(I)\\ 
\widehat{I}\subset \widehat{J}}}
\langle f, h_{I}\rangle \langle g, h_{J}\rangle
\langle Th_{I}, h_{J}^{\widehat{I}, Q}\rangle .
\end{align*} 

\begin{proposition}\label{resultattached} Let $A$, $G_1$ and $G_2$ be as defined before by using the dyadic grid $\mathcal D_i$. Let $B_i$ be the sub-bilinear form defined in \eqref{bilinearsquare}. Then
$$
|A-G_1-G_2|\lesssim \Lambda_{\mathcal S}(f,g)+B_i(f, g).
$$

\end{proposition}
\begin{proof}
We are actually going to prove that a combination of terms in $A, G_1, G_2, B_i(f, g)$ can be done arbitrarily small, in particular smaller than the simplest possible sparse form:
$$
|A-G_1-G_2-B_i(f, g)|\lesssim \langle f\rangle_Q \langle g\rangle_Q\mu(Q)\lesssim \Lambda_{\mathcal S}(f,g), 
$$
since $Q\in \mathcal S$ by definition.

We separate from $A$ terms that can be combined with each $G_i$ to write
\begin{align*}
A-G_1-G_2
&=\sum_{\substack{I\parallel J\\ \rdist(I,J)>0}}
\langle f,h_{I}\rangle 
\langle g, h_{J}\rangle
\langle Th_I,h_J\rangle
\\
&
+\sum_{\substack{I\Bumpeq J\\ 
%\ell(J)\leq \ell(I)\\ 
\widehat{J}\subset \widehat{I}}}
\langle f, h_{I}\rangle \langle g, h_{J}\rangle
\langle T(h_I-h_{I}^{\widehat{J}, Q}), h_{J}\rangle
+\sum_{\substack{I\Bumpeq J\\ 
%\ell(J)\leq \ell(I)\\ 
\widehat{I}\subset \widehat{J}}}
\langle f, h_{I}\rangle \langle g, h_{J}\rangle
\langle Th_{I}, h_J-h_{J}^{\widehat{I}, Q}\rangle 
\\
&=H_1+H_2+H_3. 
\end{align*}
% We assume $f$ and $g$ to be compactly supported on a dyadic cube $Q$, with mean zero and constant on cubes of side length $2^{-(N-1)}$. 
% %Let $\supp f\cup \supp g\subset R$ with 
% %$\langle f\rangle_R=\langle g\rangle_R=0$.
We remind the following notation:
for $N\in \mathbb N$, which we chose in \eqref{Small}, we 
%we define the collections of cubes 
write 
$\mathcal{D}(Q)_{\geq N}=\{ I\in \mathcal{D}_M^c(Q)/ \ell(I)\geq 2^{-N}\ell(Q)\}$  
and $\mathcal{D}(Q)_{N}=\{ I\in \mathcal{D}_M^c(Q)/ \ell(I)=2^{-N}\ell(Q)\}$.
% Now we split the measure $\mu $ into two parts. 
% Let $\partial \mathcal{D}_N(Q)$ be the set defined by the union of
% the boundaries of all cubes in $\mathcal{D}_N(Q)$. Then we decompose $\mu=\mu_{\rm int}+\mu_\partial$ with $\mu_\partial=\mu|_{\partial \mathcal{D}_N(Q)}$. 
% We first work with $\mu_{\rm int}$ and its associated operator $T_{\rm int}$. But, to simplify notation, we keep denoting this measure and its associated operator by $\mu $ and $T$. Then, if for every cube $I$ we denote its interior by 
% $\inte I=I\setminus \partial I$, we have 
% the equality 
% $\mu(I)=\mu(\inte I)$.
Moreover, for $I\in \mathcal{D}(Q)_{\geq N}$, let 
$I_{\theta}$ be the family of cubes $J\in \mathcal{D}(Q)_{\geq N}$
such that 
%$\dist(I,J)\leq (\frac{\ell(I\smlor J)}{\ell(I\smland J)})^\theta \ell(I\smland J)=2^{|e|\theta}\ell(I\smland J)$
$1\leq k\leq  2^{\theta |e|}$, and a symmetric definition for $J_{\theta}$. 
%$\inrdist(I,J)-1\leq \ec(I,J)^{-\theta}$.

We rewrite these three terms changing their corresponding summation symbols by  
\begin{align*}
H_1\equiv \sum_{I\in \mathcal{D}(Q)_{\geq N}}\sum_{\substack{J\parallel I\\ J\in I_{\theta}}}, 
\hskip30pt 
H_2\equiv \sum_{I\in \mathcal{D}(Q)_{\geq N}}\sum_{\substack{J\Bumpeq I\\ 
%\ell(J)\leq \ell(I)\\ 
\widehat{J}\subset \widehat{I}\\ J\in I_{\theta}}}, 
\hskip30pt 
H_3\equiv \sum_{J\in \mathcal{D}(Q)_{\geq N}}\sum_{\substack{I\Bumpeq J\\ 
%\ell(J)\leq \ell(I)\\ 
\widehat{I}\subset \widehat{J}\\ I\in J_{\theta}}} , 
\end{align*}
respectively. We plan to add some more terms to each sum. For this, 
we need to describe the collection $I_{\theta }$. But first we remind the following definition: for $I\in \mathcal{D}(Q)_{N}$, 
$I^{\rm fr}$ denotes the family of dyadic cubes $I'\in \mathcal{D}(Q)_{N}$ such that $\ell(I')=\ell(I)$ and $\dist(I',I)=0$. The cardinality of $I^{\rm fr}$ is $3^n$.  

For each fixed cube $I\in \mathcal{D}(Q)_{\geq N}$ and each side-length $\ell(J)$ so that $\ell(J)\leq \ell(I)$, 
all cubes $J\in I_{\theta}$ satisfy $J\subset I'$ with $I'\in I^{\rm fr}$. This property is due to the following inequality:
since $1\leq k\leq 2^{\theta |e|}+1$, 
\begin{align}\label{near}
\nonumber
\dist(J,\mathfrak{B}_I)&=(k-1)\ell(I\smland J)%<k\ell(J)
<2^{|e|\theta}\ell(I\smland J)
= \Big(\frac{\ell(I\smlor J)}{\ell(I\smland J)}\Big)^{\theta}\ell(I\smland J)
\\
&
=\Big(\frac{\ell(I\smland J)}{\ell(I\smlor J)}\Big)^{1-\theta}
\ell(I\smlor J)=2^{-|e|(1-\theta)}\ell(I\smlor J)\leq \ell(I\smlor J).
\end{align}
This inequality is used to show the two most relevant properties of $I_{\theta}$. 
\begin{itemize}
\item[a)] Since $\dist(J,\mathfrak{B}_I)<\ell(I)$, 
the cubes $J\in I_{\theta}$ are either contained in $I$ or in one of its $3^n$ friends $I'\in I^{\rm fr}$ as we see: 
if $J\in I_{\theta} $
and $J\subset I'$ with $\ell(I')=\ell(I)$, then
%In fact, as the eccentricity grows, the distance $\dist(J,\mathfrak B_I)$ tends to zero. 
$$
\dist(I,I')\leq \dist(I,J)+\dist(J,I')< \ell(I)+0,
$$
and so $\dist(I,I')=0$. This implies $I'\in I^{\rm fr}$.

\item[b)] For each fixed cube $I\in \mathcal{D}(Q)_{\geq N}$ and each side-length $\ell(J)$ so that $\ell(J)>\ell(I)$, we see that 
there are only up to $3^n$ cubes $J\in I_{\theta}$. The reason is similar as before: since from \eqref{near} we have in this case $\dist(J,\mathfrak{B}_I)< \ell(J)$ for each $J\in I_{\theta}$, we have that 
 $I$ is either contained in $J$ or in one of its friends $J'\in J^{\rm fr}$. That is, for each $J\in I_{\theta}$ of fixed size, either it contains $I$ or one of its friends does: if $J\in I_{\theta} $
and $I\subset J'$ with $\ell(J')=\ell(J)$, then
$$
\dist(J,J')\leq \dist(J,I)+\dist(I,J')< \ell(J)+0,
$$
and so $\dist(J,J')=0$. This implies $J'\in J^{\rm fr}$.
Then the cardinality of $J^{\rm fr}$ limits the number of possible cubes $J\in I_{\theta}$ of fixed side-length to up to $3^n$. 
\end{itemize}

Once $I_{\theta}$ is fully described, we consider the collection of cubes above $I_{\theta}$.
For each $I\in \mathcal{D}(Q)_{\geq N}$, let 
$I_{\max}$ be the family of cubes $J\in I_{\theta}$ that are maximal with respect to the inclusion. Then let 
$I_{\rm over}$ be the family of cubes 
$R\in \mathcal{D}(Q)_{\geq N}$ such that 
$J\subsetneq R$ for some $J\in I_{\max}$. 
We note that 
for all $I\in \mathcal{D}(Q)_{\geq N}$, 
either $Q\in I_{\max}$ (if $I\in Q_{\theta}$) or 
$Q\in I_{\rm over}$. So, we always have 
$Q\in I_{\theta}\cup I_{\rm over}$. 
We also note that 
all cubes in $I_{\rm over}$ satisfy that $k>2^{\theta |e|}$ with respect to $I$.

Our next step is to include the cubes in $I_{\rm over}$ and for each pair $(I,J)$ in the sum defining $H_j$, we add the siblings either $I$ or $J$ that are not already contained in $H_j$. We show the details for $H_1$ and make some comments about the other two terms. 
We define
\begin{align*}
{\rm Ad_{H_1}}&=\sum_{I\in \mathcal{D}(Q)_{\geq N}}\sum_{J\in I_{\rm over}}
\langle f,h_I\rangle \langle g,h_J\rangle
\langle Th_I,h_J \rangle 
\\
&
+\sum_{I\in \mathcal{D}(Q)_{\geq N}}\sum_{J\in I_{\theta}}\sum_{\substack{J'\in \child(\hat{J})\\ J'\notin I_{\theta}}}
\langle f,h_I\rangle \langle g,h_{J'}\rangle
\langle Th_I,h_{J'} \rangle 
\\
&
+\sum_{J\in \mathcal{D}(Q)_{\geq N}}\sum_{I\in J_{\theta}}\sum_{\substack{I'\in \child(\hat{I})\\ I'\notin J_{\theta}}}
\langle f,h_{I'}\rangle \langle g,h_J\rangle
\langle Th_{I'},h_J \rangle , 
\end{align*}
and similar for $H_2$ and  $H_3$. 

Then to previous expression we  
add and subtract ${\rm Ad_{H_j}}$ to obtain for example 
\begin{align}\label{b4teles}
|H_1|\lesssim |\sum_{I\in \mathcal{D}(Q)_{\geq N}}
\langle f,h_I\rangle 
\langle Th_I,\sum_{J\in I_{\theta}\cup I_{\rm over}}\langle g,h_J\rangle h_J \rangle |
+|{\rm Ad_{H_1}}|.
\end{align}
In the last expression, the collections  $I_{\theta}$ and $I_{\rm over}$ are not exactly the same as defined before. But we use the same notation for them because they consist of the same cubes as before plus their corresponding siblings if they were not initially in the expression defining $H_1$. 

Since all pair of cubes $(I,J)$ added satisfy that $k>2^{\theta |e|}$, we can apply the reasoning of any of the previous cases
%(adding and subtracting the corresponding part of a paraproduct when needed)
to prove that the second term  in \eqref{b4teles} can be estimated by a bilinear form acting on square functions, namely
$|{\rm Ad_{H_j}}|\lesssim B_i(f, g)$. 
We note that in the case of ${\rm Ad_{H_2}}$ this is due to the fact that the expression 
$\langle T(h_I-h_{I}^{\widehat{J}, Q}), h_{J}\rangle$ for the described cubes satisfies the inequality 
\eqref{bump3} as given in Proposition \ref{twobumplemma1}. Similar for $j=3$. 
Then we only need to study the first term in \eqref{b4teles}. 

For each  $I\in \mathcal{D}(Q)_{\geq N}$, we have that $I_{\theta}\cup I_{\rm over}$ is a convex family of cubes that contains all siblings of each cube in the sum, has minimal cubes in $\mathcal D(Q)_N$ and maximal cube $Q$. Then,  
by summing a telescoping series, 
%and using that $\langle g\rangle_Q=0$, 
we get  
$$
\sum_{J\in I_{\theta}\cup I_{\rm over}}\langle g,h_J\rangle h_J 
=\sum_{J\in \mathcal{D}(Q)_{N}\cap I_{\theta}}
\langle g\rangle_J\mathbbm{1}_J-\langle g\rangle_{Q}\mathbbm{1}_{Q}. 
%=\sum_{J\in \mathcal{D}(Q)_{N}\cap I_{\theta}}
%\langle g\rangle_J\mathbbm{1}_J, 
$$
%where we used that 
%$Q'\in \mathcal D$ such that 
%$\supp g\subset Q'$ and so, 
% Now, for $J\in \mathcal{D}(Q)_{\geq N}$ we denote as before
% $J_{\theta}$ the family of cubes $I\in \mathcal{D}(Q)_{\geq N}$
% for which $\dist(I,J)\leq (\frac{\ell(I\smlor J)}{\ell(I\smland J)})^\theta \ell(I\smland J)$
% that is, $1\leq k\leq 2^{\theta e}$. 
%Then, 
%$\langle g\rangle_{Q}=0$. 
The term 
$\langle g\rangle_{Q}\mathbbm{1}_{Q}$
can be dealt again 
by using that $Q$ can be chosen so that the average $\langle f\rangle_{Q}$ is as small as needed. 
and repeating the arguments of Lemma \ref{densityinL2-1}. 
Then we focus on the first term. 

By Fubini's theorem, the first term in the right hand side of \eqref{b4teles} can be rewritten as 
\begin{align}\label{telescopy4f}
\sum_{I\in \mathcal{D}(Q)_{\geq N}}&\sum_{J\in \mathcal{D}(Q)_{N}\cap I_{\theta}}
\langle f,h_I\rangle 
\langle g\rangle_J
\langle Th_I,\mathbbm{1}_J \rangle 
=\hspace{-.2cm}\sum_{J\in \mathcal{D}(Q)_{N}}
\langle g\rangle_J
\langle T(\hspace{-.2cm}\sum_{I\in \mathcal{D}(Q)_{\geq N}\cap J_{\theta}}\hspace{-.2cm}
\langle f,h_I\rangle h_I),\mathbbm{1}_J \rangle ,
\end{align}
where $J_\theta $ is defined as $I_\theta$ was defined before. 
% Furthermore, we note that since $g$ is constant on each cube $J\in \mathcal D(Q)_N$, if we denote by 
% $\tilde J$ the interior of $J$ we have 
% $$
% \langle g\rangle_J=\langle g\rangle_{\tilde J}
% $$
% Then, we can and will assume that $J$ is an open cube. 

For each minimal cube $J\in \mathcal{D}(Q)_{N}$, the collection 
$J_{\theta}$ consists of cubes $I$ such that $\ell(I)\geq \ell(J)$. 
Then, from previous description of $I_{\theta}$ applied now to $J_{\theta}$, we have that for each fixed side length there are up to $3^n$ cubes $I\in J_{\theta}$. But remember that for every cube in the sum we added its siblings if they were not already in the sum. Then $J_{\theta}$ can be organized in up to $3^n$ collections of cubes $\textrm{Chain}_j$ such that 
the collection consists of two parts: 
\begin{itemize}
\item an increasing convex chain of cubes
$I_{j,R_j}\subset I_{j,R_j-1}\subset \ldots \subset I_{j, r_{j}}$, where each cube is parametrized by its side-length   
$\ell(I_{j,r})=2^{-r}\ell(Q)$ with $r\in 
\{ r_{j},\ldots, R_j\}\subset 
\{0,\ldots, N\}$, 
\item and the siblings of each cube in previous chain.
\end{itemize}
Some of these collections may be empty, while for some others, its minimal cube is a cube in $J^{\rm fr}$. 
The cubes $I_{j,R_j}$ and $I_{j,r_j}$ depend on $J$, but we do not reflect that dependency on the notation. 

%We now remind the following definition: for $J\in \mathcal{D}(Q)_{N}$, 
%$J^{\rm fr}$ denotes the family of dyadic cubes $J'\in \mathcal{D}(Q)_{N}$ such that $\ell(J')=\ell(J)$ and $\dist(J',J)=0$. Then 
%we note that, since $J$ has minimal side length, the condition $I\in J_\theta $ implies that $J'\subset I$ for some $J'\in J^{\rm fr}$. 

%Moreover, the cardinality of $J^{\rm fr}$ is $3^n$ and so, we can enumerate the cubes in $J^{\rm fr}$ as 
%$\{ J_j\}_{j=1}^{3^n}$ by their fixed position with respect to $J$. 
%Then, for each $j\in \{1, \ldots, 3^n\}$ 
%%and each $J_j\in J^{\rm fr}$, 
%the cubes $I\in \mathcal{D}(Q)_{\geq N}\cap J_{\theta}$ such that $J_j\subset I$ 
%form an increasing convex chain of cubes
%$J_j=I_{j,N}\subset I_{j,N-1}\subset \ldots \subset I_{j, r_{j}}$ parametrized by their size length   
%$\ell(I_{j,r})=2^{-r}\ell(Q)$ with $r\in 
%\{ r_{j},\ldots, N\}\subset 
%\{0,\ldots, N\}$. Some chains may be empty.
%%, but we will still keep this notation as if they were not empty. 
%% Finally, if $I,J\in \mathcal{D}_N(Q)$, we have that 
%% $e=0$ and so, since $1\leq k\leq 2^{\theta |e|}$, we get $k=1$, that is, 
%% $\dist(J_{j,N},J)=0$, which in turn implies 
%% % With this, we have that 
%% % $I_{N,J'}\in \mathcal{D}(Q)_{N}$ with $\dist(I_{N,J'},J)=0$, that is, 
%% $I_{j,N}=J_j\in J^{\rm friends}$. 
%% There are only $3^n$ of  cubes in 
%% $\mathcal{D}(Q)_{N}$ with $\dist(I_{N,J'},J)=0$. 
%% %Moreover we can 
%% %assume $i=2$ for all $i\geq 2$. 
%% %%\in \{1, \ldots, 2^d\}$ while 
%% %Also $j_i\leq 3^n$. 
Each collection is convex, and such that for each cube in the collection all its siblings are also in the collection. Moreover the minimal cube is  $I_{j,R_j}$ while the maximal cube is $\widehat{I}_{j,r_j}$. 
Then, for each fixed $J\in \mathcal{D}(Q)_{N}$, we can sum a telescoping sum to get 
%for $e_j$ maximal and $R=2^{\theta e_J+1}+1$,
\begin{align*}
\sum_{I\in \mathcal{D}(Q)_{\geq N}\cap J_{\theta}}
\langle f,h_I\rangle h_I
&=\sum_{j=1}^{3^n}\alpha_j \sum_{I\in \textrm{Chain}_j}\langle f,h_I\rangle h_I
\\
&
=\sum_{j=1}^{3^n}
%\sum_{\substack{j=1\\ J_j\in J^{\rm fr}}}^{3^n}
\langle f\rangle_{{I}_{j,R_j}} \mathbbm{1}_{{I}_{j,R_j}}
-\sum_{j=1}^{3^n}\langle f\rangle_{\widehat{I}_{j,r_j}} \mathbbm{1}_{\widehat{I}_{j,r_j}}.
\end{align*}
%To simplify notation, we write each cube in the sum without the accent. 
%The cubes $I_{j,R_j}$ and $I_{j,r_j}$ depend on $J$, but we do not reflect that dependency on the notation. 
With this \eqref{telescopy4f} can be written as 
\begin{align}\label{T1T2}
\sum_{J\in \mathcal{D}(Q)_{N}}&
\sum_{j=1}^{3^n}
%\sum_{\substack{j=1\\ J_j\in J^{\rm fr}}}^{3^n}
\langle f\rangle_{I_{j,R_j}} 
\langle g\rangle_J
\langle T\mathbbm{1}_{I_{j,R_j}},\mathbbm{1}_J \rangle
-\hspace{-.4cm}\sum_{J\in \mathcal{D}(Q)_{N}}
\sum_{j=1}^{3^n}\langle f\rangle_{\widehat{I}_{j,r_j}} 
\langle g\rangle_J
\langle T\mathbbm{1}_{\widehat{I}_{j,r_j}}\hspace{-.1cm},\mathbbm{1}_J \rangle
=T_1-T_2.
\end{align}
Since exactly the same ideas work to estimate each term, we only show the calculations for the second one, which we write without the tilde just to simplify notation. 
%We now note an important detail. 
% Now, we divide the operator in two parts 
% $T=T_{\gamma , Q}= T_1+T_2$,
% corresponding to each of the two terms in the definition of the kernel
% \begin{align*}
% K_{\gamma,Q}(t,x)&=K(t,x)
% (1-\phi(\frac{|t-x|}{\gamma}))
% %\phi(\frac{2|t|}{\ell(Q)})
% %\phi(\frac{2|x|}{\ell(Q)})
% +\frac{F(t,x)}{\gamma^\alpha}
% %K(t', x')
% \phi(\frac{|t-x|}{\gamma})
% %\phi(\frac{2|t|}{\ell(Q)})
% %\phi(\frac{2|x|}{\ell(Q)})
% \\
% &
% =K_{\gamma,Q,1}(t,x)+K_{\gamma,Q,2}(t,x). 
% \end{align*}
% Note that we have used that $t, x\in Q$. 
% We then write previous expression as 
% \begin{align*}
% S_1-S_2=S_{1,1}+S_{1,2}-(S_{2,1}+S_{2,2})
% \end{align*}

We mention now that the analog expression for the terms corresponding to $H_2$ and $H_3$ can be written separated in some 
extra terms since it does not make any difference to keep them together. For example, for $H_2$ we have 
\begin{align*}
\sum_{J\in \mathcal{D}(Q)_{N}}&
\sum_{j=1}^{3^n}
%\sum_{\substack{j=1\\ J_j\in J^{\rm fr}}}^{3^n}
\langle f\rangle_{I_{j,R_j}} 
\langle g\rangle_J
\langle T\mathbbm{1}_{I_{j,R_j}},\mathbbm{1}_J \rangle
-\sum_{J\in \mathcal{D}(Q)_{N}}
\sum_{j=1}^{3^n}\langle f\rangle_{\widehat{I}_{j,r_j}} 
\langle g\rangle_J
\langle T\mathbbm{1}_{\widehat{I}_{j,r_j}},\mathbbm{1}_J \rangle
\\
&-\sum_{J\in \mathcal{D}(Q)_{N}}
\sum_{j=1}^{3^n}
%\sum_{\substack{j=1\\ J_j\in J^{\rm fr}}}^{3^n}
\langle f\rangle_{I_{j,R_j}} \mathbbm{1}_{I_{j,R_j}}(c(J))
\langle g\rangle_J
\langle T\mathbbm{1}_{Q},\mathbbm{1}_J \rangle
\\
&
+\sum_{J\in \mathcal{D}(Q)_{N}}
\sum_{j=1}^{3^n}\langle f\rangle_{\widehat{I}_{j,r_j}} \mathbbm{1}_{\widehat{I_{j,R_j}}}(c(J))
\langle g\rangle_J
\langle T\mathbbm{1}_{Q},\mathbbm{1}_J \rangle .
\end{align*}
These terms have very similar expressions and can all be treated in the same way we plan to deal with \eqref{T1T2}. 

To estimate then $T_2$, we first note that, by using that $t, x\in Q$, we can write the kernel operator as 
\begin{align*}
K_{\gamma,Q}(t,x)&=K(t,x)
(1-\phi(\frac{|t-x|}{\gamma})).
%\phi(\frac{2|t|}{\ell(Q)})
%\phi(\frac{2|x|}{\ell(Q)})
% +\frac{F(t,x)}{\gamma^\alpha}
% %K(t', x')
% \phi(\frac{|t-x|}{\gamma})
% %\phi(\frac{2|t|}{\ell(Q)})
% %\phi(\frac{2|x|}{\ell(Q)})
% \\
% &
% =K_{\gamma,Q,1}(t,x)+K_{\gamma,Q,2}(t,x). 
\end{align*}

We know that 
$K_{\gamma, Q}(t,x)=0$ for $|t-x|<\gamma $. Then, 
if $\ell(J)\leq \ell(I)<\gamma/3$ and 
$\dist(I,\mathfrak{B}_{J})< \gamma/3 $, 
we have 
$|t-x|<\gamma $ for all $t\in I$ and $x\in J$ and so,
$\langle T\mathbbm{1}_I,\mathbbm{1}_J \rangle =0$. 

From \eqref{near}, we have that 
the cubes $J$ and $I_{j,r_j}$ in $T_2$ satisfy 
%$\ell(J_j)=\ell(J)$, 
%$\dist(J_j,J)=0$. Moreover, from \eqref{near}
%we have 
$
\dist(J, \mathfrak{B}_{I_{j,r_j}})< 
%2^{|e|\theta}\ell(J_j)-1\leq \Big(\frac{\ell(J_j)}{\ell(J_{j,k_j})}\Big)^{1-\theta}\ell(J_{j,k_j})\leq 
\ell(I_{j,r_j})
$. 
%Then, since $\ell(J_j)\leq \ell(J_{j,k_j})$, we get 
%$$
%\dist(J_{j,k_j},J)\leq \dist(J_j,J)+\ell(J_j)+\dist(J_{j,k_j},J_j)
%\leq 2\ell(J_{j,k_j}).
%$$
%%\leq \ell(J)\leq \ell(J_{j,k_j})$. 
Then, when 
$\ell(I_{j,r_j})<\gamma /6$, we have 
$\dist(J, \mathfrak{B}_{I_{j,r_j}})\leq \gamma /3$ and so 
$\langle T\mathbbm{1}_{I_{j,r_j}},\mathbbm{1}_{J} \rangle =0$. 

In other words, 
the scales for which the dual pair is non-zero 
%the non-zero terms in the second sum 
satisfy 
$\ell(I_{j,r_j})=2^{-r}\ell(Q)\geq \gamma /6$, that is,
$r\leq \log\frac{6\ell(Q)}{\gamma }\leq N_0$. And since 
$r\in \{0,\ldots, N\}$, that means that the non-zero terms in $T_2$ 
consist of cubes $I_{j,r_j}$ with  
at most $N_0+1$ different side-lengths (in fact in the $N_0+1$ largest scales, all of them in $\{0, 1, \ldots , N_0\}$).

% In summary, even though $N$ can (and will) be taken as small as needed, 
% %the maximum number of scales in the previous sums is limited %to $N_0=\log\frac{2\ell(S)}{\epsilon }$. Thus, 
% whenever 
% $N>N_0$ the first sum $S_1$ will equal zero, while the second sum $S_2$ will have at most $N_0+1$ non-zero terms (the first $N_0+1$ ones).

% Since
% $J\subset I_{j,J}$, from Proposition \ref{realbound2} we have
% \begin{align*}
% |\langle T\chi_{I_{j,J}},\chi_J \rangle|
% &\lesssim \mu(3J)\log(\frac{\ell(I_{j,J})}{\ell(J)})
% \lesssim 
% \mu(3J)\Big(
% \frac{\ell(I_{j,J})}{\ell(J)}
% \Big)^{\epsilon}
% \\
% &\lesssim 3^\epsilon 
% \mu(3J)
% \frac{\ell(I_{j,J})^\epsilon }{\ell(3J)^{\epsilon}}
% \lesssim 
% \mu(3J)
% \frac{\ell(I_{j,J})^\epsilon }{\mu(3J)^{\epsilon}}
% \end{align*}
% With this we estimate as follows:

Now we use Fubini's theorem to switch the order of summation in $T_2$.
%\begin{align}\label{reparameter}
%T_{2}=
%\sum_{J\in \mathcal{D}(Q)_{N}}
%\sum_{j=1}^{3^n}\langle f\rangle_{I_{j,r_j}} 
%\langle g\rangle_J
%\langle T\mathbbm{1}_{I_{j,r_j}},\mathbbm{1}_J \rangle
%\end{align}
%in terms of the cubes $I_{j,r_j}$ instead of the cubes $J$. 
For this, 
%we remind that 
%%in \eqref{reparameter}, 
%for each $J\in \mathcal{D}(Q)_{N}$,  
%%each $J_j\in J^{\rm fr}$ with $j\in \{1,\ldots , 3^n\}$ 
%and each scale
%$r\in \{ 0,\ldots , N_0\}$,
%%let 
%%$\mathcal J_k$ be the family of cubes 
%%$J\in \mathcal{D}(Q)_{N}$ such that 
%there is an associated cube 
%$I_{j,r_j}\in J_{\theta}$ with side length 
%$\ell(J_{j,r_j})=2^{-r_j}\ell(Q)$. 
%%In that case 
%%$\ell(J)=2^{-N}\ell(S)=2^{j-N}\ell(I_{j,J})$. 
%Now 
we re-parametrize the cubes we have up to now denoted by $I_{j,r_j}$ in the following way: 
for each scale $r\in \{ 0,\ldots , N_0\}$ and 
each $i\in \{ 1,\ldots , 2^{rn}\}$, 
we denote by $I^{i,r}$ the cubes 
such that $\ell(I^{i,r})=2^{-r}\ell(Q)$. 
We note that inside $Q$, for each 
$r\in \{ 0,\ldots , N_0\}$ there are in total $2^{rn}$ of such cubes.
%$I_{j,k}$ with 
%$\ell(I_{j,k})=2^{-k}\ell(Q)$. 
%Then, we can 
Now, for each $I^{i,r}$ 
we define $\mathcal J^{i,r}$ as the family of cubes 
$J\in \mathcal D(Q)_N$ such that 
%there exists  
%$J'\in J^{\rm fr}$ 
%associated with 
$J\in (I^{i,r})_{\theta}$. 
%This means that %namely, 
%%satisfying 
%$J'\subset I_{i,k}$, 
%$\ell(I_{i,k})=2^{-k}\ell(Q)$ and 
%$J'\in (I_{i,k})_{\theta}$, that is, 
%$\dist(I_{i,k}, J')<
%(2^{|e|\theta}+1)\ell(J')$. 
%% in $\mathcal J_j$
%% associated with the same cube $I_{j,J}$, which we now denote by $I_{j}^i$. 
Finally, we denote 
$C^{i,r}=\bigcup_{J\in \mathcal J^{i,r}}3J$.
% We note that $\mu(3J)\leq \mu(C^{i,r})$. 
With this
\begin{align*}
T_{2}=\sum_{r=0}^{N_0-1}\sum_{i=0}^{2^{rn}} \sum_{J\in \mathcal J^{i,r}}
\langle f\rangle_{I^{i,r}} 
\langle g\rangle_J 
\langle T\mathbbm{1}_{I^{i,r}},\mathbbm{1}_J \rangle .
\end{align*}

Now, for fixed $I^{i,r}$, let $\mathcal{J}_{0,0}$, $\mathcal{J}_{1,0}$, $\mathcal{J}_{1,1}$ and $\mathcal{J}_{0,1}$ be  
the collection of cubes in $J\in \mathcal J^{i,r}$
such that the complex number $\langle T\mathbbm{1}_{I^{i,r}},\mathbbm{1}_J \rangle$ belongs to the first, second, third, and fourth quadrants respectively. We note that in $\mathcal{J}_{l_1,l_2}$ we have 
$|\Re\langle T\mathbbm{1}_{I^{i,r}},\mathbbm{1}_J \rangle|=(-1)^{l_1}
\Re\langle T\mathbbm{1}_{I^{i,r}},\mathbbm{1}_J \rangle $ and 
$|\Im\langle T\mathbbm{1}_{I^{i,r}},\mathbbm{1}_J \rangle|=(-1)^{l_1}
\Im\langle T\mathbbm{1}_{I^{i,r}},\mathbbm{1}_J \rangle $. 
%$\Re(\langle T\chi_{I_{i,k}},\chi_J \rangle)\geq 0$, 
%$\Re(\langle T\chi_{I_{i,k}},\chi_J \rangle)<0$, 
%$\Im(\langle T\chi_{I_{i,k}},\chi_J \rangle)\geq 0$, and $\Im(\langle T\chi_{I_{i,k}},\chi_J \rangle)>0$ respectively. 
Let also $S_{l_1,l_2}$ the union of the cubes in $\mathcal{J}_{l_1,l_2}$. Finally, we define 
$$
\tilde S
=\bigcup_{r=0}^{N_0-1}\bigcup_{i=0}^{2^{rn}}\bigcup_{J\in \mathcal{J}^{i,r}}J,
$$ 
that is, the union of 
all cubes $J\in \mathcal{D}(Q)_{N}$ such that 
$J\in (I^{i,r})_\theta$ for some $i,r$. 
We note that $S_{l_1,l_2}\subset \tilde S$ and 
$$
\tilde S
=\bigcup_{k=0}^{N_0-1}\bigcup_{i=0}^{2^{rn}}\bigcup_{l_1,l_2\in \{0,1\}}\bigcup_{J\in \mathcal{J}_{l_1,l_2}}J.
$$

Before continuing, we remind that in the decomposition obtained in \eqref{fdecom} 
we are first working on estimates for 
$\langle TP_{Q}^Nf,g_1\rangle $. In this case, the cubes $J\in \tilde{\mathcal D}$ and so, they are open cubes. Therefore, 
$C^{i,r}$ and $\tilde S$ are open sets and they satisfy 
by the choice of $N$ in \eqref{Small} that 
$\mu(\tilde S)$ is sufficiently small. 
With all this, we have 
\begin{align*}
|T_{2}|&\leq \sum_{r=0}^{N_0-1}\sum_{i=0}^{2^{rn}} \sum_{J\in \mathcal J^{i,r}}
\langle f\rangle_{I^{i,r}} 
\langle g\rangle_J 
|\langle T\mathbbm{1}_{I^{i,r}},\mathbbm{1}_J \rangle|
% \\
% &\leq  \| f\|_{L^{\infty }(\mu)}\| g\|_{L^{\infty }(\mu)}
% \sum_{k=0}^{N_0-1}\sum_{i=0}^{2^{kn}}
% \sum_{J\in \mathcal J_{i,k}}
% |\langle T\chi_{J_{i,k}},\chi_J \rangle|
\\
&\lesssim \| f\|_{L^{\infty }(\mu)}\| g\|_{L^{\infty }(\mu)}
\sum_{r=0}^{N_0-1}\sum_{i=0}^{2^{rn}}
\sum_{l_1,l_2\in \{0,1\}}\Big( \sum_{J\in \mathcal{J}_{l_1,l_2}}(-1)^{l_1}
\Re\langle T\mathbbm{1}_{I^{i,r}},\mathbbm{1}_J \rangle
\\
&
\hskip200pt +\sum_{J\in \mathcal{J}_{l_1,l_2}}(-1)^{l_1}\Im\langle T\mathbbm{1}_{I^{i,r}},\mathbbm{1}_J \rangle\Big)
\\
&=\| f\|_{L^{\infty }(\mu)}\| g\|_{L^{\infty }(\mu)}
\sum_{r=0}^{N_0-1}\sum_{i=0}^{2^{rn}}
\sum_{l_1,l_2\in \{0,1\}}
\big((-1)^{l_1}\Re\langle T\mathbbm{1}_{I^{i,r}},\mathbbm{1}_{S_{l_1,l_2}} \rangle
\\
&
\hskip195pt +(-1)^{l_1}\Im\langle T\mathbbm{1}_{I^{i,r}},\mathbbm{1}_{S_{l_1,l_2}} \rangle \big)
\\
&\lesssim \| f\|_{L^{\infty }(\mu)}\| g\|_{L^{\infty }(\mu)}
\sum_{r=0}^{N_0-1}\sum_{i=0}^{2^{rn}}
\sum_{l_1,l_2\in \{0,1\}}
|\langle T\mathbbm{1}_{I^{i,r}},\mathbbm{1}_{S_{l_1,l_2}}\rangle |.
\end{align*}

%Now we divide $S^{l_1,l_2}$ into $2n+1$ parts: 
%$S^{l_1,l_2}=\cup_{j=0}^{2n}S_j^{l_1,l_2}$, where 
%$S_j^{l_1,l_2}$ is the union of cubes $J\in \mathcal{D}(Q)_{N}$ such that 
%$J\in (I_{i,k})_\theta$ for some $i,k$, and there is $I_{i,k}^j\in I_{i,k}^{\rm fr}$ with $J\subset I_{i,k}^j$.
%This implies that $S_j^{l_1,l_2}\subset J_{i,k}^j$. 

% We can assume without loss of generality that 
% $I_{i,k}^0=I_{i,k}$, and so 
% we work first with $S_0^{l_1,l_2}$.  Even though $S^{l_1,l_2}_0$ is not a cube, 
% since $S^{l_1,l_2}_0\subset I_{i,k}$ and $I_{i,k}$ is a cube, we have from the testing condition 
% \begin{align*}
% |\langle T\chi_{I_{i,k}},\chi_{S^{l_1,l_2}} \rangle |
% &\leq \| \chi_{I_{i,k}} T\chi_{I_{i,k}}\|_{L^2(\mu)}
% \mu(S^{l_1,l_2})^{\frac{1}{2}}
% \\
% &
% \lesssim \mu(I_{i,k})^{\frac{1}{2}}F_\mu(I_{i,k})
% \mu(\tilde S)^{\frac{1}{2}}
% \\
% &
% \lesssim \mu(I_{i,k})^{\frac{1}{2}}
% \mu(\tilde S)^{\frac{1}{2}}. %(nope)
% \end{align*}

%We work with $S_j^{l_1,l_2}$ for every $j\in \{ 0, 1, \ldots, 2n\}$. 
By Lemma \ref{truncaT}, the truncated operator $T_{\gamma, Q}$ %and also $T_1$ are 
is bounded on $L^2(\mu)$ with bounds $\| T_{\gamma ,Q}\|_{2,2}
%, \| T_1\|_{2,2}
\leq 
\frac{\ell(Q)}{\gamma^\alpha }\leq 2^{N_0}$. Then, since $S_{l_1,l_2}\subset \tilde S$, we have
%
% \begin{align*}
% |\langle T\chi_{J_{i,k}},\chi_{S_i^{l_1,l_2}}\rangle |
% &\leq \mu(S_i^{l_1,l_2})^{\frac{1}{2}}
% \| \chi_{J_{i,k}^j} T\chi_{J_{i,k}}\|_{L^2(\mu)}
% \\
% &\lesssim \mu(S_i^{l_1,l_2})^{\frac{1}{2}}
% (\mu(S_i^{l_1,l_2})+\mu(J_{i,k}))^{\frac{1}{2}}.
% \end{align*}
\begin{align*}
|\langle T\mathbbm{1}_{I^{i,r}},\mathbbm{1}_{S_{l_1,l_2}}\rangle |
&\leq \| T\|_{2,2}
\mu(I^{i,r})^{\frac{1}{2}}
\mu(S_{l_1,l_2})^{\frac{1}{2}}
\leq 2^{N_0}
\mu(I^{i,r})^{\frac{1}{2}}
\mu(\tilde S)^{\frac{1}{2}}.
\end{align*}
With this,
% \begin{align*}
% |S_2|&\lesssim \sum_{J\in \mathcal{D}(Q)_{N}}
% |\langle f\rangle_{I_{j,J}}|
% |\langle g\rangle_J |
% |\langle T\chi_{I_{j,J}},\chi_{J} \rangle |
% \\
% &\lesssim \sum_{J\in \mathcal{D}(Q)_{N}}
% |\langle f\rangle_{I_{j,J}}|
% |\langle g\rangle_J |
% \mu(3J)
% \frac{\ell(I_{j,J})^\epsilon }{\mu(3J)^{\epsilon}}
% \\
% &\lesssim \Big(\sum_{J\in \mathcal{D}(Q)_{N}}
% |\langle g\rangle_J |^2\mu(3J)\Big)^{\frac{1}{2}}
% \Big(\sum_{J\in \mathcal{D}(Q)_{N}}
% |\langle f\rangle_{I_{j,J}} |^2
% \mu(3J)\frac{\ell(I_{j,J})^{2\epsilon }}{\mu(3J)^{2\epsilon}}
% \Big)^{\frac{1}{2}}
% \end{align*}
\begin{align*}
|T_2|&\lesssim \mu(\tilde S)^{\frac{1}{2}}\| f\|_{L^{\infty }(\mu)}\| g\|_{L^{\infty }(\mu)}
2^{N_0}\sum_{r=0}^{N_0-1}\sum_{i=0}^{2^{rn}}
\mu(I^{i,r})^{\frac{1}{2}}
%F_\mu(J_{i,k})
\\
&\lesssim \mu(\tilde S)^{\frac{1}{2}}\| f\|_{L^{\infty }(\mu)}\| g\|_{L^{\infty }(\mu)}
2^{N_0}
\mu(Q)^{\frac{1}{2}}
%\sup_{i,k}F_\mu(J_{i,k})
N_02^{N_0n}
\\
&
\lesssim \mu(\tilde S)^{\frac{1}{2}}\| f\|_{L^{\infty }(\mu)}\| g\|_{L^{\infty }(\mu)}
2^{N_0(n+3)}
%\mu(Q)^{\frac{1}{2}}
%\sup_{i,k}F_\mu(J_{i,k})
<\langle f\rangle_Q \langle g\rangle_Q \mu(Q). 
\end{align*}
In the second last inequality we used  $\mu(I^{i,r})\leq \mu(Q)\leq \rho (Q)\ell(Q)^\alpha \lesssim 
2^{N_0}$ and so, 
$\mu(Q)^{\frac{1}{2}}\leq 2^{N_0}$. 
The last inequality holds because 
$\tilde S\subset C_N$ and from the choice of $N$ in \eqref{Small}.

All this work finally proves Proposition \ref{resultattached}, and thus the estimate \eqref{estsquare} of Proposition \ref{dualbysquare} for 
$\langle TP_{Q}^Nf,g_1\rangle $, the first term in \eqref{fdecom}. 

To deal with the second term in \eqref{fdecom}, namely 
$\langle Tf_1,g_{1,\partial}\rangle $, we note first that the reasoning to estimate
$D$, $C$, $N_i$, and $P_i$ can be applied unchanged to this new case. For the term $A-G_1,-G_2$, we implement a small change. Since the expression is completely symmetrical with respect the cubes $I$ and $J$, we can switch the roles played by these cubes,
\begin{align*}
A-G_1-G_2
&=\sum_{\substack{I\parallel J\\ \rdist(I,J)>0}}
\langle f,h_{I}\rangle 
\langle g, h_{J}\rangle
\langle Th_I,h_J\rangle
\\
&
+\sum_{\substack{I\Bumpeq J\\ 
%\ell(J)\leq \ell(I)\\ 
\widehat{J}\subset \widehat{I}}}
\langle f, h_{I}\rangle \langle g, h_{J}\rangle
\langle T(h_I-h_{I}^{\widehat{J}, Q}), h_{J}\rangle
+\sum_{\substack{I\Bumpeq J\\ 
%\ell(J)\leq \ell(I)\\ 
\widehat{I}\subset \widehat{J}}}
\langle f, h_{I}\rangle \langle g, h_{J}\rangle
\langle Th_{I}, h_J-h_{J}^{\widehat{I}, Q}\rangle 
\\
&=H_1+H_2+H_3. 
\end{align*}

$$
A=\sum_{J\in \mathcal{D}(Q)_{\geq N}}\sum_{I\in J_{\theta}}
\langle f,h_I\rangle \langle g,h_J\rangle
\langle Th_I,h_J \rangle .
$$
We now add and subtract the term
$$
{\rm Ad}=\sum_{J\in \mathcal{D}(Q)_{\geq N}}\sum_{I\in J_{\rm over}}
\langle f,h_I\rangle \langle g,h_J\rangle
\langle Th_I,h_J \rangle + {\rm Ad}_1+{\rm Ad}_2.
$$
where ${\rm Ad}_i$ are exactly the same terms as before containing the siblings of each cube. This terms  satisfies $|{\rm Ad}|\lesssim \Lambda_{\mathcal S, a}(f, g)$
and we rewrite previous reasoning to obtain 
\begin{align*}
   & |A-{\rm Ad}|\lesssim \sum_{I\in \mathcal{D}(Q)_{N}}
\langle f\rangle_I
\langle T\mathbbm{1}_I,\sum_{J\in \mathcal{D}(Q)_{\geq N}\cap I_{\theta}}
\langle g,h_J\rangle h_J \rangle 
\\&
=\sum_{I\in \mathcal{D}(Q)_{N}}
\sum_{i=1}^{3^n}\langle f\rangle_I 
\langle g\rangle_{J_{i,R_i}}
\langle T\mathbbm{1}_I,\mathbbm{1}_{J_{i, R_i}} \rangle 
%\\&
-\sum_{I\in \mathcal{D}(Q)_{N}}
\sum_{i=1}^{3^n}\langle f\rangle_I 
\langle g\rangle_{J_{i,r_i}}
\langle T\mathbbm{1}_I,\mathbbm{1}_{J_{i,r_i}} \rangle
\\&=T_{1}
%+S_{1,2}-(S_{2,1}+
-T_{2},
\end{align*}
in similar way as we did before. %Again $T_{1}=0$, 
% while $S_{1,2}$ 
% and $S_{2,2}$, can be treated with the same reasoning we developed in the previous case using now that 
% $I\in \mathcal D(Q)_{N}$ and 
% $0<\epsilon <(\| g\|_{L^{\infty }(\mu)}\mu(2Q)^{\frac{1}{2}})^{-6}$.  
%Finally, 
We can reparametrize the sums in, say, $T_2$ as before, 
to write
\begin{align*}
    |T_{2}|&\lesssim \| f\|_{L^{\infty }(\mu)}\| g\|_{L^{\infty }(\mu)}
\sum_{r=0}^{N_0-1}\sum_{j=0}^{2^{rn}}
\sum_{l_1,l_2=0}^1
|\langle T\mathbbm{1}_{S^{l_1,l_2}}, \mathbbm{1}_{J^{j,r}}\rangle |.
\end{align*}
Now we note that the cubes $I\in \tilde{\mathcal D}$ are open and so 
$N$ can be chosen large enough so that 
\begin{align*}
|T_2|\lesssim \mu(\tilde S)^{\frac{1}{2}}\| f\|_{L^{\infty }(\mu)}\| g\|_{L^{\infty }(\mu)}
2^{N_0(n+3)}
%\mu(Q)^{\frac{1}{2}}
%\sup_{i,k}F_\mu(J_{i,k})
<\langle f\rangle_Q \langle g\rangle_Q\mu(Q).
\end{align*}
This ends the estimates of 
Proposition \ref{resultattached} and Proposition \ref{dualbysquare} for $\langle Tf_1,g_{1,\partial}\rangle $.  

The remaining part is dealt with an iteration process developed in the next section. 

\end{proof}

\section{An iteration process to end the proof of Proposition \ref{dualbysquare}}

For the last term in \eqref{fdecom}, that is
$\langle Tf_{1,\partial},g_{1,\partial}\rangle $, we reason by reiteration. 
We first note that the supports of $f_{1,\partial}$ and $g_{1,\partial}$ are contained in the union of 
$\partial I$ for all $I\in \mathcal{D}^1(Q)_{\geq N}$. % with $\ell(I)\geq 2^{-N}\ell(Q)$. 
This set, which we denote by $\partial \mathcal{D}^1(Q)_{\geq N}$, consists of the intersection with $Q$ of finitely many euclidean affine spaces of dimension $n-1$, which are either pairwise parallel or pairwise perpendicular.  

Let now $\mathcal D=\mathcal D^2=a_2+\mathcal{D}_1=\{ a_2+I :  
I\in \mathcal{D}_1\}$ defined in $\eqref{grids}$. We consider the families of cubes $\mathcal{D}^2(Q)_{\geq N}$ 
%and $\tilde{\mathcal{D}}^2(Q)$, 
and decompose $f_{1,\partial}=f_2+f_{2,\partial}$ 
where $f_{2,\partial}=f_{1,\partial}\mathbbm{1}_{\partial \mathcal{D}^2(Q)_{\geq N}}$ and similar for $g_{1,\partial}$. Now, using the Haar wavelet system 
$(h_I)_{I\in \mathcal D^2}$ 
%and $(\psi_I)_{I\in \tilde{\mathcal D}^2}$,  
we decompose as before:
\begin{align}\label{fdecom2}
\langle Tf_{1,\partial},g_{1,\partial}\rangle 
&=\langle Tf_{1,\partial},g_2\rangle 
+\langle Tf_2,g_{2,\partial}\rangle 
%\\&
+\langle Tf_{2,\partial},g_{2,\partial}\rangle . 
\end{align}
Then we can apply all previous work developed in the first part of the proof of Proposition \ref{dualbysquare} to estimate the first two terms. 

For the third term, we note that the supports of $f_{2,\partial}$ and $g_{2,\partial}$ are now contained in $\partial \mathcal{D}^1(Q)_{\geq N}\cap \partial \mathcal{D}^2(Q)_{\geq N}$. Moreover, $\partial \mathcal{D}^2(Q)_{\geq N}$ consists of the intersection with $Q$ of finitely many euclidean affine spaces of dimension $n-1$, which are either pairwise parallel or pairwise perpendicular, and also either parallel or perpendicular
to every affine space of dimension $n-1$ of 
$\partial \mathcal{D}^1(Q)_{\geq N}$. Then 
$\partial \mathcal{D}^1(Q)_{\geq N}\cap \partial \mathcal{D}^2(Q)_{\geq N}$ is 
%,  with $A_0=0$. 
a set consisting of finitely many euclidean affine spaces of dimension $n-2$ intersected with $Q$. 

Then, by repeating the same argument $k=n-[\alpha ]+\delta(\alpha -[\alpha])$ times in total, we obtain   
$P_{M_N}f=\sum_{i=1}^{k}f_i+f_{i,\partial}$ and similar for $P_{M_N}g$ such that the appropriate estimates hold for $|\langle Tf_i, \cdot \rangle |$ and $|\langle \cdot , T^*g_i\rangle |$ for all 
with $i\in \{1, \ldots ,k\}$  and 
the functions $f_{k,\partial }, g_{k,\partial }$,  are supported on 
$\bigcap_{i=1}^{k}\partial \mathcal{D}^i(Q)_{\geq N}$. 
By repeating previous reasoning on parallel and perpendicular affine spaces, we conclude that this set consists of the intersection with $Q$ of finitely many euclidean affine spaces of dimension $n-k=[\alpha]-\delta(\alpha -[\alpha])$, which are either pairwise parallel or pairwise perpendicular.

But now we can show that $\bigcap_{i=1}^{k}\partial \mathcal{D}^i(Q)_{\geq N}$  %a set consisting of finitely many euclidean spaces of dimension zero that is, 
has measure zero with respect to $\mu$. 
%is a discrete set of points and so, their measure is zero. %One last reiteration, leaves us with a  
%  measure $\mu_{n+1}$ which is supported on 
% $\bigcap_{i=0}^{n+1}\partial \mathcal{D}_{A_i}(S)_{N_i}$. Now this set is empty and so, the remaining measure is zero. 
Let $I$ be an arbitrary $n-k$ dimensional dyadic cube with side length $\ell(I)$. Let 
$(J_i)_{i=1}^{m}$ be a 
%It can be covered by a 
family of pairwise disjoint $n$-dimensional cubes $J_i$ with fixed side length $r>0$ such that 
$J_i\cap I\neq \emptyset$ and $I\subset \bigcup_{i}J_i$. This family has 
%The family of $n$-dimensional cubes can be chosen to  
cardinality  
$m=(\frac{\ell(I)}{r})^{n-k}$. Then 
$$
\mu(I)\leq \sum_{i=1}^{m}\mu(J_{i})
\lesssim (\frac{\ell(I)}{r})^{n-k}r^{\alpha}
=\ell(I)^{n-k}r^{\alpha -n+k}. 
$$
Since 
$\alpha -n+k=\alpha-[\alpha]+\delta(\alpha -[\alpha])>0$,
we have  
$$
\mu(I)\lesssim \ell(I)^{n-k}
\lim_{r\rightarrow 0}r^{\alpha -n+k}=0
$$ 
for all cubes $I$ of dimension $n-k$. This shows that $\mu(\bigcap_{i=1}^{k}\partial \mathcal{D}^i(Q)_{\geq N})=0$ and so, $\langle Tf_{k,\partial},g_{k,\partial}\rangle=0$. This fully 
finishes the proof of the Proposition \ref{dualbysquare}. 

%------
%
%For $N\in \mathbb N$, 
%let 
%$\mathcal{D}(Q)_{\geq N}=\{ I\in \mathcal{D}(Q)/ \ell(I)\geq 2^{-N}\ell(Q)\}$.  
%%and $\mathcal{D}(Q)_{N}=\{ I\in \mathcal{D}_M^c(Q)/ \ell(I)=2^{-N}\ell(Q)\}$.
%For $I\in \mathcal{D}(Q)_{\geq N}$, let 
%$I_{\theta}$ be the family of cubes $J\in \mathcal{D}(Q)_{\geq N}$
%for which 
%$1\leq k\leq  2^{\theta |e|}$.
%
%
%
%Then we can rewrite $B$ as
%$$
%B=\sum_{I\in \mathcal{D}(Q)_{\geq N}}\sum_{J\in I_{\theta}}
%\langle f,h_I\rangle \langle g,h_J\rangle
%\langle Th_I,h_J \rangle .
%$$
%
%For very 
%$\epsilon >0$ there exists 
%$N$ large enough so that 
%$$
%|B|<\epsilon
%$$
%
%The cubes in $B$ are contained in $C$, the cartesian product of an arbitrary dyadic cube with side length $\ell(I)$ and an interval of side length $\epsilon $. Then we cover this set with a family of cubes $(J_i)_{i=1}^m$ with side length $\epsilon $ and cardinality 
%$m=(\frac{\ell(I)}{\epsilon })^{n-1}$. Then 
%$$
%\mu(I)\leq \sum_{i=1}^{m}\mu(J_{i})
%\lesssim (\frac{\ell(I)}{\epsilon })^{n-1}\epsilon^{\alpha}
%=\ell(I)^{n-1}\epsilon^{\alpha -n+1}.
%$$
%With this 
%$$
%\mu(I)\lesssim \ell(I)^{n-1}
%\lim_{r\rightarrow 0}r^{\alpha -n+1}=0
%$$ 
%since $\alpha >n -\delta \geq n-1$. 
%Finally then 
%$$
%|B|\leq \| T_\gamma \|\mu(I)^{\frac{1}{2}}\mu(J)^{\frac{1}{2}}
%$$

\section{Boundedness of the square functions on $L^2$ and weak $L^1$}

\begin{proposition}
For $i\in \{1,\ldots , k\}$ and $j\in \{1,2, 3\}$, let $S_{j,i}^1$ and $S_{j,i}^{-1}$ be the square functions of Definition \ref{square}, which we simply denote by $S_j$ and $S_j'$ respectively. Then 
$S_j, S_j'$  are bounded on $L^2(\mu)$ and 
\begin{align*}
\|S_jf\|_{L^2(\mu)}+\|S_j'f\|_{L^2(\mu)}
\lesssim  a^{\frac{1}{2}} 
\|f\|_{L^2(\mu)}, 
\end{align*}
with $a =\sup_{\substack{m\in \mathbb N\\e\in \mathbb Z}}\sup_{I\in \mathcal D}(1+\rho(2^eI))(1+\rho(I))\sup_{J\in I_{e,m}}F(I,J)$.
% and $a_{I,i}=\sup_{J\in I_{e,m}}F(I,J)$.  
% \begin{align*}
% \|S_2f\|_{L^2(\mu)}
% \lesssim 
% \|f\|_{L^2(\mu)}
% \end{align*}
\end{proposition}
\begin{proof}
As in the classical case, boundedness of these  square functions follows from Plancherel's inequality. 
The proofs for $S_2f$, $S_2'f$, and $S_3'f$ are straightforward:  
\begin{align*}
\|S_{2}'f\|_{L^2(\mu)}^2
&=\| \Big( \sum_{J\in \mathcal{D}_i}
a_{J}\rho(2^eJ)
|\langle f,h_{J}\rangle |^2
\frac{\mathbbm{1}_{J}}{\mu(J)}
\Big)^{\frac{1}{2}}\|_{L^2(\mu)}^2
\\
&= \sum_{J\in \mathcal{D}_i}
a_{J}\rho(2^eJ)
|\langle f,h_{J}\rangle |^2\leq a\| f\|_{L^2(\mu)}^2. 
\end{align*}
Meanwhile, 
\begin{align*}
\|S_{2}f\|_{L^2(\mu)}^2
&=\| \Big( \sum_{I\in \mathcal{D}_i}
a_{I}
|\langle f,h_{I}\rangle |^2
\frac{\mu(I)}{\ell(I)^{2\alpha} }\mathbbm{1}_{3I}
\Big)^{\frac{1}{2}}\|_{L^2(\mu)}^2
= \sum_{I\in \mathcal{D}_i}
a_{I}
|\langle f,h_{J}\rangle |^2\frac{\mu(I)}{\ell(I)^{\alpha} }\frac{\mu(3I)}{\ell(I)^{\alpha}}
\\
&\lesssim \sum_{I\in \mathcal{D}_i}
a_{I}
|\langle f,h_{J}\rangle |^2\rho(I)\rho(3I)
\leq a\| f\|_{L^2(\mu)}^2. 
\end{align*}
And since $\mu(I)\leq \mu_{\epsilon}(I)$,  
\begin{align*}
\|\tilde S_{3}'f\|_{L^2(\mu)}^2
&=\| \Big( \sum_{I\in \mathcal{D}_i}
a_{I}
|\langle f,h_{I}\rangle |^2
\frac{\mathbbm{1}_{I'}}{\mu_{\epsilon}(I')}
\Big)^{\frac{1}{2}}\|_{L^2(\mu)}^2
\leq \sum_{I\in \mathcal{D}_i}
a_{I}
|\langle f,h_{J}\rangle |^2
\leq a\| f\|_{L^2(\mu)}^2. 
\end{align*}
The proofs for $S_1f$, $S_1'f$, and $S_3f$ require a bit more work. 
By using again that $\mu(J'_{I})\leq \mu_{\epsilon}(J'_{I})$, we have 
\begin{align*}
\|S_{1}f\|_{L^2(\mu)}^2
&=\|\Big( \sum_{\substack{m\in \mathbb Z\\ m\geq 2}}\frac{1}{m^{\alpha +\delta}}\sum_{I}
a_{I}
|\langle f,h_{I}\rangle |^2
\frac{1}{\ell(I)^{\alpha}}
\sum_{J \in I_{-e,m}}\frac{\mu(J)}{\mu_{\epsilon}(J_I')}\mathbbm{1}_{J_I'}\Big)^{\frac{1}{2}}\|_{L^2(\mu)}^2
\\
&
\leq \sum_{\substack{m\in \mathbb Z\\ m\geq 2}}\frac{1}{m^{\alpha +\delta}}\sum_{I}
a_{I}
|\langle f,h_{I}\rangle |^2
\frac{1}{\ell(I)^{\alpha}}
\sum_{J \in I_{-e,m}}\mu(J).
\end{align*}
Now we note that 
the cubes $J\in I_{-e,m}$ are pairwise disjoint and their union is 
$I_m=mI\setminus (m-1)I$. Then 
\begin{align*}
\|S_{1}f\|_{L^2(\mu)}^2
&
\leq \sum_{I}
a_{I}
|\langle f,h_{I}\rangle |^2
\frac{1}{\ell(I)^{\alpha}}
\sum_{\substack{m\in \mathbb Z\\ m\geq 2}}\frac{1}{m^{\alpha +\delta}}\mu(I_{m})
\\
&
\lesssim \sum_{I}
a_{I} \rho(I)
|\langle f,h_{I}\rangle |^2
\lesssim a\|f\|_{L^2(\mu)}^2. 
\end{align*}
by using Lemma \ref{Abel}, namely,  
$
\frac{1}{\ell(I)^{\alpha}}\sum_{\substack{m\in \mathbb Z\\ m\geq 2}}\frac{1}{m^{\alpha +\delta}}\mu(I_{m})
\lesssim \rho(I) 
$. 

To deal with $S_{1}'g$, we denote $J_e\in \mathcal D$ such that $J\subset J_e$ and $\ell(J_e)=2^e\ell(J)$. Then
\begin{align*}
\|S_{1}'g\|_{L^2(\mu)}^2
&= \|\Big( \sum_{\substack{m\in \mathbb Z\\ m\geq 2}}\frac{1}{m^{\alpha +\delta}}\sum_{J}
b_{J}
|\langle g,h_{J}\rangle |^2
\frac{1}{\ell(J_e)^{\alpha}}
\sum_{I\in J_{e,m}}\frac{\mu(I)}{\mu_{\epsilon}(J_I')}\mathbbm{1}_{J_I'}\Big)^{\frac{1}{2}}\|_{L^2(\mu)}^2
\\
&
\leq \sum_{\substack{m\in \mathbb Z\\ m\geq 2}}\frac{1}{m^{\alpha +\delta}}\sum_{J}
b_{J}
|\langle g,h_{J}\rangle |^2
\frac{1}{\ell(J_e)^{\alpha}}
\sum_{I \in J_{e,m}}\mu(I). 
\end{align*}
Similarly as before, the cubes $J\in I_{-e,m}$ are pairwise disjoint and their union is 
$(J_e)_m=mJ_e\setminus (m-1)J_e$. Then 
\begin{align*}
\|S_{1}'g\|_{L^2(\mu)}^2
&
\leq \sum_{J}
b_{J}
|\langle g, h_{J}\rangle |^2
\frac{1}{\ell(J_e)^{\alpha}}
\sum_{\substack{m\in \mathbb Z\\ m\geq 2}}\frac{1}{m^{\alpha +\delta}}\mu((J_e)_m)
\\
&
\lesssim \sum_{J}
b_{J} \rho(J_e)
|\langle g,h_{J}\rangle |^2
\lesssim b\|g\|_{L^2(\mu)}^2, 
\end{align*}
where we used Lemma \ref{Abel}. 
Finally we work with $S_3f$. 
\begin{align*}
\|S_3f\|_{L^2(\mu)}^2
&=\| \Big(\sum_{I\in {\mathcal D}_i}
a_{I}
|\langle f,h_{I}\rangle |^2
%\frac{\mathbbm{1}_{I}(x)}{\mu(I)}
\sum_{R\in \{I,\hat{I}\}}
\sum_{k=2^{\theta e}}^{2^{e}}\sum_{J\in I_{-e,1,k}}
\frac{\mu(R\cap J)}{\mu(R)}
\frac{1}{\mu_{\epsilon}(J_I')}\mathbbm{1}_{J'_I)}\Big)^{\frac{1}{2}}\|_{L^2(\mu)}^2
\\
&\leq \sum_{I\in {\mathcal D}_i}
a_{I}
|\langle f,h_{I}\rangle |^2
%\frac{\mathbbm{1}_{I}(x)}{\mu(I)}
\sum_{R\in \{I,\hat{I}\}}\frac{1}{\mu(R)}
\sum_{k=2^{\theta e}}^{2^{e}}\sum_{J\in I_{-e,1,k}}\mu(R\cap J). 
\end{align*}
We note that $\mu(R\cap J)\neq 0$ only if 
$J\subset R\subset \widehat{I}$. Moreover,   
the cubes $J\in \cup_{k=2^{\theta e}}^{2^{e}}I_{-e,1,k}$ are pairwise disjoint and their union is $\widehat{I}$. Then
$$
\sum_{k=2^{\theta e}}^{2^{e}}\sum_{J\in I_{-e,1,k}}\mu(R\cap J)
=\sum_{k=2^{\theta e}}^{2^{e}}\sum_{J\in I_{-e,1,k}}\mu(J)\leq \mu(R)
$$
With this, 
\begin{align*}
\|S_3f\|_{L^2(\mu)}^2&\lesssim \sum_{I\in {\mathcal D}_i}
a_{I}
|\langle f,h_{I}\rangle |^2\leq a
\|f\|_{L^2(\mu)}^2. 
\end{align*}
\end{proof}

To extend the result to weak $L^1$ we use an adaptation of the Calder\'on-Zygmund decomposition to non-doubling measures. %This decomposition is described in \cite{CAP2019}*{Theorem 4.2} and is related to the decomposition of \cite{LSMP2014}*{Theorem 2.1}.
%
%\begin{lemma}
%\label{CZDecomposition}
%Let $\mu$ be a Radon measure. Let $f \in L^1(\mu)$ be  nonnegative and $\lambda>0$.  
%%(or $\lambda >\frac{\|f\|_{L^1(\mu)}}{\|\mu\|}$ if $\mu$ is a finite measure), 
%Then 
%there exist a collection of pairwise disjoint dyadic cubes $\mathcal Q$ such that 
%$$
%\displaystyle{\sum_{Q\in \mathcal Q}\mu(Q)\leq \lambda^{-1}\|f\|_{L^1(\mu)}},
%$$ 
%and the decomposition
%$$
%    f%=g+b
%    =g+\sum_{Q\in \mathcal Q}b_Q,
%$$
%with $b_Q=f\mathbbm{1}_Q-\langle f\mathbbm{1}_Q\rangle_{\widehat{Q}}\mathbbm{1}_{\widehat{Q}}$ satisfies
%\begin{enumerate}
%    \item\label{goodfunction} $\|g\|_{L^{2}(\mu)}^2\lesssim \lambda^{-1} \|f\|_{L^1(\mu)}$, %for $1\leq p<\infty$, 
%    \item $\operatorname{supp}b_Q \subseteq \widehat{Q}$, 
%    $\int_{\mathbb{R}^n}b_Q\,d\mu=0$ for each $j$ and $\displaystyle{\sum_{Q\in \mathcal Q}\|b_Q\|_{L^1(\mu)} \lesssim \|f\|_{L^1(\mu)}}$. 
%\end{enumerate}
%\end{lemma}

%We now prove the
%following weak $L^1$ type estimate. 
\begin{proposition}\label{weakL1} For $i\in \{1,\ldots , k\}$, $j\in \{1,2, 3\}$, and $s_e\in \{1, -1\}$, let $S_{j,i}^{s_e}$ be the square functions of Definition \ref{square}. 
%, which we simply denote by $S_j$ and $S_j'$ respectively. 

Let %$a_{I}=\sup_{J\in I_{e,m}}F(I,J)$, and 
$a =\sup_{\substack{m\in \mathbb N\\e\in \mathbb Z}}\sup_{I\in \mathcal D}(1+\rho(2^eI))^2(1+\rho(I))^2\sup_{J\in I_{e,m}}F(I,J)$. 
%Let $a=\sup{I\in ?}a_I^{\frac{1}{2}}(1+\rho(I))$. $a =\sup_{I\in \mathcal D(Q)}\tilde a_{I,i}(1+\rho(I))$, and $a_{I,i}=\sup_{J\in I_{e,m}}F(I,J)$.
Then 
%$$
%\mu(\{ S_{1}f(x)>\lambda \}) \lesssim %(2^{-k}n+1)
%\frac{a}{\lambda}\log{|\vec{m}|}\| f\|_{L^1(\mu)} 
%$$
%while 
$$
\mu(\{ S_{j,i}^{s_e}f(x)>\lambda \})\lesssim %(2^{-k}n+1)
\frac{a^{\frac{1}{2}}}{\lambda}
\| f\|_{L^1(\mu)} 
$$
for all $f$ $\mu$-integrable  
and $\lambda >0$. 
%with implicit constant independent of $f$ and $\lambda $.
\end{proposition}
\proof
%We assume that $f$ is non-negative. 
Consider the collection $\mathcal P$ of dyadic cubes $P\in \mathcal D_i$ that are maximal with respect
to set inclusion and such that
$$
\mu(3P)^{-1}\int_{P}|f(x)|d\mu(x) >\frac{\lambda }{a^{\frac{1}{2}}}.
$$
% If $Q\subsetneq I$, we have 
% $\mu(I)^{-1}\int_{I}|f(x)|d\mu(x) =
% \mu(I)^{-1}\int_{Q}|f(x)|d\mu(x)
% \leq \mu(Q)^{-1}\int_{Q}|f(x)|d\mu(x)
% =\mu(Q)^{-1}\|f\|_{1}
% $ and so the cubes 
% $I$ can be chosen so that $I\subset Q$.

%we
%define $E$ as
%the union of all $3I$ with $I\in {\mathcal I}$. 
%With this
% By maximality, the cubes in $\mathcal I$ are pairwise disjoint and so
% $$
% \mu(E)\leq \sum_{I\in \mathcal I}\mu(I)\leq \lambda^{-1}\sum_{I\in \mathcal I}\int_{I}|f(x)|d\mu(x) \leq \lambda^{-1}\|f\|_{L^1(\mu)}
% $$
Let $E$ be the union of all cubes $3P$ with $P\in {\mathcal P}$. 
We then define the Calder\'on-Zymund decomposition %of Lemma \ref{CZDecomposition} 
at level 
$\lambda /a_P$, that is $f=g+b$, given by
$$
    g=\sum_{P\in \mathcal P} \langle f\mathbbm{1}_{P}\rangle_{\widehat{P}}\, \mathbbm{1}_{\widehat{P}}+f\mathbbm{1}_{\mathbb{R}^n\setminus E}
\quad\text{and}\quad
b=\sum_{P\in \mathcal P}b_P,
$$
with $b_P=f\mathbbm{1}_{P}-\langle f\mathbbm{1}_{P}\rangle_{\widehat{P}}\, \mathbbm{1}_{\widehat{P}}$.

We now prove the estimate for all square functions, which we simply denote by $S_j$ when $s_e=1$ and $S_j'$ when $s_e=-1$. Since the first part of the proof is the same for all of them, we only make distinctions towards the end. 
First
\begin{align}\label{setmeasure}
\mu(\{ S_{j}f > \lambda\})
\leq \mu(\{ S_{j}g > \lambda/2\})
+\mu(\{ S_{j}b > \lambda/2\}).
\end{align}

Due the $L^2$ boundedness of $S_{j}$ and the fact that $g\lesssim \lambda /a^{\frac{1}{2}}$, we have
$$
\| S_{j}g \|_{L^2(\mu)}^{2}
\lesssim a \| g\|_{L^2(\mu)}^{2}
\lesssim a^{\frac{1}{2}} \lambda 
\| f\|_{L^1(\mu)}. 
$$
Then, by Chebyshev's inequality, we can control the first term in \eqref{setmeasure} as follows: 
$$
\mu(\{ S_{i}g > \lambda/2\})\lesssim \lambda^{-2}
\| S_{j}g\|_{L^2(\mu)}^2
\lesssim a^{\frac{1}{2}}\lambda^{-1} 
\| f\|_{L^1(\mu)}.
$$

We now work with $b=\sum_{P\in \mathcal P}b_P$, for which we have
%we
%define $E$ as
%the union of all $3I$ with $I\in {\mathcal I}$. 
%With this
$$
\mu(\{ S_{j}b (x)> \lambda/2\} ) \leq  \mu(E)+\lambda^{-1}\| S_{j}b\|_{L^1((\mathbb R^n\setminus E),\mu)}. 
$$

For the first term, we reason as follows.
%since $E$ is Ahlfors-regular
%$
%\mu(3I')\leq \ell(3I')^{\alpha}
%\lesssim \ell(I')^{\alpha}\lesssim \mu(I')
%$,
%or since $\mu$ is doubling
%$
%\mu(3I')\lesssim \mu(I').
%$
%Moreover, 
The cubes $P\in \mathcal P$ are pairwise disjoint by maximality. With this and the stopping condition 
$\lambda < a^{\frac{1}{2}}\mu(3P)^{-1}\int_{P}|f(x)|d\mu(x)$, 
%$\lambda <\langle f\rangle_{I'}$, 
we have 
$$
\mu(E)\leq 
\sum_{P\in \mathcal P} \mu(3P)
%\leq \sum_{I'\in \mathcal I} \mu(I')
<\sum_{P\in \mathcal P}
\frac{a^{\frac{1}{2}}}{\lambda}\|f\mathbbm{1}_{P}\|_{L^1(\mu)}\leq \frac{a^{\frac{1}{2}}}{\lambda}\|f\|_{L^1(\mu)}.
$$
% $$
% \mu(E)
% =\sum_{j=1}^{\infty}\mu(Q_j)\leq 
% \varepsilon \lambda^{-1} 
% \|f\|_{L^1(\mu)}. 
% %\lesssim 
% %\lambda^{-1}\|f \|_{L^1(\mu)}.
% $$ 
For the second term, we distinguish between the different square functions, starting again with the most straihtforward cases: 
$S_{2}$, $S_2'$, and $S_3'$.
%In the case of $S^i_{1,m}$ we reason as follows. For $\beta=\alpha +\delta -\frac{n-1}{2}$, 
% we are going to show that 
% $$
% \| S_{1}b \|_{L^1((\mathbb R^n \setminus E)\times \mathbb R^n,\tilde \mu)} \lesssim %2^{-k/2}
% \varepsilon  W(Q,Q')
% \| f\|_{L^1(\mu)}.
% $$

1) For $S_2$ we have 
\begin{align*}
\| S_{2}b \|_{L^1(\mathbb R^n \setminus E, \mu)} 
&\leq
\sum_{P\in \mathcal P}
\| S_{2}b_{P} \|_{L^1(\mathbb R^n \setminus E, \mu)} 
\\
%&
%=  \sum_{Q\in \mathcal Q}
%\Big\| \sum_{m\geq 2}\frac{1}{m^{\beta}} \Big( \sum_{I\in \mathcal D}
%\varepsilon_{I,1}
%|\langle b_{Q},h_{I}\rangle |^2
%\frac{\mu(I)}{\ell(I)^{2\alpha}}
%\mathbbm{1}_{I_m}
%\Big)^{\frac{1}{2}}\Big\|_{L^1(\mathbb R^n \setminus E, \mu)}
&= \sum_{P\in \mathcal P}
\Big\| \Big( \sum_{I\in \mathcal D}
a_{I}
|\langle b_{P},h_{I}\rangle |^2
\frac{\mathbbm{1}_{I}}{\mu(I)}
\Big)^{\frac{1}{2}}\Big\|_{L^1(\mathbb R^n \setminus E, \mu)}. 
\end{align*}

For each fixed cube $P\in \mathcal P$ and varying $I\in \mathcal{D}_i$,  
%with $Q\subsetneq \widehat{P}$, 
we consider the following cases:  
\begin{enumerate}
\item[(a)] when $\widehat{P}\cap \widehat{I}=\emptyset $, we have $\langle b_{P},h_I\rangle=0$ due to the disjoint supports of $b_P$ and $h_I$, 
\item[(b)] when $\widehat{P} \subsetneq \widehat{I}$, we have $\langle b_{P}, h_I\rangle=0$ since  $h_I$ is constant on $\widehat{P}$ and $b_{P}$ has mean value zero on $\widehat{P}$. %we have $x\notin \widehat{P}$ and so, $h_P(x)=0$,
\item[(c)] when $\widehat{I} \subsetneq \widehat{P}$, we have that $\widehat{I} \subseteq P'$ for some $P' \in \ch{(\widehat{P})}$. We consider two cases:
\begin{enumerate}
\item If $P' \neq P$, we have $\langle b_{P},h_I\rangle =0$ since 
$b_{P}$ is constant on $P'$ and $h_I$ has mean value zero on $\widehat{I}\subseteq P'$. 

\item If $P' =P$, we have $I\subset \widehat{I} \subseteq P\subset E$ and so 
$\mathbbm{1}_{I}(x)=0$ for all 
$x\notin E$. Then these cubes $I$ do not contribute to $\| S_{1}b_P \|_{L^1(\mathbb R^n \setminus E, \mu)}$. 
% since $\text{supp}(h_I) \subseteq \widehat{I}\subseteq Q$ and $x \not \in Q\subset E$.
\end{enumerate}
\end{enumerate}
The only case remaining is $\widehat{I}=\widehat{P}$. Then we have that $\langle b_{P},h_I\rangle\neq 0$ implies 
$I\in \child{P}$. 
With this, 
\begin{align*}
\| S_{2}b \|_{L^1(\mathbb R^n \setminus E, \mu)} 
&\leq \sum_{P\in \mathcal P}
\Big\| \Big( \sum_{I\in \mathcal \child{P}}
a_{I}
|\langle b_{P},h_{I}\rangle |^2
\frac{\mathbbm{1}_{I}}{\mu(I)}
\Big)^{\frac{1}{2}}\Big\|_{L^1(\mathbb R^n \setminus E, \mu)}
\\
&\leq  \sum_{P\in \mathcal P}
\Big\| \sum_{I\in \child{P}}
a_{I}^{\frac{1}{2}}
|\langle b_{P},h_{I}\rangle |
\frac{\mathbbm{1}_{I}}{\mu(I)^{\frac{1}{2}}}
\Big\|_{L^1(\mathbb R^n \setminus E, \mu)}
\\
&\leq a^{\frac{1}{2}} \sum_{P\in \mathcal P}
\sum_{I\in \child{P}}
|\langle b_{P},h_{I}\rangle |\mu(I)^{\frac{1}{2}}. 
\end{align*}

We now remind that 
$b_{P}=f\mathbbm{1}_{P}-\langle f\mathbbm{1}_{P}\rangle_{\widehat{P}}\, \mathbbm{1}_{\widehat{P}}$, and $h_{I}=\mu(I)^{\frac{1}{2}}\Big(
    \frac{1}{\mu(I)}\mathbbm{1}_{I}-\frac{1}{\mu(\widehat{I}\,)}\mathbbm{1}_{\widehat{I}}\Big)$. 
 Since   $\langle f\mathbbm{1}_{P}\rangle_{\widehat{P}}\, \mathbbm{1}_{\widehat{P}}$ is constant on $\widehat{I}$ and 
 $h_I$ has mean value zero on its support $\widehat{I}$,   
     we have 
\begin{align*}
    \langle b_{P} , h_I \rangle &=\langle f\mathbbm{1}_{P} , h_I \rangle 
    =\mu(I)^{\frac{1}{2}}(\langle f\mathbbm{1}_{P}\rangle_{I}-\langle f\mathbbm{1}_{P}\rangle_{\widehat{I}}).
%    \\
%    &
%    =
%     \mu(I)^{\frac{1}{2}}\int_{\mathbb{R}^n} f\mathbbm{1}_{Q}\frac{\mathbbm{1}_I}{\mu(I)}
%    -\langle f\mathbbm{1}_{Q}\rangle_{\widehat{Q}}
%    \frac{\mathbbm{1}_{I}}{\mu(I)}
%    -f\mathbbm{1}_{Q}\frac{\mathbbm{1}_{\widehat{I}}}{\mu(\widehat{I})}
%    +\langle f\mathbbm{1}_{Q}\rangle_{\widehat{Q}}
%    %\mathbbm{1}_{\widehat{I'}}
%    \frac{\mathbbm{1}_{\widehat{I}}}{\mu(\widehat{I})}
%    \,d\mu 
%    \\
%    &
%    =\mu(I)^{\frac{1}{2}}\Big(\frac{1}{\mu(I)}\int_{\mathbb{R}^n} f\mathbbm{1}_{I}\,d\mu -\langle f\mathbbm{1}_{Q}\rangle_{\widehat{Q}}
%    -\frac{1}{\mu(\widehat{I}\,)}\int_{\mathbb{R}^n} f\mathbbm{1}_{\widehat{I}}\,d\mu
%    +\langle f\mathbbm{1}_{Q}\rangle_{\widehat{Q}}\Big)
%\\
%&
%    =\mu(I)^{-\frac{1}{2}}\int_{\mathbb{R}^n} f\mathbbm{1}_{I}\,d\mu 
%    -\frac{\mu(I)^{\frac{1}{2}}}{\mu(\widehat{I}\,)}\int_{\mathbb{R}^n} f\mathbbm{1}_{\widehat{I}}\,d\mu
\end{align*}
Then, since $\mu(I)\leq \mu(\widehat{I}\,)$, we get 
%$$
%    |\langle b_{Q},h_I\rangle| \lesssim \mu(I)^{-\frac{1}{2}}\|f\mathbbm{1}_{Q}\|_{L^1(\mu)}.
%$$
\begin{align}\label{dual1Q}
    |\langle b_{P},h_I\rangle| &
    %\lesssim \mu(I)^{-\frac{1}{2}}\|f\mathbbm{1}_{I}\|_{L^1(\mu)}+\rho(I)^{\frac{1}{2}}\mu(\widehat{I}\, )^{-\frac{1}{2}}\|f\mathbbm{1}_{\widehat{I}}\|_{L^1(\mu)}
    %\\
    %&
    \lesssim \sum_{R\in \{ I,\widehat{I}\}}\mu(R)^{-\frac{1}{2}}\|f\mathbbm{1}_{P\cap R}\|_{L^1(\mu)} 
    \lesssim \mu(I)^{-\frac{1}{2}}\|f\mathbbm{1}_{P\cap \widehat{I}}\, \|_{L^1(\mu)}. 
\end{align}

With this %and $a_{I}^{\frac{1}{2}}\rho(J_I)\leq a^{\frac{1}{2}}(?)$ 
and the disjointness of the cubes $P\in \mathcal P$, we get  
\begin{align*}
\| S_{2}b \|_{L^1(\mathbb R^n \setminus E, \mu)} 
&
\lesssim a^{\frac{1}{2}} \sum_{P\in \mathcal P}
\|f\mathbbm{1}_{P}\|_{L^1(\mu)}
\leq a^{\frac{1}{2}} \|f\|_{L^1(\mu)}, 
\end{align*}
and hence
$$
\mu(\{ S_{2}b (x)> \lambda/2\} ) \leq  \frac{a^{\frac{1}{2}}}{\lambda} \| f\|_{L^1(\mu)}. 
$$

2) For $S_2'$ we can apply the same argument since all ideas previously used also hold for this operator. There is only point that needs a small modification: when $\widehat{I} \subseteq P \subset \widehat{P}$, we have $3I\subset 3\widehat{I} \subseteq 3P\subseteq E$ and so 
$\mathbbm{1}_{3I}(x)=0$ for all 
$x\notin E$. Then these cubes $I$ do not contribute to $\| S_{1}b_P \|_{L^1(\mathbb R^n \setminus E, \mu)}$, and again $\langle b_{P},h_I\rangle\neq 0$ implies $I\in \child{P}$. With this,
\begin{align*}
\| S_{2}'b \|_{L^1(\mathbb R^n \setminus E, \mu)} 
&\leq  \sum_{P\in \mathcal P}
\Big\| \Big( \sum_{I\in \mathcal D}
a_{I}
|\langle b_{P},h_{I}\rangle |^2
\frac{\mu(I)}{\ell(I)^{2\alpha}}\mathbbm{1}_{3I}
\Big)^{\frac{1}{2}}\Big\|_{L^1(\mathbb R^n \setminus E, \mu)}
\\
&\leq  \sum_{P\in \mathcal P}
\Big\| \sum_{I\in \child{P}}
a_{I}^{\frac{1}{2}}
|\langle b_{P},h_{I}\rangle |
\frac{\mu(I)^{\frac{1}{2}}}{\ell(I)^{\alpha}}
\mathbbm{1}_{3I}
\Big\|_{L^1(\mathbb R^n \setminus E, \mu)}
\\
&\leq a^{\frac{1}{2}} \sum_{P\in \mathcal P}
\sum_{I\in \child{P}}
|\langle b_{P},h_{I}\rangle |\mu(I)^{\frac{1}{2}}
\leq a^{\frac{1}{2}} \|f\|_{L^1(\mu)}.
\end{align*}

3) For $S_{3}'$ we can still apply the same reasoning: 
\begin{align*}
\|S_{3}'b \|_{L^1(\mathbb R^n \setminus E, \mu)} 
&\leq  \sum_{P\in \mathcal P}
\Big\| \Big( \sum_{I\in \mathcal D}a_{I}
|\langle b_P,h_{I}\rangle |^2
\frac{\mathbbm{1}_{I'}(x)}{\mu_{\epsilon}(I')}
\Big)^{\frac{1}{2}}\Big\|_{L^1(\mathbb R^n \setminus E, \mu)}
\\
&\leq  \sum_{P\in \mathcal P}
\Big\| \sum_{I\in \child(P)}a_{I}^{\frac{1}{2}}
|\langle b_P,h_{I}\rangle |
\frac{\mathbbm{1}_{I'}(x)}{\mu_{\epsilon}(I')^\frac{1}{2}}
\Big\|_{L^1(\mathbb R^n \setminus E, \mu)}
\\
&\leq  \sum_{P\in \mathcal P}
\sum_{I\in \child(P)}a_{I}^{\frac{1}{2}}
|\langle b_P,h_{I}\rangle |
\frac{\mu(I')}{\mu_{\epsilon}(I')^\frac{1}{2}}.
\end{align*}

Now, since $\mu(I')\leq \tilde  \mu(I')$ we have $\frac{\mu(I')}{\mu_{\epsilon}(I')^\frac{1}{2}}\leq \mu(I')^\frac{1}{2}$. With this, the inequality  
 $\mu(I')\leq \mu(I)$, and the fact that the cardinality of $\child{P}$ is $2^n$, we get
\begin{align*}
\|S_{3}'b \|_{L^1(\mathbb R^n \setminus E, \mu)}
&\lesssim a^{\frac{1}{2}}\sum_{P\in \mathcal P}
\sum_{I\in \child(P)}
\|f\mathbbm{1}_{P}\|_{L^1(\mu)}\mu(I)^{-\frac{1}{2}}
\mu(I')^\frac{1}{2}
\\
&\lesssim a^{\frac{1}{2}}\sum_{P\in \mathcal P}
\|f\mathbbm{1}_{P}\|_{L^1(\mu)}
\leq a^{\frac{1}{2}}
\|f\|_{L^1(\mu)}.
\end{align*}

4) For $S_1$, $S_1'$, and $S_3$ previous reasoning needs further modifications. In the case of $S_1$, we have as before 
\begin{align*}
\| S_{1}b &\|_{L^1(\mathbb R^n \setminus E, \mu)} 
\leq
\sum_{P\in \mathcal P}
\| S_{1}b_{P} \|_{L^1(\mathbb R^n \setminus E, \mu)} 
\\
%&
%=  \sum_{Q\in \mathcal Q}
%\Big\| \sum_{m\geq 2}\frac{1}{m^{\beta}} \Big( \sum_{I\in \mathcal D}
%\varepsilon_{I,1}
%|\langle b_{Q},h_{I}\rangle |^2
%\frac{\mu(I)}{\ell(I)^{2\alpha}}
%\mathbbm{1}_{I_m}
%\Big)^{\frac{1}{2}}\Big\|_{L^1(\mathbb R^n \setminus E, \mu)}
&= \sum_{P\in \mathcal P}
\Big\| \Big( \sum_{m\geq 2}\frac{1}{m^{\alpha +\delta}} \sum_{I\in \mathcal D}
%\sum_{I\in \mathcal D_{k,m}(I')}
a_{I}
|\langle b_{P},h_{I}\rangle |^2
\frac{1}{\ell(I)^{\alpha}}\sum_{J\in I_{-e,m}}
\frac{\mu(J)}{\mu_{\epsilon}(J'_I)}
\mathbbm{1}_{J'_I}
\Big)^{\frac{1}{2}}\Big\|_{L^1(\mathbb R^n \setminus E, \mu)}. 
\end{align*}

For a fixed cube $P\in \mathcal P$ and varying cubes $I\in \mathcal{D}_i$, we get as before that 
$\langle b_{P}, h_I\rangle=0$ if either $\widehat{P}\cap \widehat{I}=\emptyset $, or $\widehat{P} \subsetneq \widehat{I}$, 
or $\widehat{I} \subsetneq P'$ with $P' \in \ch{(\widehat{P})}$, $P'\neq P$. 

When $P' =P$, that is $\widehat{I} \subseteq P$,
we parametrize the cubes $I$ by their side-length: 
for $r\in \mathbb{N}$, we consider the collection of cubes $I\subset P$ such that $\ell(I)=2^{-r}\ell(P)$.

We note that $J\in I_{-e,m}$ implies $J\subset I_m=mI\setminus (m-1)I\subset mI$. 
Now, if $r>\log (2m)$, we have $\ell(I)=2^{-r}\ell(P)<(2m)^{-1}\ell(P)$ and so
$\ell(mI)= m\ell(I)\leq 2^{-1}\ell(P)<\ell(P)$. 
%Moreover, $I\subset Q\subset E$ and so $\ell(mI)^c\rangle )\leq 2\ell(I)+\dist(I,I_m)\leq 10^{-1}(2m^{-1}+1)\ell(Q)<\ell(Q)$. 
This inequality and the fact that 
$I\subset P$ imply that $J_{I}'\subset J\subset I_m\subset mI \subset 3P \subset E$. 

With this we get 
$\mathbbm{1}_{J_I'}(x)=0$ for all 
$x\notin E$ and so, these cubes $I$ do not contribute to $\| S_{1}b_P \|_{L^1(\mathbb R^n \setminus E, \mu)}$. 

We are then left with the cubes $I\in \mathcal D$ such that $\widehat{I}=\widehat{P}$ or 
$\widehat{I}\subseteq P$, with $\ell(I)=2^{-r}\ell(P)$, and 
$0\leq r\leq \log{m}$. We parametrize such cubes as follows: for $0\leq r\leq \log{m}$, we denote by $\mathcal P_{r}$ the collection of the cubes $I\in \mathcal D$ with $\widehat{I}\subset P$, $\ell(I)=2^{-r}\ell(P)$. 
Moreover, for every fixed $x\in \mathbb R^n\setminus E$ and every fixed $r\geq 0$ there is a quantity comparable to $m^{n-1}$ of cubes $I\in \mathcal P_{m,r}$ such that their corresponding cubes $I_m$ all satisfy $x\in I_m$.% (not needed?). 
Then
\begin{align*}
\| S_1b_{P}&\|_{L^1(\mathbb R^n \setminus E,\mu)}
\\
&= 
\Big\| \Big( \sum_{m\geq 2}\frac{1}{m^{\alpha +\delta}} \sum_{r=0}^{\log m}\sum_{I\in \mathcal P_{r}}
%\sum_{I\in \mathcal D_{k,m}(I')}
a_{I}
|\langle b_{P},h_{I}\rangle |^2
\frac{1}{\ell(I)^{\alpha}}\sum_{J\in I_{-e,m}}
\frac{\mu(J)}{\mu_{\epsilon}(J'_I)}
\mathbbm{1}_{J'_I}
\Big)^{\frac{1}{2}}\Big\|_{L^1(\mathbb R^n \setminus E, \mu)}
\\
&\leq  
\Big\| \sum_{m\geq 2}\frac{1}{m^{\frac{\alpha +\delta}{2}}} \sum_{r=0}^{\log m}\sum_{I\in \mathcal P_{r}}
%\sum_{I\in \mathcal D_{k,m}(I')}
a_{I}^{\frac{1}{2}}
|\langle b_{P},h_{I}\rangle |
\frac{1}{\ell(I)^{\frac{\alpha}{2}}}\sum_{J\in I_{-e,m}}
\frac{\mu(J)^{\frac{1}{2}}}{\mu_{\epsilon}(J'_I)^{\frac{1}{2}}}
\mathbbm{1}_{J'_I}\Big\|_{L^1(\mathbb R^n \setminus E, \mu)}
%\\
%&=  
%\sum_{m\geq 2}\frac{1}{m^{\frac{\alpha +\delta}{2}}} \sum_{r=0}^{\log m}\sum_{I\in \mathcal P_{r}}
%%\sum_{I\in \mathcal D_{k,m}(I')}
%a_{I}^{\frac{1}{2}}
%|\langle b_{P},h_{I}\rangle |
%\frac{1}{\ell(I)^{\frac{\alpha}{2}}}\sum_{J\in I_{-e,m}}
%\frac{\mu(J)^{\frac{1}{2}}}{\tilde \mu(J'_I)^{\frac{1}{2}}}\mu(J'_I)
\\
&\leq  
\sum_{m\geq 2}\frac{1}{m^{\frac{\alpha +\delta}{2}}} \sum_{r=0}^{\log m}\sum_{I\in \mathcal P_{r}}
%\sum_{I\in \mathcal D_{k,m}(I')}
a_{I}^{\frac{1}{2}}
|\langle b_{P},h_{I}\rangle |
\frac{1}{\ell(I)^{\frac{\alpha}{2}}}\sum_{J\in I_{-e,m}}\mu(J)^{\frac{1}{2}}
\mu(J'_I)^{\frac{1}{2}}. 
\end{align*}
We used in the last inequality that $\mu(J'_I)\leq \mu_{\epsilon}(J'_I)$. By using that 
$\mu(J'_I)\leq \frac{\mu(I)\mu(J)}{\mu(Q)}$,  and $I_{-e,m}\neq \emptyset$ only if
$m\leq \lambda=2^{N}\ell(P)$, 
we get
 \begin{align*} 
\| S_1b_{P}\|_{L^1(\mathbb R^n \setminus E,\mu)}
&\leq  
\sum_{m\geq 2}\frac{1}{m^{\frac{\alpha +\delta}{2}}} \sum_{r=0}^{\log m}\sum_{I\in \mathcal P_{r}}
%\sum_{I\in \mathcal D_{k,m}(I')}
a_{I}^{\frac{1}{2}}
|\langle b_{P},h_{I}\rangle | \frac{\mu(I)^{\frac{1}{2}}}{\mu(Q)^{\frac{1}{2}}}
\frac{1}{\ell(I)^{\frac{\alpha}{2}}}\sum_{J\in I_{-e,m}}\mu(J).
\end{align*}
Now we use the following facts: that the cubes $J\in I_{-e,m}$ are pairwise disjoint with union 
$I_m=(mI\setminus (m-1)I)\cap P$; 
$\ell(I)=2^{r}\ell(P)$ for $I\in \mathcal P_{m,r}$ and so it is independent of $m$; 
and Fubini's inequality. 
%Furthermore, we define $\mathcal P_r$ as the collection of cubes $I\in \mathcal P_{r}$ for all $m$. 
With all this, we write 
 \begin{align*}
\| S_1b_{P}\|_{L^1(\mathbb R^n \setminus E,\mu)}
&\leq  
\sum_{m\geq 2}\frac{1}{m^{\frac{\alpha +\delta}{2}}} \sum_{r=0}^{\log m}\sum_{I\in \mathcal P_{r}}
%\sum_{I\in \mathcal D_{k,m}(I')}
a_{I}^{\frac{1}{2}}|\langle b_{P},h_I\rangle|\mu(I)^{\frac{1}{2}} 
\frac{1}{\mu(Q)^{\frac{1}{2}}}
\frac{1}{\ell(I)^{\frac{\alpha}{2}}}\mu(I_m)
\\
&\leq  \sum_{r\geq 0}\sum_{I\in \mathcal P_{ r}}
%\sum_{I\in \mathcal D_{k,m}(I')}
a_{I}^{\frac{1}{2}}
|\langle b_{P},h_I\rangle|\mu(I)^{\frac{1}{2}} 
\frac{1}{\ell(I)^{\frac{\alpha}{2}}}\sum_{m=2^r}^{\lambda}\frac{1}{m^{\frac{\alpha +\delta}{2}}}\mu(I_m). 
\end{align*}

Since by $\mu(mI)\leq \mu(Q)$, by using Lemma \ref{Abel} again we have 
\begin{align*}
\frac{1}{\mu(Q)^{\frac{1}{2}}}
\frac{1}{\ell(I)^{\frac{\alpha}{2}}}\sum_{m= 2^r}^{\lambda }\frac{1}{m^{\frac{\alpha +\delta}{2}}}
\mu(I_m)
&\lesssim \frac{1}{\mu(Q)^{\frac{1}{2}}}
\frac{1}{\ell(I)^{\frac{\alpha}{2}}}
\sum_{m\geq 2^r} \frac{\mu(mI\cap Q)}{m^{\frac{\alpha}{2}} }\frac{1}{m^{\frac{\delta}{2}+1}}
\\
&\lesssim 
\sum_{m\geq 2^r} \frac{\mu(mI)^{\frac{1}{2}}}{\ell(mI)^{\frac{\alpha}{2}} }\frac{1}{m^{\frac{\delta}{2}+1}}
\\
&
= \sum_{m\geq 1}\frac{\mu((m+2^r)I)^{\frac{1}{2}}}{\ell((m+2^r)I)^{\frac{\alpha}{2}} }\frac{1}{(m+2^r)^{\frac{\delta }{2}+1}}
\\
&
\leq 2^{-\frac{\delta r}{4}}\sum_{m\geq 1}\frac{\mu((m+2^r)I)}{\ell((m+2^r)I)^{\alpha} }\frac{1}{(m+2^r)^{\frac{1}{2}(\frac{\delta }{2}+1)}}
\\
&
\lesssim 2^{\frac{-\delta r}{4}}\rho(I). 
\end{align*}
Then
\begin{align*}
\| S_1b_{P}\|_{L^1(\mathbb R^n \setminus E,\mu)}
&\lesssim 
\sum_{r\geq 0} 2^{-\frac{\delta r}{4}}
\sum_{I\in \mathcal P_{r}}
%\sup_{I\in \mathcal Q_{r,m}}
a_{I}^{\frac{1}{2}}\rho_{\textrm{out}}(I)
|\langle b_{P},h_I\rangle|\mu(I)^{\frac{1}{2}}
\\
&\lesssim a^{\frac{1}{2}}
\sum_{r\geq 0} 2^{-\frac{\delta r}{4}} 
\sum_{I\in \mathcal P_{r}}
%\sup_{I\in \mathcal Q_{r,m}}
|\langle b_{P},h_I\rangle|\mu(I)^{\frac{1}{2}}.
\end{align*}
Since $\widehat{I}\subseteq P$, inequality \eqref{dual1Q} stands as 
$ |\langle b_{P},h_I\rangle|\lesssim \mu(I)^{-\frac{1}{2}}\|f\mathbbm{1}_{\widehat{I}\cap P}\, \|_{L^1(\mu)}$. Moreover, 
for each fixed $r$, the cubes $I\in \mathcal P_{r}$ are such that the union of their parents $\widehat{I}$ is included in 
$4P$
With this, 
\begin{align*}
\| S_1b_{P}\|_{L^1(\mathbb R^n \setminus E,\mu)}
&\lesssim a^{\frac{1}{2}}
\sum_{r\geq 0} 2^{-\frac{\delta r}{4}} 
\sum_{I\in \mathcal P_{r}}
%\sup_{I\in \mathcal Q_{r,m}}
\|f\mathbbm{1}_{\widehat{I}\cap P}\, \|_{L^1(\mu)}. 
\end{align*}
We now group together all cubes with the same parent and use that, for fixed $r\geq 0$, the cubes in $K\in \mathcal D$ with $K\subseteq P$ and $\ell(K)=2^{-r+1}\ell(P)$ are pairwise disjoint, to write 
\begin{align*}
\sum_{I\in \mathcal P_{r}}
\|f\mathbbm{1}_{\widehat{I}\cap P}\, \|_{L^1(\mu)}
&=\sum_{\substack{K\in \mathcal D\\ K\subseteq P\\ \ell(K)=2^{-r+1}\ell(P)}}\sum_{\substack{I\in \mathcal D\\ \widehat{I}=K}}
\|f\mathbbm{1}_{K\cap P}\, \|_{L^1(\mu)}
\\&
\lesssim 
\sum_{\substack{K\in \mathcal D\\ K\subseteq P\\ \ell(K)=2^{-r+1}\ell(P)}}
\|f\mathbbm{1}_{K\cap P}\, \|_{L^1(\mu)}\leq \|f\mathbbm{1}_{P}\, \|_{L^1(\mu)}. 
\end{align*}
Then
\begin{align*}
\| S_1b_{P}\|_{L^1(\mathbb R^n \setminus E,\mu)}
&\lesssim a^{\frac{1}{2}}
\sum_{r\geq 0} 2^{-\frac{\delta r}{4}} \|f\mathbbm{1}_{P}\, \|_{L^1(\mu)}\lesssim a \|f\mathbbm{1}_{P}\, \|_{L^1(\mu)}. 
\end{align*}
Finally, since the cubes $P\in \mathcal P$ are pairwise disjoint by maximality, we  have 
\begin{align*}
\|S_{1}b \|_{L^1(\mathbb R^n \setminus E, \mu)}
&\lesssim a^{\frac{1}{2}}\sum_{P\in \mathcal P}
\|f\mathbbm{1}_{P}\|_{L^1(\mu)}
\leq a^{\frac{1}{2}}
\|f\|_{L^1(\mu)}.
\end{align*}

5) The reasoning for $S_{1}'b$ is very similar and so, we only sketch the calculations to highlight the few differences. 
We remind that we denote by $J_e\in \mathcal D$ the cube such that $J\subset J_e$ and $\ell(J_e)=2^e\ell(J)$. Then
\begin{align*}
\| S_{1}'b_{P}\|_{L^1(\mathbb R^n \setminus E,\mu)}
&\hspace{-.1cm}= \hspace{-.1cm}
\Big\| \Big( \sum_{m\geq 2}\frac{1}{m^{\alpha +\delta}} \sum_{r=0}^{\log m}\hspace{-.1cm}\sum_{J\in \mathcal P_{r}}
%\sum_{I\in \mathcal D_{k,m}(I')}
b_{J}
|\langle b_{P},h_{J}\rangle |^2
\frac{1}{\ell(J_e)^{\alpha}}\sum_{I}
\frac{\mu(I)}{\mu_{\epsilon}(J'_I)}
\mathbbm{1}_{J'_I}
\Big)^{\frac{1}{2}}\Big\|_{L^1(\mathbb R^n \setminus E, \mu)}
%\\
%&\leq  
%\Big\| \sum_{m\geq 2}\frac{1}{m^{\frac{\alpha +\delta}{2}}} \sum_{r=0}^{\log m}\sum_{I\in \mathcal Q_{r}}
%%\sum_{I\in \mathcal D_{k,m}(I')}
%a_{I}^{\frac{1}{2}}
%|\langle b_{Q},h_{I}\rangle |
%\frac{1}{\ell(I)^{\frac{\alpha}{2}}}\sum_{J}
%\frac{\mu(J)^{\frac{1}{2}}}{\tilde \mu(J'_I)^{\frac{1}{2}}}
%\mathbbm{1}_{J'_I}\Big\|_{L^1(\mathbb R^n \setminus E, \mu)}
%\\
%&\leq 
%\sum_{m\geq 2}\frac{1}{m^{\frac{\alpha +\delta}{2}}} \sum_{r=0}^{\log m}\sum_{J\in \mathcal P_{r}}
%%\sum_{I\in \mathcal D_{k,m}(I')}
%b_{J}^{\frac{1}{2}}
%|\langle b_{P},h_{J}\rangle |
%\frac{1}{\ell(J_e)^{\frac{\alpha}{2}}}\sum_{I}
%\frac{\mu(I)^{\frac{1}{2}}}{\tilde \mu(J'_I)^{\frac{1}{2}}}\mu(J'_I)
\\
&\leq  
\sum_{m\geq 2}\frac{1}{m^{\frac{\alpha +\delta}{2}}} \sum_{r=0}^{\log m}\sum_{J\in \mathcal P_{r}}
%\sum_{I\in \mathcal D_{k,m}(I')}
b_{J}^{\frac{1}{2}}
|\langle b_{P},h_{J}\rangle |
\frac{1}{\ell(J_e)^{\frac{\alpha}{2}}}\sum_{I}\mu(I)^{\frac{1}{2}}
\mu(J'_I)^{\frac{1}{2}}. 
\end{align*}
Since $\mu(J'_I)\leq \frac{\mu(I)\mu(J)}{\mu(Q)}$,  and $J_{e,m}\neq \emptyset$ only if
$m\leq \lambda=2^{N}\ell(P)$, 
we have
 \begin{align*}
\| S_1'b_{P}\|_{L^1(\mathbb R^n \setminus E,\mu)}
&\leq  
\sum_{m\geq 2}\frac{1}{m^{\frac{\alpha +\delta}{2}}} \sum_{r=0}^{\log m}\sum_{J\in \mathcal P_{r}}
%\sum_{I\in \mathcal D_{k,m}(I')}
b_{J}^{\frac{1}{2}}
|\langle b_{P},h_{J}\rangle | \frac{\mu(J)^{\frac{1}{2}}}{\mu(Q)^{\frac{1}{2}}}
\frac{1}{\ell(J_e)^{\frac{\alpha}{2}}}\sum_{I}\mu(I)
\\
&= 
 \sum_{r\geq 0}\sum_{J\in \mathcal P_{r}}
%\sum_{I\in \mathcal D_{k,m}(I')}
b_{J}^{\frac{1}{2}}
|\langle b_{P},h_{J}\rangle | \frac{\mu(J)^{\frac{1}{2}}}{\mu(Q)^{\frac{1}{2}}}
\frac{1}{\ell(J_e)^{\frac{\alpha}{2}}}\sum_{m=2^r}^{\lambda}\frac{1}{m^{\frac{\alpha +\delta}{2}}}\mu((J_e)_m), 
\end{align*}
with $(J_e)_m=(mJ_e\setminus (m-1)J_e)\cap P$. 
Since by $\mu((J_e)_m)\leq \mu(Q)$, by using Lemma \ref{Abel} again we have 
\begin{align*}
\frac{1}{\mu(Q)^{\frac{1}{2}}}
\frac{1}{\ell(J_e)^{\frac{\alpha}{2}}}\sum_{m= 2^r}^{\lambda }\frac{1}{m^{\frac{\alpha +\delta}{2}}}
\mu((J_e)_m)
&\lesssim \frac{1}{\mu(Q)^{\frac{1}{2}}}
\frac{1}{\ell(J_e)^{\frac{\alpha}{2}}}
\sum_{m\geq 2^r} \frac{\mu(mJ_e)}{m^{\frac{\alpha}{2}}}\frac{1}{m^{\frac{\delta}{2}+1}}
\\
&\lesssim \sum_{m\geq 2^r} \frac{\mu(mJ_e)^{\frac{1}{2}}}{\ell(mJ_e)^{\frac{\alpha}{2}} }\frac{1}{m^{\frac{\delta}{2}+1}}
%\\
%&
%= \sum_{m\geq 1}\frac{\mu((m+2^r)J_e)^{\frac{1}{2}}}{\ell((m+2^r)J_e)^{\frac{\alpha}{2}} }\frac{1}{(m+2^r)^{\frac{\delta }{2}+1}}
\\
&
\leq 2^{-\frac{\delta r}{4}}\sum_{m\geq 1}\frac{\mu((m+2^r)J_e)}{\ell((m+2^r)J_e)^{\alpha} }\frac{1}{(m+2^r)^{\frac{1}{2}(\frac{\delta }{2}+1)}}
\\
&
\lesssim 2^{\frac{-\delta r}{4}}\rho_{\textrm{out}}(J_e). 
\end{align*}
By applying the same reasoning we used before and $b_{I}^{\frac{1}{2}}\rho(J_I)
\leq b^{\frac{1}{2}}$, we gradually get 
\begin{align*}
\| S_1'b_{P}\|_{L^1(\mathbb R^n \setminus E,\mu)}
&\lesssim 
\sum_{r\geq 0} 2^{-\frac{\delta r}{4}}
\sum_{J\in \mathcal P_{r}}
%\sup_{I\in \mathcal Q_{r,m}}
b_{J}^{\frac{1}{2}}\rho_{\textrm{out}}(J_e)
|\langle b_{P},h_J\rangle|\mu(J)^{\frac{1}{2}}
\\
&\lesssim b^{\frac{1}{2}}
\sum_{r\geq 0} 2^{-\frac{\delta r}{4}} 
\sum_{J\in \mathcal P_{r}}
%\sup_{I\in \mathcal Q_{r,m}}
|\langle b_{P},h_J\rangle|\mu(J)^{\frac{1}{2}}
\lesssim b^{\frac{1}{2}} \sum_{r\geq 0} 2^{-\frac{\delta r}{4}}
\sum_{I\in \mathcal P_{r}}
\|f\mathbbm{1}_{\widehat{I}}\, \|_{L^1(\mu)}
\\
&
\lesssim b^{\frac{1}{2}}\sum_{r\geq 0} 2^{-\frac{\delta r}{4}}
\|f\mathbbm{1}_{P}\|_{L^1(\mu)}
\lesssim b^{\frac{1}{2}} \|f\mathbbm{1}_{P}\|_{L^1(\mu)}.
\end{align*}

Finally, by disjointness of the cubes $P\in \mathcal P$, 
\begin{align*}
\| S_1'b\|_{L^1(\mathbb R^n \setminus E, \mu)}
&\leq \sum_{P\in \mathcal P}
\| S_1'b_{P}\|_{L^1(\mathbb R^n \setminus E, \mu)}
%\lesssim \frac{a}{\lambda}%\log|\vec{m}| %\sum_{I'\in \mathcal I}\sum_{I\in \child{I'}}
%\|f\mathbbm{1}_{Q}\|_{L^1(\mu)}
%\\
%&
\lesssim b^{\frac{1}{2}}\sum_{P\in \mathcal P}
\|f\mathbbm{1}_{P}\|_{L^1(\mu)}
\leq b^{\frac{1}{2}}\|f\|_{L^1(\mu)}.
\end{align*}

6) It only remains to deal with $S_3$. We start by noting that, since $J'_I\subset J_I\subset I$, we can apply the same reasoning we used for $S_0$ to reduce the sum in cubes $I\in \mathcal D$, to the much smaller sum in cubes $I\in \child{P}$. Then
\begin{align*}
\| S_3b&\|_{L^1(\mathbb R^n \setminus E, \mu)} 
\\
&
\leq \sum_{P\in \mathcal P}
\Big\| \Big( \sum_{I\in {\mathcal D}_i}
a_{I}
|\langle b_P,h_{I}\rangle |^2
%\frac{\mathbbm{1}_{I}(x)}{\mu(I)}
\hspace{-.1cm}\sum_{R\in \{I,\hat{I}\}}
\sum_{k=2^{\theta e}}^{2^{e}}\sum_{J\in I_{-e,1,k}}
\hspace{-.1cm}
\frac{\mu(R\cap J)}{\mu(R)}
\frac{\mathbbm{1}_{J'_I}(x)}{\mu_{\epsilon}(J_I')}
\Big)^{\frac{1}{2}}\Big\|_{L^1(\mathbb R^n \setminus E, \mu)}
\\
&\leq \sum_{P\in \mathcal P}
\Big\|  \sum_{I \in \child(P)}
a_{I}^{\frac{1}{2}}
|\langle b_P,h_{I}\rangle |
%\frac{\mathbbm{1}_{I}(x)}{\mu(I)}
\sum_{R\in \{I,\hat{I}\}}
\sum_{k=2^{\theta e}}^{2^{e}}\sum_{J\in I_{-e,1,k}}
\Big(\frac{\mu(R\cap J)}{\mu(R)}\Big)^{\frac{1}{2}}
\frac{\mathbbm{1}_{J'_I}(x)}{\mu_{\epsilon}(J_I')^{\frac{1}{2}}}
\Big\|_{L^1(\mathbb R^n \setminus E, \mu)}
\\
&\leq a^{\frac{1}{2}}\sum_{P\in \mathcal P}
 \sum_{I\in \child(P)}
|\langle b_P,h_{I}\rangle |
%\frac{\mathbbm{1}_{I}(x)}{\mu(I)}
\sum_{R\in \{I,\hat{I}\}}
\frac{1}{\mu(R)^{\frac{1}{2}}}
\sum_{k=2^{\theta e}}^{2^{e}}\sum_{J\in I_{-e,1,k}}
\mu(R\cap J)^{\frac{1}{2}}
\frac{\mu(J'_I)}{\mu_{\epsilon}(J_I')^{\frac{1}{2}}}.
\end{align*}
We now use the inequalities $\frac{\mu(J_I')}{\mu_{\epsilon}(J_I')^\frac{1}{2}}= \big(\frac{\mu(J_I')}{\mu_{\epsilon}(J_I')}\big)^\frac{1}{2}
\mu(J_I')^\frac{1}{2}\leq \mu(J_I')^\frac{1}{2}$, and $\mu(J_I')\leq \mu(J)$, and the fact 
that the cubes $J\in I_{-e,1,k}$ and the cubes $J_I'\subseteq J_I\subset I$ with $J\in I_{-e,1,k}$ both define collections of pairwise disjoint cubes,%(? for the latter), 
to obtain 
\begin{align*}
\sum_{k=2^{\theta e}}^{2^{e}}\sum_{J\in I_{-e,1,k}}
\mu(R\cap J)^{\frac{1}{2}}
\frac{\mu(J'_I)}{\mu_{\epsilon}(J_I')^{\frac{1}{2}}}
&\leq \sum_{k=2^{\theta e}}^{2^{e}}\sum_{J\in I_{-e,1,k}}
\mu(R\cap J)^{\frac{1}{2}}
\mu(J_I')^{\frac{1}{2}}
\\
&\leq \Big(\sum_{k=2^{\theta e}}^{2^{e}}\sum_{J\in I_{-e,1,k}}
\mu(R\cap J)\Big)^{\frac{1}{2}}
\Big(\sum_{k=2^{\theta e}}^{2^{e}}\sum_{J\in I_{-e,1,k}}\mu(J_I')\Big)^{\frac{1}{2}}
\\
&\leq 
\mu(R)^{\frac{1}{2}}\mu(I)^{\frac{1}{2}}. 
\end{align*} 
With this, 
\begin{align*}
\| S_3b\|_{L^1(\mathbb R^n \setminus E, \mu)} 
&\lesssim a^{\frac{1}{2}}\sum_{P\in \mathcal P}
\sum_{I\in \child(P)}
|\langle b_P,h_{I}\rangle | \sum_{R\in \{I,\hat{I}\}}
\frac{1}{\mu(R)^{\frac{1}{2}}}\mu(R)^{\frac{1}{2}}\mu(I)^{\frac{1}{2}}
\\
&
\lesssim a^{\frac{1}{2}}\sum_{P\in \mathcal P}
\sum_{I\in \child(P)}
|\langle b_P,h_{I}\rangle |\mu(I)^{\frac{1}{2}}
\\
&
\lesssim a^{\frac{1}{2}}\sum_{P\in \mathcal P}
\sum_{I\in \child(P)}
\|f\mathbbm{1}_{P}\|_{L^1(\mu)}
\leq a^{\frac{1}{2}}
\|f\|_{L^1(\mu)}.
\end{align*}

%Finally for $S_{3}$ we get
%\begin{align*}
%\|S_{3,|e|}^ib \|_{L^1(\mathbb R^n \setminus E, \mu)} 
%&\leq  \sum_{P\in \mathcal P}
%\Big\| \Big( \sum_{I\in \mathcal D}a_{I,1}
%|\langle b_P,h_{I}\rangle |^2
%\frac{\mathbbm{1}_{I'}(x)}{\tilde \mu(I')}
%\Big)^{\frac{1}{2}}\Big\|_{L^1(\mathbb R^n \setminus E, \mu)}
%\\
%&\leq  \sum_{P\in \mathcal P}
%\Big\| \sum_{I\in \child(P)}a_{I,1}^{\frac{1}{2}}
%|\langle b_P,h_{I}\rangle |
%\frac{\mathbbm{1}_{I'}(x)}{\tilde \mu(I')^\frac{1}{2}}
%\Big\|_{L^1(\mathbb R^n \setminus E, \mu)}
%\\
%&\leq  \sum_{P\in \mathcal P}
%\sum_{I\in \child(P)}a_{I,1}^{\frac{1}{2}}
%|\langle b_P,h_{I}\rangle |
%\frac{\mu(I')}{\tilde \mu(I')^\frac{1}{2}}. 
%\end{align*}
%Now, from $\frac{\mu(I')}{\tilde \mu(I')^\frac{1}{2}}\leq \mu(I')^\frac{1}{2}$ and 
% $\mu(I')\leq \mu(I)$ we get
%\begin{align*}
%\|S_{3,|e|}^ib \|_{L^1(\mathbb R^n \setminus E, \mu)}
%&\lesssim a_1^{\frac{1}{2}}\sum_{P\in \mathcal P}
%\sum_{I\in \child(P)}
%\|f\mathbbm{1}_{P}\|_{L^1(\mu)}\mu(I)^{-\frac{1}{2}}
%\mu(I')^\frac{1}{2}
%\\
%&\lesssim a_1^{\frac{1}{2}}\sum_{P\in \mathcal P}
%\|f\mathbbm{1}_{P}\|_{L^1(\mu)}
%\leq a_1^{\frac{1}{2}}
%\|f\|_{L^1(\mu)}. 
%\end{align*}

\section{A recursion process to dominate the square forms by a sparse form}

In this last section, we ended the proof of the main result in the paper, namely, Theorem \ref{mainresult}. 
Previously, in Proposition \ref{dualbysquare}, we had shown that the dual pair can be estimated by square forms and paraproducts: 
\begin{align}\label{estsquare2}
|\langle T_{\gamma, Q}f,g\rangle|
%=|\langle Tf,g\rangle|
&\lesssim \sum_{i=1}^{k} 
B_i(f,g)
+ |\Pi_1^i(f,g)|+|\Pi_2^i(f,g)|
\\
&
\nonumber
+\sum_{\substack{I, J\in \mathcal D(Q)\\ \lambda(\widehat{I},\widehat{J}\,)\leq 1}}
|\langle f,h_{I}\rangle |
|\langle g, h_{J}\rangle |
|\langle Th_I,h_J\rangle |, 
%+\langle |f|\rangle_Q\langle |g|\rangle_Q\mu(Q)
\end{align}
%where $\beta=\alpha +\delta-\frac{n-1}{2}>1$.
%where $B_i(f,g)=\langle S_if, S_ig\rangle$.
where 
\begin{equation}\label{bilinearsquare2}
B_i(f,g)=\sum_{e\in \mathbb Z}
\hskip5pt 2^{-|e|\frac{\theta \delta }{2}}\Big( 
%2^{-|e|\delta }
%\sum_{\substack{m\in \mathbb Z\\ m\geq 2}}\frac{1}{m^{\beta}}
\langle S_{1,i}^{s_e}f,S_{1,i}^{-s_e}g\rangle 
+%2^{-|e|(\theta (\alpha +\delta)-\alpha)}
%2^{-|e|\frac{\theta \delta }{2}}
\langle S_{2,i}^{s_e}f,S_{2,i}^{-s_e}g\rangle
+%2^{-|e|\theta \delta }
\langle S_{3,i}^{s_e}f,S_{3,i}^{-s_e}g\rangle \Big) , 
\end{equation}
$S_{j,i}^k$ are the square functions of Definition \ref{square}, and $\Pi_j^i(f,g)$ are the paraproducts of Definition 
\ref{paraproducts}. 

Moreover, in Proposition \ref{estimate4paraproducts}, we estimated the paraproducts 
$\Pi_j^i(f,g)$ by sparse forms. The last term in previous estimate, which corresponds to the essentially attached cubes, is already a term of the sparse form. Then the only thing that 
remains to be done is to estimate the square form by sparse forms. 
By symmetry and summability of all series involved, this is achieved once we show a uniform estimate:
\begin{lemma} For all $i\in \{1, \ldots, k\}$, $j\in \{1,2,3\}$, $s_e\in \{1,-1\}$, we have 
$$
\langle S_{j,i}^{s_e}f,S_{j,i}^{-s_e}g\rangle \lesssim \Lambda_{\mathcal S} (f,g). 
$$
\end{lemma}
\begin{proof} To obtain the stated estimate, we use recursive process that constructs the sparse form $\Lambda (f,g)$. The argument is standard, so we only describe the main steps of the recursion step. 

For fixed $i,j, s_e$, we simplify the notation by writing $S_{j,i}^{s_e}f=Sf$ and $S_{j,i}^{-s_e}f=S'f$. 
%Let $S\!B(f,g)=\langle Sf,S'g\rangle $. 
We localize the square forms as follows: for each selected cube $R\in \mathcal S$, we define square function by using the expressions of Definition \ref{square}, but changing the grid $\mathcal D_i$ by $\mathcal D_i(R)$. For example, the first square function is 
\begin{align*}
S_Rf=S_{1,i, R}^1f
&=\Big( \sum_{\substack{m\in \mathbb Z\\ m\geq 2}}\frac{1}{m^{\alpha +\delta}}\sum_{I\in \mathcal D_i(R)}
a_{I}
|\langle f,h_{I}\rangle |^2
\frac{1}{\ell(I)^{\alpha}}
\sum_{J \in I_{-e,m}}\frac{\mu(J)}{\mu_{\epsilon }(J_I')}\mathbbm{1}_{J_I'} 
\Big)^{\frac{1}{2}}. 
\end{align*}
We note that $S_{1,i, R}^1f=S_{1,i, R}^1(f\mathbbm{1}_{R})$. 
%$$
%%S_{Q}(f,g)=\sum_{i=1,3,4}\int_{Q} S_{i}(f)(x)S_{2}g(x)d\mu(x)
%S_{R}(f)=S(f\mathbbm{1}_{R})\mathbbm{1}_{R}, 
%$$
Similarly, we define 
$S_{R}'$, 
and $B_{R}(f,g)=\langle S_{R}f,S_{R}'g\rangle $. 
We also define the local exceptional set 
\begin{align}\label{excep}
E_{R}&=\{ x\in S / S_{R}f(x)>C a_R^{\frac{1}{2}}
\langle |f\mathbbm{1}_{R}|\rangle_S\}
\cup \{ x\in R / S_{R}'g(x)>C a_R^{\frac{1}{2}}\langle |g\mathbbm{1}_{R}|\rangle_{R}\}.
\end{align}
where $a =\sup_{\substack{m\in \mathbb N\\e\in \mathbb Z}}\sup_{I\in \mathcal D}(1+\rho(2^eI))^2(1+\rho(I))^2\sup_{J\in I_{e,m}}F(I,J)$. 

The %$S_Qf\leq S(f\mathbbm{1}_{Q})$, $S_{0,Q}f\leq S_0(f\mathbbm{1}_{Q})$ and 
same way we proved that $S, S'$ are bounded, we can show that  
$\| S_R\|_{L^{1}(\mu)\rightarrow L^{1,\infty }( \mu)}+\| S_R'\|_{L^{1}(\mu)\rightarrow L^{1,\infty }( \mu)}\lesssim a_R^{\frac{1}{2}}$, we have 
\begin{align}\label{eqsparse}
\nonumber
\mu(E_{R})&
\leq \frac{\| S_Rf\|_{L^1(\mu)}}{C  a_R^{\frac{1}{2}}\langle |f\mathbbm{1}_{R}|\rangle_R}
+\frac{\| S_R'f\|_{L^1(\mu)}}{Ca_R^{\frac{1}{2}} \langle |g\mathbbm{1}_{R}|\rangle_{R}}
\\
&
\leq 
\frac{\| f\mathbbm{1}_{R}\|_{L^1(\mu)}}{C  \langle |f\mathbbm{1}_{R}|\rangle_R}
+\frac{\| g\mathbbm{1}_{R}\|_{L^1(\mu)}}{C \langle |g\mathbbm{1}_{R}|\rangle_{R}}
\leq 2\frac{\mu(Q)}{C}
<\frac{\mu(R)}{2}, 
\end{align}
by the choice of $C$. Now
\begin{align*}
B_{R}(f,g)&=
%\sum_{i=1,2,3}
%\int_{Q} S_{Q}f(x)S_{0,Q}g(x)d\mu(x) 
%\\
%&
%=\sum_{i=1,2,3}
\int_{R\setminus E_R} S_{R}f(x)S_{R}'g(x)d\mu(x) 
+% \sum_{i=1,2,3}
\int_{E_R} S_{R}f(x)S_{R}'g(x)d\mu(x).
% \\
% &
% =2^{e/2}m^{-(d+\delta)}\varepsilon_{Q} R(Q)
% = 2^{e/2}m^{-(d+\delta)}
% \varepsilon_{Q} (R(Q\setminus E)+ R(E))
\end{align*}

For the first term, we use \eqref{excep} to write 
\begin{align*}
\int_{R\setminus E} S_{R}f(x)S_{R}'g(x)d\mu(x) 
&\leq C^2a_{R}\langle |f|\rangle_R \langle |g|\rangle_{R}\mu(R)\lesssim \Lambda_{\mathcal S} (f,g),  
\end{align*}
since $R\in \mathcal{S}$.

To deal with the second term, we 
define 
$\mathcal E_{R}$ as the collection of cubes $Q'\in \mathcal S(R)$ such that 
that are maximal with respect to inclusion. 
Then, since $\mathcal E_{Q}$ is a partition of $E_{S}$, we have
%Moreover, it is a sparse collection since by \eqref{eqsparse} we have  
%\begin{equation}\label{Psparse}
%\sum_{Q'\in \mathcal E_{Q}}\mu(Q')\leq \mu(E_{Q})< \frac{1}{2}\mu(Q).
%\end{equation}
%We simply note that $\mathcal S(R)$ is a sparse collection 
%
%Then
\begin{align*}
B_R(f,g)
&\lesssim  a_{R}\langle |f|\rangle_R \langle |g|\rangle_{R}
\mu(R)+\sum_{R'\in \mathcal{S}(R)}B_{R'}(f,g). 
\end{align*}

Now we can repeat the same argument on $B_{R'}(f,g)$ and continue the recursion process. 
% for each cube $'$ in the sparse collection $\mathcal{E}_Q$.

%Then the sparse form is constructed by using previous iterative process. For this, we first fix a 
%cube $Q_0\in \mathcal C$ such that $\sup{f}\cup \sup{g}\subseteq 2^{-1}Q_0$.

\end{proof}

\end{document}